\documentclass[final] {dmtcs-episciences}

\usepackage[utf8]{inputenc}
\usepackage{subfigure}

\usepackage{,amsmath,amsthm,amsfonts,amssymb,graphicx,float,microtype,underscore,mathtools,xurl,multirow, lineno} 

\usepackage[shortlabels]{enumitem}

\theoremstyle{plain}
\newtheorem{thm}{Theorem}
\newtheorem{lem}[thm]{Lemma}
\newtheorem{cor}[thm]{Corollary}
\newtheorem{prop}[thm]{Proposition}

\newtheorem*{conj*}{Conjecture}

\theoremstyle{definition}
\newtheorem{remark}{Remark}

\usepackage{pgf,tikz,ifthen,arrayjobx}

\usetikzlibrary{arrows}
\usetikzlibrary{calc}

\newcommand{\oG}{\overline{G}}
\newcommand{\ov}{\overline{v}}

\usepackage[round, authoryear]{natbib}
\usepackage[round, authoryear]{natbib}
\newcommand{\doi}[1]{doi: \href{[https://doi.org/#1](https://doi.org/#1)}{\nolinkurl{#1}}}
\author[Dickson Y. B. Annor et al.]{Dickson Y. B. Annor \affiliationmark{1}
  \and Yuri Nikolayevsky\affiliationmark{2}\
  \and Michael S. Payne\affiliationmark{1}}

\title[Three Theorems on Negami's Planar Cover Conjecture]{Three Theorems on Negami's Planar Cover Conjecture}
  
\affiliation{
  Department of Mathematical and Physical Sciences, La Trobe University, Bendigo, Victoria, Australia\\
  Department of Mathematical and Physical Sciences, La Trobe University, Melbourne, Victoria, Australia
 }

\begin{document}

\publicationdata{vol. 28:3}{2026}{14}{10.46298/dmtcs.15167}{2025-01-30; 2025-01-30; 2026-07-28}{2026-09-15}

\maketitle

\begin{abstract}
\vspace{0.2cm}
A long-standing conjecture of S. Negami states that a connected graph has a finite planar cover if and only if it embeds in the projective plane. It is known that the conjecture is equivalent to the claim that \emph{the graph $K_{1,2, 2, 2}$ has no finite planar cover}. We prove three theorems showing that $K_{1,2, 2, 2}$ admits no planar cover with certain structural properties, and that a minimal planar cover of $K_{1,2, 2, 2}$ (if it exists) must be $4$-connected.  
\end{abstract}

\section{Introduction}

All graphs in this paper are finite, simple and undirected. For a graph $K$, let $V(K)$ and $E(K)$ denote the vertex set and the edge set of $K$, respectively. A graph $G$ is called a \emph{cover} of a graph $K$ if there exists an onto mapping $\pi : V(G) \to V(K)$, called a (\emph{covering}) \emph{projection}, such that $\pi$ maps the neighbours of any vertex $v$ in $G$ bijectively onto the neighbours of $\pi(v)$ in $K$. If $K$ is connected, then the number $|\pi^{-1}(v)| = n$ is the same for all $v \in V(K)$, and then $\pi$ is called an \emph{$n$-fold cover}.

A cover is called \emph{planar} if it is a planar graph. A planar graph trivially has a planar cover by the identity projection; a nonplanar graph which embeds in the projective plane has a $2$-fold planar cover.

In $1988$, \cite{negami1988} made the following conjecture:

\begin{conj*}\label{conj:negami}
A connected graph $K$ has a finite planar cover if and only if $K$ embeds in the projective plane.
\end{conj*}

The results of  \cite{archdeacon2002two}, \cite{fellows1985encoding}, \cite {hlinvey1998k4, hlinueny2001}, and \cite{negami1988graphs} combined to show that Negami's Conjecture is equivalent to the following statement: \emph{the graph $K_{1, 2, 2, 2}$ has no finite planar cover}. The graph $K_{1, 2, 2, 2}$ consists of the octahedron with an apex vertex connected to all other vertices, see Figure~\ref{fig:K2221} (the \emph{apex} vertex is the only one of degree $6$ and is labelled $0$ in the figure). ~\cite{archdeacon1990parity} showed that for any planar cover of a nonplanar graph, the fold number is necessarily even. For further results on Negami's Conjecture, we refer the reader to \cite{hlinveny201020}. ~\cite{brianski2024high} have proved several results on the extension of Negami’s Conjecture to graphs on surfaces of higher genus.  Recently, the present authors have shown the following.

\begin{thm}[\cite{annor25k1222}]\label{tha:no12fold}
For $n < 14$, no $n$-fold cover of the graph $K_{1, 2, 2, 2}$ is planar.
\end{thm}

In this paper, we prove nonplanarity of certain covers of $K_{1, 2, 2, 2}$ and establish several structural properties of such planar covers (if they exist), of arbitrary fold.

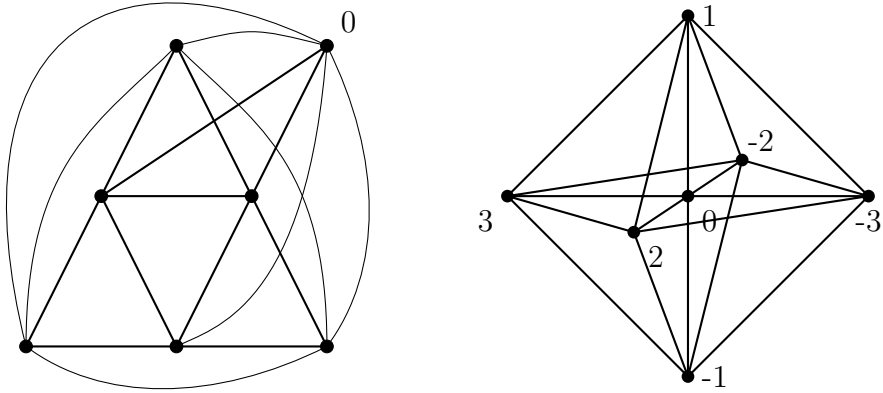
\begin{figure}[h]
    \centering
    \begin{tikzpicture}
    \begin{scope}[scale=0.45]
 \draw[ thick ] (0,0)--(8,0);
 \draw[ thick ] (0,0)--(4,8);
 \draw[ thick ] (8,0)--(4,8);
 \draw[ thick ] (2,4)--(6,4);
 \draw[ thick ] (4,0)--(6,4);
 \draw[ thick ] (4,0)--(2,4);
 \draw[ thick ] (6,4)--(8,8);
 \draw[ thick ] (2,4)--(8,8);
 \draw[thick] (0,0) .. controls (-2,8) and (2, 11) .. (8, 8);
 \draw[thick] (0,0) .. controls (0,5) and (2, 6) .. (4, 8);
 \draw[thick] (8,0) .. controls (8,5) and (6, 6) .. (4, 8);
 \draw[thick] (0,0) .. controls (2,-1.5) and (5, -1.5) .. (8, 0);
 \draw[thick] (4,0) .. controls (5, 0.5) and (7.5, 1) .. (8, 8);
 \draw[thick] (8,0) .. controls (9.5,2) and (9.5, 5)  .. (8, 8);
 \draw[thick] (4,8) .. controls (6,8.5) .. (8, 8);
 \filldraw[black] (0,0) circle (4.8pt);
 \filldraw[black] (4,0) circle (4.8pt);
 \filldraw[black] (8,0) circle (4.8pt);
 \filldraw[black] (2,4) circle (4.8pt);
 \filldraw[black] (8,8) circle (4.8pt); \draw (8,8) node[above right=.1em] {$0$};
 \filldraw[black] (4,8) circle (4.8pt);
 \filldraw[black] (6,4) circle (4.8pt);
\end{scope}
\hfill 1.5cm
\begin{scope}[scale=0.8, shift={(11,2.5)},z={(-.3,-.2)}, 
                    line join=round, line cap=round]
  \draw[thick] (3,0,0) -- (0,0,3) -- (0,3,0) -- (3,0,0) -- (-3,0,0) -- (0,-3,0) -- (0,3,0) -- (-3,0,0) -- (0,0,3) -- (0,-3,0) -- (3,0,0);
  \draw[thick] (0,3,0) -- (0,0,-3) -- (3,0,0) (0,-3,0) -- (0,0,-3) -- (-3,0,0) (0,0,3) -- (0,0,-3);
  \fill (3,0,0) circle (3pt); \fill (-3,0,0) circle (3pt);\fill (0,0,3) circle (3pt);
  \fill (0,3,0) circle (3pt); \fill (0,-3,0) circle (3pt);\fill (0,0,-3) circle (3pt);
  \draw (0,0,3) node[below right=.11em] {$2$}; \draw (0,0,-3) node[shift=({0.25,0.25})] {-$2$};
  \draw (3,0,0) node[below=.1em] {-$3$}; \draw (-3,0,0) node[below left=.1em] {$3$};
  \draw (0,3,0) node[right=.1em] {$1$}; \draw (0,-3,0) node[right=.1em] {-$1$};
  \fill (0,0,0) circle (3pt); \draw (0,0,0) node[below right=.1em] {$0$};
\end{scope}
\end{tikzpicture}
    \caption{Two drawings of the graph $K_{1, 2, 2, 2}$.}
    \label{fig:K2221}
\end{figure}

In Section~\ref{sec:vertcon} we study the connectivity of cover graphs, showing the following:

\begin{thm} \label{th:4conn}
If $K$ is a $4$-connected graph, then a minimal planar cover of $K$ is $4$-connected.
\end{thm}

An application of Theorem~\ref{th:4conn} is the following corollary.

\begin{cor} 
A minimal planar cover of $K_{1,2,2,2}$ is $4$-connected.
\end{cor}

Recall that for a vertex $v$ in a graph, its (open) neighbourhood $N(v)$ is the set of vertices adjacent to $v$, and the closed neighbourhood $N[v]$ is the union of $N(v)$ and the vertex $v$ itself. A \emph{wheel graph} (denoted as \(W_{m}\) for $m$ vertices) is a graph formed by connecting a single central vertex to all vertices of a cycle graph $C_{m-1}$. 
The second main result of the paper is proved in Section~\ref{s:perf}:

\begin{thm} \label{th:noperf}
  The graph $K_{1,2,2,2}$ does not admit a planar cover in which the closed neighbourhood of each vertex covering the apex vertex induces the wheel graph $W_7$.
\end{thm}

Further in Section~\ref{s:all9thm} we establish several structural results on a planar cover of $K_{1,2,2,2}$. They require a construction of a specific \emph{semi-cover} $G'$ of $K_{1,2,2,2}$ which was central in the proof of Theorem~\ref{tha:no12fold}. Non-existence of such a semi-cover implies Negami's Conjecture, see~\cite[Conjecture~2]{annor25k1222}. In Theorem~\ref{th:notall9} we prove that there is no such planar semi-cover $G'$ satisfying a certain condition on the length of the faces of the planar cover of the graph $K_4$ within it (see Section~\ref{s:all9thm} for unexplained terminology and the precise statement).

\section{Preliminaries}
\label{sec:prel}

In this section we introduce some notions and terminology to be used in subsequent sections.

Let $G$ be a cover of graph $K$. For any subgraph $K'$ of $K$, we call the graph $H' = \pi^{-1}(K')$ the \emph{lift of} $K'$ \emph{into G}. The lift of a cycle $C$ of $K$ into $G$ is a disjoint union of cycles, the length of each of which is a multiple of the length of $C$. A cycle $C'$ in the lift of a cycle $C$ to $G$ is called \emph{short} if its length is equal to the length of $C$. Otherwise, it is called a \emph{long} cycle. A cycle $C$ is called \emph{peripheral} in a graph $K$ if it is chordless and $K\setminus C$ is connected. We will several times use the following result.

\begin{lem}[\cite{archdeacon2002two}]\label{lem:shortcycle}
Let $G$ be a plane cover of a graph $K$. Let $C'$ be a short cycle of $G$ covering a peripheral cycle $C$ in $K$. Then $C'$ is a face boundary.
\end{lem}

Throughout the paper we will identify a planar cover with its drawing when this creates no ambiguity.

The following definition was introduced in~\cite[Section~2]{hlinveny201020}. Let $G'$ be a plane graph, and let $F_e$ be the outer face of $G'$. The graph $G'$ is called a \emph{semi-cover} of a graph $K$ if there exits an onto mapping $\pi' : V(G') \to V(K)$, called a \emph{semi-projection}, such that for each vertex $v$ of $G'$ not incident with $F_e$, $\pi'$ maps the neighbours of $v$ bijectively onto the neighbours of $\pi'(v)$ in $K$, and for each vertex $w$ of $G'$ incident with $F_e$, $\pi'$ maps the neighbours of $w$ injectively to the neighbours of $\pi'(w)$. It is clear that every cover is a semi-cover, but the converse is not true.

We adopt the following notation convention in Sections~\ref{s:perf} and~\ref{s:all9thm} (but not in Section~\ref{sec:vertcon}, as Theorem~\ref{th:4conn} is not specific to the covers of $K_{1,2,2,2}$). We call the vertex of degree $6$ of $K_{1,2,2,2}$ the \emph{apex} vertex, and we call the objects (vertices, edges, cycles and subgraphs) lying in the subgraph $K_{2,2,2} \subset K_{1,2,2,2}$ \emph{octahedral}. We label the vertices of $K_{1,2,2,2}$ as follows: the vertex of degree $6$ is labelled $0$, and the six octahedral vertices (of degree $5$) are labelled $\{\pm 1, \pm 2, \pm 3\}$ in such a way that the vertices $i$ and $-i$ are not adjacent for $i=1,2,3$. Note that every vertex $\pm i$ of degree $5$ is adjacent to the vertex $0$ and to four vertices $\pm j, \pm k$, where $\{i,j,k\}=\{1,2,3\}$. Given a cover $G$ of $K_{1,2,2,2}$, we label every vertex $v$ of $G$ with the label of the corresponding vertex $\pi(v)$ of $K_{1,2,2,2}$. We will occasionally add a dash or a subscript to distinguish between two vertices of $G$ projected to the same vertex of $K_{1,2,2,2}$. By relabelling we mean permuting the labels $1,2,3$, and accordingly the labels $-1,-2,-3$ (and also sometimes swapping the labels within one or more of the pairs $(i, -i), \, i =1,2,3$); the precise meaning will be always clear from the context.

\section{Vertex-connectivity of Planar Covers. Proof of Theorem~\ref{th:4conn}}
\label{sec:vertcon}

Let $K$ be a connected graph, and let $G$ be a minimal planar cover of $K$, with the covering map $\pi: G \to K$. The graph $G$ is connected (otherwise, $G$ has a connected component with a fold number less than that of $G$). As Theorem~\ref{th:4conn} is not specific to the case $K=K_{1,2,2,2}$, our notation convention throughout this section will be different to the one introduced in Section~\ref{sec:prel}. We label the vertices of $G$ by lower case letters, and vertices of $K$ by capital letters, and we say that a vertex $v \in V(G)$ is labelled $A \in V(K)$ if $\pi(v) = A$.

Suppose that $X \subset V(G)$ is a vertex cut set of $G$. Let $G_1$ and $G_2$ be the partition of the set $V(G) \setminus X$ such that there are no edges from $G_1$ to $G_2$. We call $G_1$ and $G_2$ the \emph{sides} of $V(G) \setminus X$. For $x$ in $X$ and $i = 1, 2$, we define $N_i(x)$ to be the set of neighbours of $x$ in $G_i$, and $L_i(x)$ to be the set of labels of the vertices in $N_i(x)$.

We need the following observations.

\begin{lem}\label{lem2a}
Suppose $K$ is $k$-connected and $G$ is a cover of $K$. Suppose $X$ is a vertex cut in $G$ containing at most $k-1$ vertices. If just one vertex $x \in X$ is labelled $A$, then all neighbours of $x$ with labels that do not appear in the cut are on the same side.
\end{lem}

\begin{proof}
Let $x$ be the only vertex in $X$ labelled $A$. Let $L$ be the set of labels of the other vertices in $X$. Note that $|L|\leq k- 2$, and moreover, the degree of $x$ is at least $k$. Let $u, v$ be two neighbours of $x$ with the labels $U$ and $V$, respectively, such that $U, V \notin L$. Suppose that $u$ and $v$ lie on the different sides. The path $(u,x,v)$ in $G$ projects onto the path $(U,A,V)$ in $K$. Since $K$ is $k$-connected, the graph $K \setminus(\{A\} \cup L)$ is connected, and so there is a path from $U$ to $V$ in $K \setminus(\{A\} \cup L)$. This implies that there is a cycle $C$ in $K \setminus L$ that contains the path $(U,A,V)$. Lifting $C$ into $G$ we get a set of disjoint cycles one of which contains the path $(u,x,v)$. But this cycle must pass through a vertex in the cut with label in $L$ to get back, a contradiction.
\end{proof}

Lemma~\ref{lem2a} immediately implies that any minimal planar cover of a $2$-connected graph is $2$-connected. The following fact is proved in \cite[Corollary~10]{negami2024another}, but we include its proof for completeness.

\begin{lem}\label{lem3}
If a graph $K$ is $3$-connected, then any minimal planar cover $G$ of $K$ is $3$-connected.
\end{lem}
\begin{proof}
Suppose that $X = \{x, y\}$, with $\pi(x)=A$ and $\pi(y) = B$, is a vertex cut in $G$. First suppose that $A \ne B$. By Lemma~\ref{lem2a}, all the neighbours of $x$ with the labels different from $B$ must appear on the same side, say in $G_2$. If $x$ has no neighbour in $G_1$, we contradict the minimality assumption on $X$. Hence $x$ has exactly one neighbour $z$ in $G_1$. But then $\pi(z) =B$, and so $\{z,y\}$ is another minimal cut, with both vertices having the same label.

We can therefore assume that $A=B$. We claim that $L_1(x) = L_2(y)$ and $L_2(x) = L_1(y)$. Indeed, choose $p \in N_1(x)$ and $q \in N_2(x)$ (both $N_1(x)$ and $N_2(x)$ are nonempty by minimality of $X$), and denote $\pi(p)=P$ and $\pi(q) = Q$. The path $(p,x,q)$ in $G$ projects onto the path $(P,A,Q)$ in $K$. Since $K \setminus \{A\}$ is connected, there is a cycle containing the path $(P,A,Q)$ in $K$, and hence there is a cycle $C'$ in $G$ that contains the path $(p,x,q)$. Since $C'$ passes through $x \in X$, it must go through $y \in X$ to get back. Hence, there is a path $(q',y,p')$ contained in $C'$, where $q' \in G_1, \, p' \in G_2$ and $\pi(p')=P, \, \pi(q') = Q$, as claimed.

But now we can delete the vertex cut $X$ and join the neighbours of $x$ and $y$ in say $G_1$, to a single vertex labelled $A$, obtaining a smaller planar cover, a contradiction.
\end{proof}

It remains to show that if $K$ is $4$-connected, then $G$ is also $4$-connected.

\begin{proof}[of Theorem~\ref{th:4conn}]
Let $X$ be a minimal vertex cut in $G$. By Lemma~\ref{lem3} we have $|X| \ge 3$. Suppose $X = \{x, y, z\}$. We consider the following cases.
\begin{enumerate}[label=(\alph*),ref=\alph*]
    \item\label{c1} All three vertices in $X$ have the same label.

    \item\label{c2} Exactly two vertices in $X$ have the same label.

    \item\label{c3} All three vertices in $X$ have different labels.
\end{enumerate}

For case~\eqref{c1}, let $x, y$ and $z$ be labelled $A$, and let $P$ and $Q$ be adjacent to $A$ in $K$. We claim that the set $N_1(x) \cup N_1(y) \cup N_1(z)$ contains the same number of vertices labelled $P$ and vertices labelled $Q$. Indeed, the path $(P,A,Q)$ is contained in a cycle of $K$ whose lift to $G$ is a disjoint union of cycles containing the paths $(p,x,q), \, (p',y,q')$ and $(p'',z,q'')$, where $p,p',p'' \in \pi^{-1}(P)$ and $q,q',q'' \in \pi^{-1}(Q)$. But if one of these cycles crosses the set $X$ from $G_1$ to $G_2$, it must cross back exactly once, and the claim follows. Now let $n_i, \, i=1,2$, be the number of vertices labelled $P$ in $N_i(x) \cup N_i(y) \cup N_i(z)$. Then we have $n_1 + n_2 = 3$ and it follows from the proof of the above claim that $n_1, n_2 > 0$. 
Up to relabelling we can assume that $n_1 = 1$, so that for any $P$ adjacent to $A$ in $K$, there is exactly one vertex $p \in N_1(x) \cup N_1(y) \cup N_1(z)$ labelled $P$. But now we can delete the cut $X$ (and the set $G_2$) and join all these vertices to a single vertex labelled $A$ hence obtaining a smaller planar cover, a contradiction.

For case~\eqref{c2}, let $x, y$ and $z$ be labelled $A, B$ and $B$, respectively. From Lemma~\ref{lem2a}, all the neighbours of $x$ with labels different from $B$ are on the same side, say in $G_2$. As $N_1(x)$ is nonempty, it must contain a single vertex $w$, and moreover, $\pi(w) = B$. But now we can take the set $\{w,y,z\}$ as a minimal vertex cut reducing this case to case~\eqref{c1}.

In case~\eqref{c3}, assume that $x, y$ and $z$ are labelled $A, B$ and $C$, respectively. By Lemma~\ref{lem2a}, all the neighbours of $x$ with
labels different from $B$ and $C$ are on the same side, say in $G_2$. The set $N_1(x)$ is nonempty and contains no more than two vertices whose labels belong to $\{B, C\}$. If $N_1(x)$ consists of a single vertex $w$, then the set $\{w,y,z\}$ is a minimal vertex cut; this takes us to case~\eqref{c2}. Thus, $N_1(x)$ consists of two vertices which are labelled $B$ and $C$.

Let $P\in L_2(x)$ and consider the lift of a cycle containing the path $(B,A,P)$. The cycle in this lift that contains $x$ must pass back over $X$ via $y$, which implies that the vertex adjacent to $y$ labelled $A$ is in $G_1$. By similar arguments, we see that $y$ is adjacent to two vertices in $G_1$ (labelled $A$ and $C$) and $z$ is adjacent to two vertices in $G_1$ (labelled $A$ and $B$). All other vertices adjacent to  $X$ are in $G_2$.

To complete the proof, we delete $G_1$ and join the vertices in $X$ with a $3$-cycle $(x,y,z)$ to get a smaller planar cover, a contradiction.
\end{proof}

\section{Proof of Theorem~\ref{th:noperf}}
\label{s:perf}

Let $G$ be a cover of $K_{1,2,2,2}$ with $\pi: G \to K_{1,2,2,2}$ the covering projection. Throughout this section we call vertices of degree $d, \; d=5,6$, of $G$ and of $K_{1,2,2,2}$ $d$-vertices. We use the labelling of the vertices of both $K_{1, 2, 2, 2}$ and $G$ introduced in Section~\ref{sec:prel}, so that a $6$-vertex is labelled $0$, and the $5$-vertices are labelled $\pm 1, \pm 2$ and $\pm 3$ in such a way that no two vertices $i$ and $-i$ are adjacent for $i=1,2,3$.

Suppose $G$ is a planar cover of $K_{1, 2, 2, 2}$; we identify $G$ with its planar drawing. Consider a $6$-vertex $0$ in $G$ and its open and closed neighbourhoods $N(0)$ and $N[0]$ respectively ($N(0)$ consists of the six $5$-vertices connected to the vertex $0$, and $N[0] = N(0) \cup \{0\}$). The induced subgraphs $B=B(0) \subseteq N(0)$ and $S=S(0) \subseteq N[0]$
will be called the \emph{ball} and the \emph{sphere} centred at the $6$-vertex $0$, respectively.

\begin{lem} \label{l:nobade}
  Let $0$ be a $6$-vertex of $G$. Then any edge $(u,v)$ of $S(0)$ lies on the boundary of a triangular face $(0,u,v)$ of $G$; in particular, the edges $(0,u)$ and $(0,v)$ must be consecutive in the rotation diagram at $0$. The sphere $S(0)$ is either the disjoint union of paths and isolated vertices or a $6$-cycle.
\end{lem}

\begin{proof}
  The proof is a direct application of Lemma~\ref{lem:shortcycle}: any $3$-cycle of $G$ must bound a face, as any such cycle is short and is peripheral.
\end{proof}

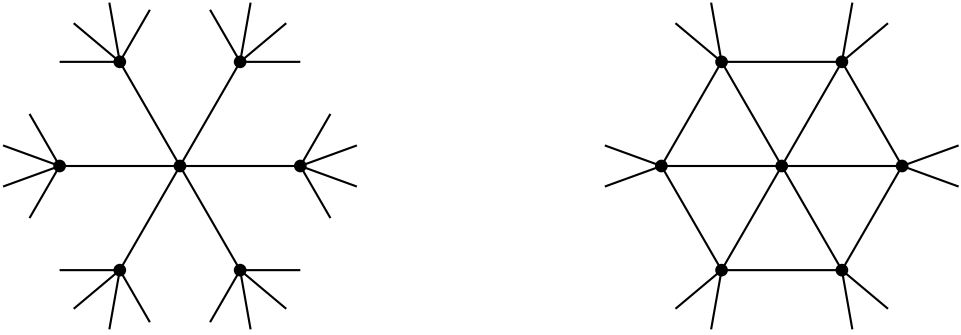
\begin{figure}[h]
\centering
\begin{tikzpicture}[scale=0.75]
\begin{scope}
\foreach \x in {0,1,...,5} {
\fill  ({2*cos((2*pi*\x/6) r)},{2*sin((2*pi*\x/6) r)}) circle (3pt);
\draw[thick] ({2*cos((2*pi*\x/6) r)},{2*sin((2*pi*\x/6) r)}) -- (0,0);
\foreach \y in {0,1,...,3} {
\draw[thick] ({2*cos((2*pi*\x/6) r)},{2*sin((2*pi*\x/6) r)}) -- ({2*cos((2*pi*\x/6) r)+cos((2*pi*\x/6-pi/3+2*pi/9*\y) r)},{2*sin((2*pi*\x/6) r)+sin((2*pi*\x/6-pi/3+2*pi/9*\y) r)});
}
}
\fill (0,0) circle (3pt);
\end{scope}
\begin{scope}[shift={(10,0)}]
\foreach \x in {0,1,...,5} {
\fill  ({2*cos((2*pi*\x/6) r)},{2*sin((2*pi*\x/6) r)}) circle (3pt);
\draw[thick] ({2*cos((2*pi*(\x+1)/6) r)},{2*sin((2*pi*(\x+1)/6) r)}) -- ({2*cos((2*pi*\x/6) r)},{2*sin((2*pi*\x/6) r)}) -- (0,0);
\foreach \y in {1,2} {
\draw[thick] ({2*cos((2*pi*\x/6) r)},{2*sin((2*pi*\x/6) r)}) -- ({2*cos((2*pi*\x/6) r)+cos((2*pi*\x/6-pi/3+2*pi/9*\y) r)},{2*sin((2*pi*\x/6) r)+sin((2*pi*\x/6-pi/3+2*pi/9*\y) r)});
}
}
\fill (0,0) circle (3pt);
\end{scope}
\end{tikzpicture}
\caption{A ball and a wheel.}
\label{fig:snowflake2}
\end{figure}

We call an edge joining two $5$-vertices of different balls \emph{external}. We call the number $d(B)$ of external edges attached to the $5$-vertices of a ball $B$ its \emph{external degree}. Note that $d(B)$ is equal to $24$ minus twice the number of the edges of the sphere $S \subset B$, and so $d(B)$ is even and $12 \le d(B) \le 24$.

A $2$-path of the octahedral subgraph $K_{2,2,2} \subset K_{1,2,2,2}$ is called \emph{straight} if its endpoints are not adjacent.

\begin{remark}\label{rem:stli}
If a $2$-path $p$ lying in a sphere $S \subset G$ projects to a straight $2$-path of $K_{2,2,2}$, then the endpoints of the two external edges of $G$ incident to the middle vertex of $p$ are not adjacent (as they are not adjacent in  $K_{2,2,2}$).
\end{remark}

We call a ball $B=B(0)$ a \emph{wheel} if it is isomorphic to the wheel graph $W_7$ (having seven vertices). Equivalently, a ball $B(0)$ is a wheel if its sphere $S(0)$ is a cycle (equivalently, if $d(B)=12$), see Figure~\ref{fig:snowflake2}. If the ball $B(0)$ is a wheel, we use the notation $W(0)$ instead of $B(0)$. It is easy to see that the sphere of a wheel projects, up to symmetry and relabelling, to one of the two Hamiltonian $6$-cycles of $K_{2,2,2}$ shown in Figure~\ref{fig:zs}; we call them a \emph{sphericon} and a \emph{zigzag}, respectively.

\begin{figure}[h]
\centering
\begin{tikzpicture}[scale=0.8]
\begin{scope}[z={(-.3,-.2)}, 
                    line join=round, line cap=round
                   ]
  \draw[thick] (0,3,0) -- (-3,0,0) -- (0,-3,0) -- (3,0,0) -- (0,3,0) -- (0,0,3) -- (0,-3,0);
  \draw[line width=2.5pt] (0,-3,0) -- (3,0,0) -- (0,0,3) -- (-3,0,0) -- (0,3,0);
  \draw[line width=2.5pt, dashed] (0,3,0) -- (0,0,-3) -- (0,-3,0);
  \draw[thick, dashed]  (3,0,0) -- (0,0,-3) -- (-3,0,0);
  \fill (3,0,0) circle (3pt); \fill (-3,0,0) circle (3pt);\fill (0,0,3) circle (3pt);
  \fill (0,3,0) circle (3pt); \fill (0,-3,0) circle (3pt);\fill (0,0,-3) circle (3pt);
  \draw (0,0,3) node[below right=.1em] {$2$}; \draw (0,0,-3) node[shift=({-0.5,0.25})] {-$2$};
  \draw (3,0,0) node[below=.1em] {-$3$}; \draw (-3,0,0) node[below left=.1em] {$3$};
  \draw (0,3,0) node[right=.1em] {$1$}; \draw (0,-3,0) node[right=.1em] {-$1$};
\end{scope}
\begin{scope}[shift={(10,0)},z={(-.3,-.2)}, 
                    line join=round, line cap=round
                   ]
  \draw[thick] (-3,0,0) -- (0,-3,0) (3,0,0) -- (0,0,3) -- (0,3,0) -- (3,0,0);
  \draw[line width=2.5pt] (0,3,0) -- (-3,0,0) -- (0,0,3) -- (0,-3,0) -- (3,0,0);
  \draw[line width=2.5pt,dashed] (0,3,0) -- (0,0,-3) -- (3,0,0);
  \draw[thick, dashed] (0,-3,0) -- (0,0,-3) -- (-3,0,0);
  \fill (3,0,0) circle (3pt); \fill (-3,0,0) circle (3pt);\fill (0,0,3) circle (3pt);
  \fill (0,3,0) circle (3pt); \fill (0,-3,0) circle (3pt);\fill (0,0,-3) circle (3pt);
  \draw (0,0,3) node[below right=.1em] {$2$}; \draw (0,0,-3) node[shift=({-0.5,0.25})] {-$2$};
  \draw (3,0,0) node[below=.1em] {-$3$}; \draw (-3,0,0) node[below left=.1em] {$3$};
  \draw (0,3,0) node[right=.1em] {$1$}; \draw (0,-3,0) node[right=.1em] {-$1$};
\end{scope}
\end{tikzpicture}
\caption{A sphericon (on the left) and a zigzag (on the right) in  $K_{2,2,2}$.}
\label{fig:zs}
\end{figure}
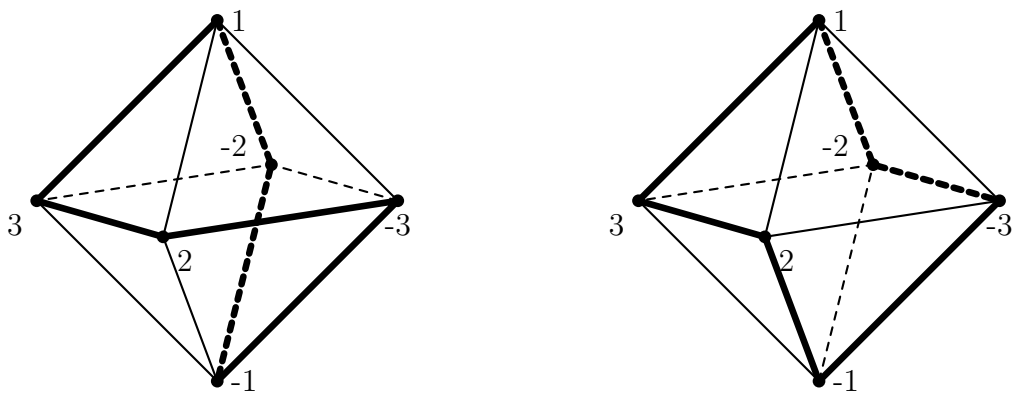

\begin{remark} \label{rem:svsz}
  Note that a zigzag contains no straight $2$-paths, while a sphericon contains exactly two. Moreover, a sphericon and a zigzag may be distinguished from one another by the labelling: on a zigzag, the distance between any pair of vertices $\pm i, \, i=1,2,3$, is $3$, while on a sphericon, there is only one pair of vertices $\pm i$ at distance $3$, and moreover, these vertices are the midpoints of the straight $2$-paths ($i = 2$ in the labelling as on the left in Figure~\ref{fig:zs}).
\end{remark}

\begin{proof}[of Theorem~\ref{th:noperf}]
 Our proof follows the scheme of the proof in~\cite{hlinvey1998k4}.

 Suppose $G$ is a planar cover of $K_{1,2,2,2}$ (which we identify with its plane drawing) such that all the balls are wheels. Every wheel has $12$ external edges attached to its $5$-vertices. We call a sequence $e_1, \dots, e_m,\; m \ge 1$, of external edges joining $5$-vertices of a wheel $W$ with $5$-vertices of a wheel $W'$ \emph{parallel} if for $i=1, \dots, m-1$, the edges $e_i$ and $e_{i+1}$ are consecutive in the positive (respectively, in the negative) direction along the boundary of $W$ (respectively, of $W'$). It is not hard to see that up to symmetries, all parallel sequences with $m=2, 3$ are as those given in Figure~\ref{fig:parseq}.

\begin{figure}[h]
\centering
\begin{tikzpicture}[scale=0.65]
\begin{scope} 
\foreach \z in {0,1} {
\foreach \x in {0,1,...,5} {
\fill  ({2*cos((2*pi*\x/6) r)},{6*\z+2*sin((2*pi*\x/6) r)}) circle (3pt);
\draw[thick] ({2*cos((2*pi*(\x+1)/6) r)},{6*\z+2*sin((2*pi*(\x+1)/6) r)}) -- ({2*cos((2*pi*\x/6) r)},{6*\z+2*sin((2*pi*\x/6) r)}) -- (0,6*\z+0);
\foreach \y in {1,2} {
\ifthenelse{\NOT \(\(0 = \z \AND \(\(1=\x \AND 2=\y\) \OR \(2=\x \AND 1=\y\)\)\) \OR \(1 = \z \AND \(\(4=\x \AND 2=\y\) \OR \(5=\x \AND 1=\y\)\)\)\)
}{\draw[thick] ({2*cos((2*pi*\x/6) r)},{6*\z+2*sin((2*pi*\x/6) r)}) -- ({2*cos((2*pi*\x/6) r)+cos((2*pi*\x/6-pi/3+2*pi/9*\y) r)},{6*\z+2*sin((2*pi*\x/6) r)+sin((2*pi*\x/6-pi/3+2*pi/9*\y) r)})}{};
}
}
\fill (0,6*\z) circle (3pt);
}
\draw[thick]({2*cos((2*pi/6) r)},{2*sin((2*pi/6) r)}) -- ({2*cos((2*pi*5/6) r)},{6+2*sin((2*pi*5/6) r)});
\draw[thick]({2*cos((2*pi*(2)/6) r)},{2*sin((2*pi*(2)/6) r)}) -- ({2*cos((2*pi*4/6) r)},{6+2*sin((2*pi*4/6) r)});
\end{scope}
\begin{scope}[shift={(7,0)}] 
\foreach \z in {0,1} {
\foreach \x in {0,1,...,5} {
\fill  ({2*cos((2*pi*\x/6) r)},{6*\z+2*sin((2*pi*\x/6) r)}) circle (3pt);
\draw[thick] ({2*cos((2*pi*(\x+1)/6) r)},{6*\z+2*sin((2*pi*(\x+1)/6) r)}) -- ({2*cos((2*pi*\x/6) r)},{6*\z+2*sin((2*pi*\x/6) r)}) -- (0,6*\z+0);
\foreach \y in {1,2} {
\ifthenelse{\NOT \(\(0 = \z \AND \(\(1=\x \AND 2=\y\) \OR \(2=\x \AND 1=\y\)\)\) \OR \(1 = \z \AND 4=\x \)\)
}{\draw[thick] ({2*cos((2*pi*\x/6) r)},{6*\z+2*sin((2*pi*\x/6) r)}) -- ({2*cos((2*pi*\x/6) r)+cos((2*pi*\x/6-pi/3+2*pi/9*\y) r)},{6*\z+2*sin((2*pi*\x/6) r)+sin((2*pi*\x/6-pi/3+2*pi/9*\y) r)})}{};
}
}
\fill (0,6*\z) circle (3pt);
}
\draw[thick]({2*cos((2*pi/6) r)},{2*sin((2*pi/6) r)}) -- ({2*cos((2*pi*4/6) r)},{6+2*sin((2*pi*4/6) r)});
\draw[thick]({2*cos((2*pi*(2)/6) r)},{2*sin((2*pi*(2)/6) r)}) -- ({2*cos((2*pi*4/6) r)},{6+2*sin((2*pi*4/6) r)});
\end{scope}
\begin{scope}[shift={(16,0)}] 
\foreach \z in {0,1} {
\foreach \x in {0,1,...,5} {
\fill  ({2*cos((2*pi*\x/6) r)},{6*\z+2*sin((2*pi*\x/6) r)}) circle (3pt);
\draw[thick] ({2*cos((2*pi*(\x+1)/6) r)},{6*\z+2*sin((2*pi*(\x+1)/6) r)}) -- ({2*cos((2*pi*\x/6) r)},{6*\z+2*sin((2*pi*\x/6) r)}) -- (0,6*\z+0);
\foreach \y in {1,2} {
\ifthenelse{\NOT \(\(0 = \z \AND \(1=\x \OR \(2=\x \AND 1=\y\)\)\) \OR \(1 = \z \AND \(4=\x \OR \(5=\x \AND 1=\y\)\)\)\)
}{\draw[thick] ({2*cos((2*pi*\x/6) r)},{6*\z+2*sin((2*pi*\x/6) r)}) -- ({2*cos((2*pi*\x/6) r)+cos((2*pi*\x/6-pi/3+2*pi/9*\y) r)},{6*\z+2*sin((2*pi*\x/6) r)+sin((2*pi*\x/6-pi/3+2*pi/9*\y) r)})}{};
}
}
\fill (0,6*\z) circle (3pt);
}
\draw[thick]({2*cos((2*pi/6) r)},{2*sin((2*pi/6) r)}) -- ({2*cos((2*pi*5/6) r)},{6+2*sin((2*pi*5/6) r)});
\draw[thick]({2*cos((2*pi*(2)/6) r)},{2*sin((2*pi*(2)/6) r)}) -- ({2*cos((2*pi*4/6) r)},{6+2*sin((2*pi*4/6) r)});
\draw[thick]({2*cos((2*pi/6) r)},{2*sin((2*pi/6) r)}) -- ({2*cos((2*pi*4/6) r)},{6+2*sin((2*pi*4/6) r)});
\end{scope}
\end{tikzpicture}
\caption{Parallel sequences with $m=2$ and $m=3$.}
\label{fig:parseq}
\end{figure}
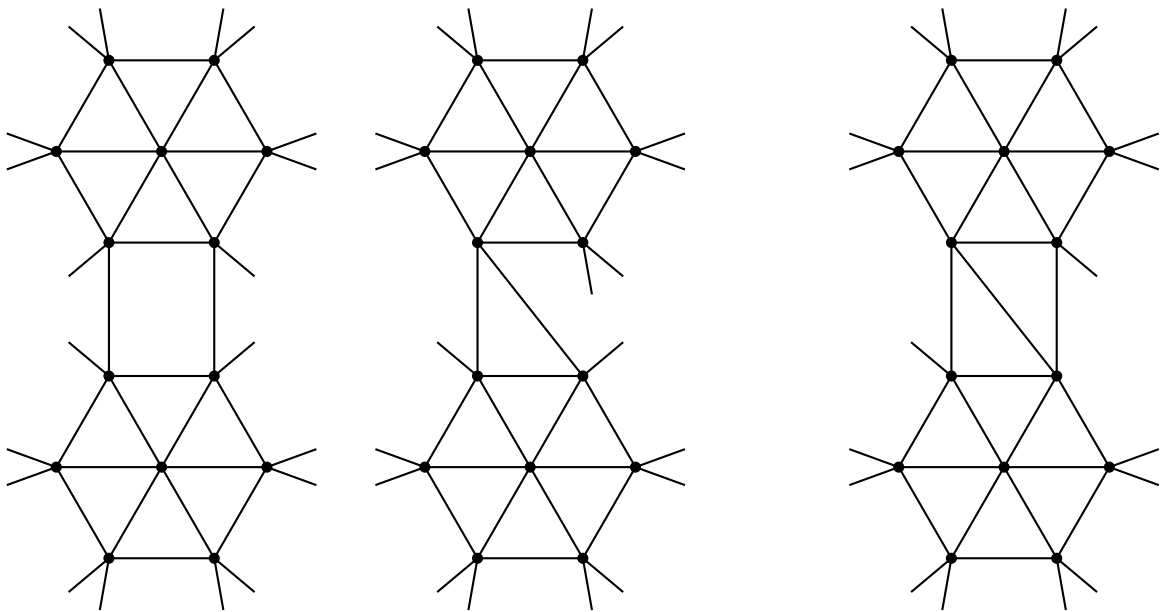

  We now construct a graph $\oG$ by contracting $G$ as follows. Every wheel of $G$ is contracted to a single points; these points are the vertices of $\oG$. Every \emph{maximal} parallel sequence of edges of $G$ (a parallel sequence which is not contained in a longer parallel sequence) is contracted to a single edge; these edges are the edges of $\oG$ (and the rotation diagram at every vertex of $\oG$ is that ``inherited" from $G$). We call the \emph{weight} of an edge of $\oG$ the length of the parallel sequence which has been contracted to it. The sum of the weights of the edges incident to a vertex of $\oG$ is $12\, (=d(W))$. Furthermore, $\oG$ is planar (and we identify it with its plane drawing obtained by ``contracting" the plane drawing of $G$). The graph $\oG$ may contain multiple edges (although no loops), but there are no $2$-faces (as the edges of $\oG$ are obtained by contracting maximal parallel sequences of edges of $G$). The graph $\oG$ has the following properties.
  
  
   \begin{prop} \label{p:Ghex}
    {\ }

    \begin{enumerate}[label=\emph{(\alph*)},ref=\alph*]
      \item \label{it:Ghex4}
      The degree of every vertex of $\oG$ is at least $4$.

      \item \label{it:Ghex4no3}
      No vertex of degree $4$ of $\oG$ lies on the boundary of a triangular face of $\oG$.

      \item  \label{it:Ghex533}
      Each vertex of degree $5$ of $\oG$ lies on the boundary of no more than three triangular faces of $\oG$.
    \end{enumerate}
  \end{prop}

  Assuming Proposition~\ref{p:Ghex}, the proof of the theorem is completed by the following Euler's formula argument.

  Let $\mathbf{v}_i$ be the number of vertices of degree $i$ in $\oG$ (by Proposition~\ref{p:Ghex}\eqref{it:Ghex4}, we have $\mathbf{v}_i=0$ for $i < 4$). Let $\mathbf{f}_j, \, j \ge 3$, be the number of $j$-gonal faces of $\oG$. Counting the sum over the vertices of $\oG$ of the numbers of adjacent triangular faces, from Proposition~\ref{p:Ghex}\eqref{it:Ghex4},\eqref{it:Ghex4no3},\eqref{it:Ghex533} we obtain

  \begin{equation}\label{eq:f3v5}
    3 \mathbf{f}_3 \le 3 \mathbf{v}_5 + \sum\nolimits_{i=6}^{\infty} i\mathbf{v}_i.
  \end{equation}
  If $\mathbf{v}, \mathbf{e}$ and $\mathbf{f}$ denote the numbers of the vertices, the edges and the faces of $\oG$ respectively, we have
  \begin{gather*}
    \mathbf{v} = \sum\nolimits_{i=4}^{\infty} \mathbf{v}_i, \quad 2\mathbf{e} = \sum\nolimits_{i=4}^{\infty} i\mathbf{v}_i = \sum\nolimits_{j=3}^{\infty} j\mathbf{f}_j, \\
    \mathbf{f} = \sum\nolimits_{j=3}^{\infty} \mathbf{f}_j = 2 + \sum\nolimits_{i=4}^{\infty} (\tfrac12 i -1)\mathbf{v}_i,
  \end{gather*}
 by Euler's formula. Expressing $2\mathbf{e}-4\mathbf{f}$ in terms of the $\mathbf{v}_i$'s and in terms of the $\mathbf{f}_j$'s we get
 \begin{equation*}
   2\mathbf{e}-4\mathbf{f} = -\mathbf{f}_3 + \sum\nolimits_{j=5}^{\infty} (j-4) \mathbf{f}_j = -8 - \mathbf{v}_5 + \sum\nolimits_{i=6}^{\infty} (4 - i)\mathbf{v}_i,
 \end{equation*}
 and so by~\eqref{eq:f3v5} we obtain
 \begin{equation*}
    0 < 8 +\sum\nolimits_{j=5}^{\infty} (j-4) \mathbf{f}_j = \mathbf{f}_3 - \mathbf{v}_5 + \sum\nolimits_{i=6}^{\infty} (4 - i)\mathbf{v}_i \le \sum\nolimits_{i=7}^{\infty} (4 - \tfrac23 i)\mathbf{v}_i \le 0,
 \end{equation*}
 a contradiction.
 \end{proof}

It remains to prove Proposition~\ref{p:Ghex}.

\begin{proof}[of Proposition~\ref{p:Ghex}]
We start with the following lemma\footnote{We are thankful to the anonymous reviewer for an idea which substantially simplified the proof.}.

\begin{lem} \label{l:par3}
The following holds.
\begin{enumerate}[label=\emph{(\alph*)},ref=\alph*]
    \item \label{it:widthle3}
    The maximal length of a parallel sequence of $G$ \emph{(}that is, the maximal weight of an edge of $\oG$\emph{)} is $3$. Moreover, if two wheels $W$ and $W'$ are connected by a parallel sequence of length $3$, then their corresponding spheres $S$ and $S'$ project to the same sphericon of $K_{2,2,2}$ and have the same orientation; up to relabeling of the vertices of $K_{2,2,2}$, we have the labelling of vertices of $S$ and $S'$ as in Figure~\ref{fig:threeseq}. 

   \item \label{it:two3inarow}
    A triangular face of $\oG$ cannot have two edges of weight $3$.
\end{enumerate}
\end{lem}

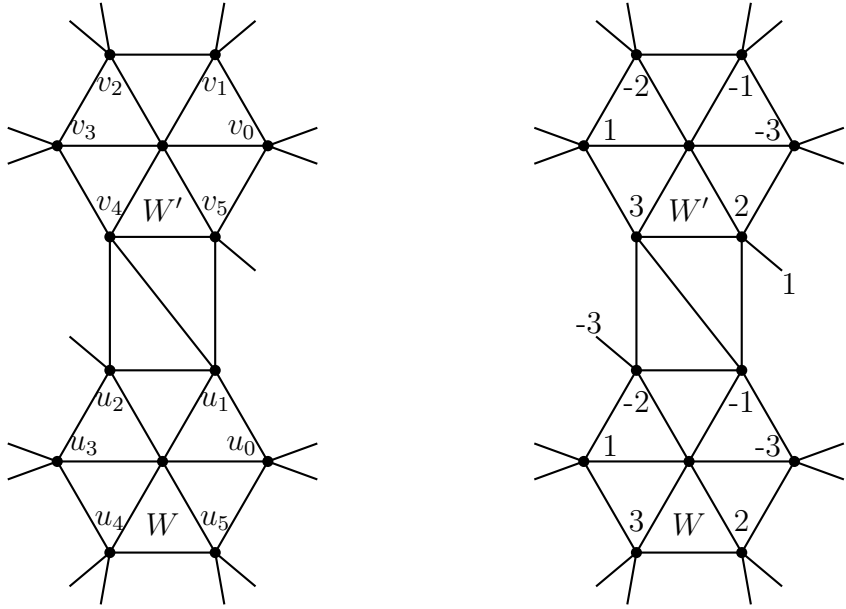
\begin{figure}[h]
\centering
\begin{tikzpicture}[scale=0.7]
%
\begin{scope}[shift={(10,0)},rotate=270] 
\def \r {2}
\def \a {pi/3}
\def \b {2*pi/9}
\def \h {6}
\newarray\testarray
\readarray{testarray}{-1&-2&1&3&2&-3}
\foreach \z in {0,1} {
\foreach \x in {0,1,...,5} {
\fill  ({\r*cos((\x*\a) r)},{\h*\z+\r*sin((\x*\a) r)}) circle (3pt);
\draw[thick] ({\r*cos(((\x+1)*\a) r)},{\h*\z+\r*sin(((\x+1)*\a) r)}) -- ({\r*cos((\x*\a) r)},{\h*\z+\r*sin((\x*\a) r)}) -- (0,\h*\z);
\ifthenelse{\x=1 \OR \x=2}{\node at ({\r*cos((\x*\a) r)},{\h*\z+\r*sin((\x*\a) r)-0.6}) {\testarray(\x)}}
{
\ifthenelse{\x=4 \OR \x=5}{\node at ({\r*cos((\x*\a) r)},{\h*\z+\r*sin((\x*\a) r)+0.6}) {\testarray(\x)}}{
\ifthenelse{\x=0}{\node at (\r-0.5,\h*\z+0.3) {-3}}{\node at (-\r+0.5,\h*\z+0.3) {1}
}
}
};
\foreach \y in {1,2} {
\ifthenelse{\NOT \(\(0 = \z \AND \(1=\x \OR \(2=\x \AND 1=\y\)\)\) \OR \(1 = \z \AND \(4=\x \OR \(5=\x \AND 1=\y\)\)\)\)
}{\draw[thick] ({\r*cos((\x*\a) r)},{\h*\z+\r*sin((\x*\a) r)}) --
({\r*cos((\x*\a) r)+cos((\x*\a-\a+\b*\y) r)},{\h*\z+\r*sin((\x*\a) r)+sin((\x*\a-\a+\b*\y) r)})
;}{}
}
}
\fill (0,\h*\z) circle (3pt);
}
\draw[thick]({\r*cos((\a) r)},{\r*sin((\a) r)}) -- ({\r*cos((5*\a) r)},{\h+\r*sin((5*\a) r)});
\draw[thick]({\r*cos((2*\a) r)},{\r*sin((2*\a) r)}) -- ({\r*cos((4*\a) r)},{\h+\r*sin((4*\a) r)});
\draw[thick]({\r*cos((\a) r)},{\r*sin((\a) r)}) -- ({\r*cos((4*\a) r)},{\h+\r*sin((4*\a) r)});
\node at ({\r*cos((5*\a) r)+0.9},{\h+\r*sin((5*\a) r)-0.9}) {1};
\node at ({\r*cos((2*\a) r)-0.9},{\r*sin((2*\a) r)+0.9}) {-3};
\node at (0,-3*\r/5) {$W$};\node at (0,\h-3*\r/5) {$W'$};
\end{scope}
\end{tikzpicture}
\caption{A parallel sequence of length $m=3$.}
\label{fig:threeseq}
\end{figure}

\begin{proof}
   For assertion~\eqref{it:widthle3} we note that a parallel sequence connecting two wheels $W$ and $W'$ projects to a path of $K_{2,2,2}$ which is edge disjoint from $\pi(S)$ and such that any two vertices at distance $2$ of that path are adjacent. From Figure~\ref{fig:zs} we see that no such path of length $4$ on $K_{2,2,2}$ exists, and if the path has length 3, then $\pi(S)$ is necessarily a sphericon, and the endpoints of the path are the midpoints of the straight $2$-paths of that sphericon. In the labelling as on the left in Figure~\ref{fig:zs}, the path is $(-2,3,-1,2)$ (up to symmetry). Then there is a unique labelling of the vertices of $S'$, as given in Figure~\ref{fig:threeseq}.

%

  For assertion~\eqref{it:two3inarow} we note that two consecutive edges of weight $3$ on the boundary of a face $\overline{F}$ of $\oG$ correspond to two parallel sequences of length $3$, each attached to the sphere $S$ of a wheel $W$ of $G$, which immediately follow one another along $S$. Up to reflection, there are two different such configurations depending on whether the parallel sequences share a vertex of $S$ or not, as shown in Figure~\ref{fig:two3s}. By assertion~\eqref{it:widthle3}, the labellings of the vertices in Figure~\ref{fig:two3s} are unique, up to relabelling the vertices of $S$. But then the face $\overline{F}$ cannot be triangular, as the vertices $2$ and $-2$ in the drawing on the left in Figure~\ref{fig:two3s} and the vertices $1$ and $-1$ (respectively, on the right) are not adjacent. 
\end{proof}

\begin{figure}[h]
\centering
\begin{tikzpicture}[scale=0.6]
\begin{scope}
\def \r {2}
\def \a {pi/3}
\def \b {2*pi/9}
\def \h {6}
\def \xs {5}
\def \ys {0}
\newarray\testarray
\readarray{testarray}{-1&-2&1&3&2&-3}
\foreach \z in {0,1} {
\foreach \x in {0,1,...,5} {
\fill  ({\r*cos((\x*\a) r)},{\h*\z+\r*sin((\x*\a) r)}) circle (3pt);
\draw[thick] ({\r*cos(((\x+1)*\a) r)},{\h*\z+\r*sin(((\x+1)*\a) r)}) -- ({\r*cos((\x*\a) r)},{\h*\z+\r*sin((\x*\a) r)}) -- (0,\h*\z);
\ifthenelse{\x=1 \OR \x=2}{\node at ({\r*cos((\x*\a) r)},{\h*\z+\r*sin((\x*\a) r)-0.6}) {\testarray(\x)}}
{
\ifthenelse{\x=4 \OR \x=5}{\node at ({\r*cos((\x*\a) r)},{\h*\z+\r*sin((\x*\a) r)+0.6}) {\testarray(\x)}}{
\ifthenelse{\x=0}{\node at (\r-0.5,\h*\z+0.3) {-3}}{\node at (-\r+0.5,\h*\z+0.3) {1}
}
}
};
\foreach \y in {1,2} {
\ifthenelse{\NOT \(\(0 = \z \AND \(0=\x \OR 1=\x \OR \(2=\x \AND 1=\y\)\OR \(5=\x \AND 2=\y\)\)\) \OR \(1 = \z \AND \(4=\x \OR \(5=\x \AND 1=\y\)\)\)\)
}{\draw[thick] ({\r*cos((\x*\a) r)},{\h*\z+\r*sin((\x*\a) r)}) --
({\r*cos((\x*\a) r)+cos((\x*\a-\a+\b*\y) r)},{\h*\z+\r*sin((\x*\a) r)+sin((\x*\a-\a+\b*\y) r)})
;}{}
}
}
\fill (0,\h*\z) circle (3pt);
}
\draw[thick]({\r*cos((\a) r)},{\r*sin((\a) r)}) -- ({\r*cos((5*\a) r)},{\h+\r*sin((5*\a) r)});
\draw[thick]({\r*cos((2*\a) r)},{\r*sin((2*\a) r)}) -- ({\r*cos((4*\a) r)},{\h+\r*sin((4*\a) r)});
\draw[thick]({\r*cos((\a) r)},{\r*sin((\a) r)}) -- ({\r*cos((4*\a) r)},{\h+\r*sin((4*\a) r)});
%
\foreach \x in {0,1,...,5} {
\fill  ({\xs+\r*cos((\x*\a) r)},{\ys+\r*sin((\x*\a) r)}) circle (3pt);
\draw[thick] ({\xs+\r*cos(((\x+1)*\a) r)},{\ys+\r*sin(((\x+1)*\a) r)}) -- ({\xs+\r*cos((\x*\a) r)},{\ys+\r*sin((\x*\a) r)}) -- (\xs,\ys);
\ifthenelse{\x=1 \OR \x=2}{\node at ({\xs+\r*cos((\x*\a) r)},{\ys+\r*sin((\x*\a) r)-0.6}) {\testarray(\x)}}
{
\ifthenelse{\x=4 \OR \x=5}{\node at ({\xs+\r*cos((\x*\a) r)},{\ys+\r*sin((\x*\a) r)+0.6}) {\testarray(\x)}}{
\ifthenelse{\x=0}{\node at (\xs+\r-0.5,\ys+0.3) {-3}}{\node at (\xs-\r+0.5,\ys+0.3) {1}
}
}
};
\foreach \y in {1,2} {
\ifthenelse{\NOT \( \(2=\x \AND 2=\y\) \OR 3=\x \)
}{\draw[thick] ({\xs+\r*cos((\x*\a) r)},{\ys+\r*sin((\x*\a) r)}) --
({\xs+\r*cos((\x*\a) r)+cos((\x*\a-\a+\b*\y) r)},{\ys+\r*sin((\x*\a) r)+sin((\x*\a-\a+\b*\y) r)})
;}{}
}
}
\fill (\xs,\ys) circle (3pt);
\draw[thick] ({\xs+\r*cos((2*\a) r)},{\ys+\r*sin((2*\a) r)}) -- ({\r*cos((0*\a) r)},{\r*sin((0*\a) r)}) -- ({\xs+\r*cos((3*\a) r)},{\ys+\r*sin((3*\a) r)}) -- ({\r*cos((5*\a) r)},{\r*sin((5*\a) r)});
\node at (\xs/2,\r) {$F$};
\node at (0,-3*\r/5) {$W$};
\end{scope}
\begin{scope}[shift={(16.5,0)}] 
\def \r {2}
\def \a {pi/3}
\def \b {2*pi/9}
\def \h {6}
\def \xs {-5}
\def \ys {2}
\newarray\testarray
\readarray{testarray}{-1&-2&1&3&2&-3}
\foreach \z in {0,1} {
\foreach \x in {0,1,...,5} {
\fill  ({\r*cos((\x*\a) r)},{\h*\z+\r*sin((\x*\a) r)}) circle (3pt);
\draw[thick] ({\r*cos(((\x+1)*\a) r)},{\h*\z+\r*sin(((\x+1)*\a) r)}) -- ({\r*cos((\x*\a) r)},{\h*\z+\r*sin((\x*\a) r)}) -- (0,\h*\z);
\ifthenelse{\x=1 \OR \x=2}{\node at ({\r*cos((\x*\a) r)},{\h*\z+\r*sin((\x*\a) r)-0.6}) {\testarray(\x)}}
{
\ifthenelse{\x=4 \OR \x=5}{\node at ({\r*cos((\x*\a) r)},{\h*\z+\r*sin((\x*\a) r)+0.6}) {\testarray(\x)}}{
\ifthenelse{\x=0}{\node at (\r-0.5,\h*\z+0.3) {-3}}{\node at (-\r+0.5,\h*\z+0.3) {1}
}
}
};
\foreach \y in {1,2} {
\ifthenelse{\NOT \(\(0 = \z \AND \(1=\x \OR 2=\x \OR 3=\x \)\) \OR \(1 = \z \AND \(4=\x \OR \(5=\x \AND 1=\y\)\)\)\)
}{\draw[thick] ({\r*cos((\x*\a) r)},{\h*\z+\r*sin((\x*\a) r)}) --
({\r*cos((\x*\a) r)+cos((\x*\a-\a+\b*\y) r)},{\h*\z+\r*sin((\x*\a) r)+sin((\x*\a-\a+\b*\y) r)})
;}{}
}
}
\fill (0,\h*\z) circle (3pt);
}
\draw[thick]({\r*cos((\a) r)},{\r*sin((\a) r)}) -- ({\r*cos((5*\a) r)},{\h+\r*sin((5*\a) r)});
\draw[thick]({\r*cos((2*\a) r)},{\r*sin((2*\a) r)}) -- ({\r*cos((4*\a) r)},{\h+\r*sin((4*\a) r)});
\draw[thick]({\r*cos((\a) r)},{\r*sin((\a) r)}) -- ({\r*cos((4*\a) r)},{\h+\r*sin((4*\a) r)});
%
\foreach \x in {0,1,...,5} {
\fill  ({\xs+\r*cos((\x*\a) r)},{\ys+\r*sin((\x*\a) r)}) circle (3pt);
\draw[thick] ({\xs+\r*cos(((\x+1)*\a) r)},{\ys+\r*sin(((\x+1)*\a) r)}) -- ({\xs+\r*cos((\x*\a) r)},{\ys+\r*sin((\x*\a) r)}) -- (\xs,\ys);
\ifthenelse{\x=1 \OR \x=2}{\node at ({\xs+\r*cos((\x*\a) r)},{\ys+\r*sin((\x*\a) r)-0.6}) {\testarray(\x)}}
{
\ifthenelse{\x=4 \OR \x=5}{\node at ({\xs+\r*cos((\x*\a) r)},{\ys+\r*sin((\x*\a) r)+0.6}) {\testarray(\x)}}{
\ifthenelse{\x=0}{\node at (\xs+\r-0.5,\ys+0.3) {-3}}{\node at (\xs-\r+0.5,\ys+0.3) {1}
}
}
};
\foreach \y in {1,2} {
\ifthenelse{\NOT \( 0=\x \OR \(5=\x \AND 2=\y\)\)
}{\draw[thick] ({\xs+\r*cos((\x*\a) r)},{\ys+\r*sin((\x*\a) r)}) --
({\xs+\r*cos((\x*\a) r)+cos((\x*\a-\a+\b*\y) r)},{\ys+\r*sin((\x*\a) r)+sin((\x*\a-\a+\b*\y) r)})
;}{}
}
}
\fill (\xs,\ys) circle (3pt);
\draw[thick] ({\r*cos((2*\a) r)},{\r*sin((2*\a) r)}) -- ({\xs+\r*cos((0*\a) r)},{\ys+\r*sin((0*\a) r)}) -- ({\r*cos((3*\a) r)},{\r*sin((3*\a) r)}) -- ({\xs+\r*cos((5*\a) r)},{\ys+\r*sin((5*\a) r)});
\node at (9*\xs/20,5*\ys/3) {$F$};
\node at (0,-3*\r/5) {$W$};
\end{scope} 
\end{tikzpicture}
\caption{Two cases for two consecutive parallel sequences of length $3$ attached to the sphere of a wheel $W$.}
\label{fig:two3s}
\end{figure}
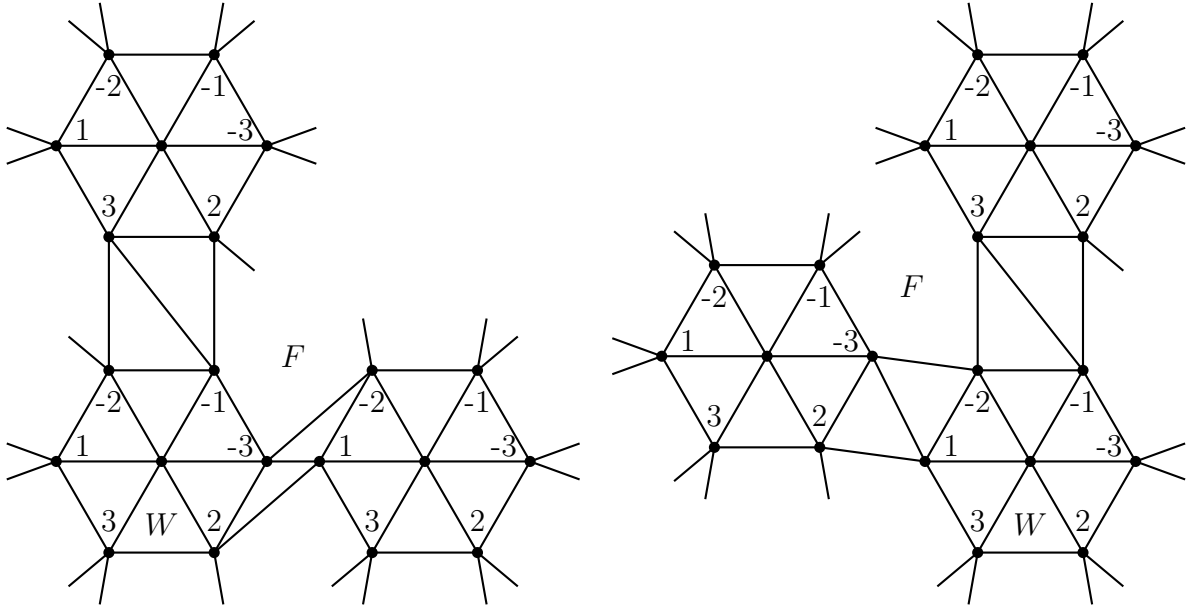

Assertions~\eqref{it:Ghex4} and~\eqref{it:Ghex4no3} of the proposition now follow immediately from Lemma~\ref{l:par3} and the fact that the sum of the weights of the edges of $\oG$ incident to a vertex is $12$, 

To prove assertion~\eqref{it:Ghex533} of the proposition we take a vertex $\ov$ of $\oG$ of degree $5$ and consider all possible weight distributions for the five edges incident to it. As the sum of the weights is $12$ and as each of them is at most $3$ by Lemma~\ref{l:par3}, we obtain the following six possible cyclic sequences of weights, counted in the positive direction of rotation at the vertex $\ov$:

\begin{enumerate}[label=(\roman*),ref=\roman*]
  \item \label{it:33321}
  $(3,3,3,2,1)$,

  \item \label{it:33312}
  $(3,3,3,1,2)$,

  \item \label{it:33231}
  $(3,3,2,3,1)$,

  \item \label{it:33132}
  $(3,3,1,3,2)$,

  \item \label{it:33222}
  $(3,3,2,2,2)$,

  \item \label{it:32322}
  $(3,2,3,2,2)$.
\end{enumerate}

For the sequences~\eqref{it:33321} and \eqref{it:33312}, the claim of assertion~\eqref{it:Ghex533} of the proposition follows from Lemma~\ref{l:par3}\eqref{it:two3inarow}, as a face of $\oG$ having two consecutive edges of weight $3$ on its boundary cannot be triangular.

Consider the sequences~\eqref{it:33231}, \eqref{it:33132}, \eqref{it:33222} and~\eqref{it:32322}. The vertex $\ov$ is the contraction of a wheel $W \subset G$. As there are edges of weight $3$ incident to $\ov$, the projection of the sphere $S \subset W$ to $K_{2,2,2}$ must be a sphericon, by Lemma~\ref{l:par3}. With the labelling as in Figure~\ref{fig:threeseq} (it agrees with the labelling in Figure~\ref{fig:two3s} and on the left in Figure~\ref{fig:zs}), the $2$-paths whose middle vertices are $\pm 2$ project to straight paths of $K_{2,2,2}$, see Remark~\ref{rem:svsz}. Then from Remark~\ref{rem:stli} it follows that no parallel sequence attached to $S$ may have two edges incident to a vertex $\pm 2$, as the two other endpoints of these edges are not adjacent. We obtain, with the labelling as in Figure~\ref{fig:threeseq}, that the two external edges incident to the vertex $2$ belong to two different parallel sequences attached to $S$. It follows that the two external edges incident to the vertex $-3$ and one of the external edges incident to the vertex $2$ belong to the same parallel sequence of length $3$ attached to $S$, as no parallel sequences of lengths $1$ or $2$ attached to $S$ can be consecutive in the sequences~\eqref{it:33231}, \eqref{it:33132}, \eqref{it:33222} and~\eqref{it:32322}. It follows that every parallel sequence of length $3$ attached to $S$ must have another parallel sequence of length $3$ immediately following it in the cyclic order, and that these two parallel sequences do not share a common vertex of $S$. This shows that neither of the sequences~\eqref{it:33231}, \eqref{it:33132} nor \eqref{it:32322} may occur.

Hence the only remaining case is the sequence~\eqref{it:33222}, for which the above argument shows that the two consecutive parallel sequences of length $3$ attached to $S$ do not share a vertex, and so may only be attached to $S$ as on the left in Figure~\ref{fig:two3s}. From Lemma~\ref{l:par3}\eqref{it:two3inarow} we already know that the face of $\oG$ having two consecutive edges of weight $3$ on its boundary is not triangular, and so it suffices to show that at least one of the other four faces of $\oG$ having $\ov$ on its boundary is not triangular. Assume that all of them are triangular. This means that each of the corresponding faces of $G$, before contracting to $\oG$, has exactly three external edges on its boundary (and any number of edges belonging to spheres).

We start with the configuration on the left in Figure~\ref{fig:two3s}. Note that the $6$ remaining external edges attached to $S$ form three consecutive parallel sequences of length $2$. In particular, two external edges attached to $S$, one incident to vertex $2$ and one, to vertex $3$ belong to a parallel sequence of length $2$. We have two possible cases: either these two edges share a common endpoint or not (as in Figure~\ref{fig:parseq}, in the middle or on the left, respectively).

In the first case, the common endpoint of the two edges is a $5$-vertex with label $-1$ (as this is the only remaining label for vertex $2$ of $W$); it cannot belong to $W$ by Lemma~\ref{l:nobade}, and it cannot belong to any of the other two wheels on the left in Figure~\ref{fig:two3s}, as we get a face with a ``very long" boundary. Furthermore, the second external edge incident to vertex $3$ of $W$ ends at a $5$-vertex with label $-2$. That edge belongs to a parallel sequence of length $2$ attached to $S$, the second one being incident to vertex $1$ of $S$. The other endpoints of two external edges incident to that vertex have labels $2$ and $-3$, and so the parallel sequence is of the type shown on the left in Figure~\ref{fig:parseq}; as the vertices $2$ and $-2$ are not connected, the parallel edges must be $(3,-2)$ and $(1,-3)$. We arrive at the drawing as on the left in Figure~\ref{fig:33222first} (where for clarity, we do not show the interior vertices and edges of the wheels). But now for the face $F'$ to have no more than three external edges on its boundary, there must be an external edge joining the vertex $-2$ of $W_1$ with the vertex $u$ of $W_2$. As the vertex $-2$ of $W_1$ is already connected to vertices $\pm 3$, we obtain $u = \pm 1$, a contradiction: $W_2$ already has a vertex $-1$, and vertices $1$ and $-1$ cannot be connected.

In the second case, the argument is similar. We start with the drawing as on the left in Figure~\ref{fig:two3s}, with two external edges attached to $W$, at the vertex $2$ and at the vertex $3$ belonging to a parallel sequence of $2$ edges with no common endpoint (as on the left in Figure~\ref{fig:parseq}). The endpoint of the external edge incident to the vertex $2$ must be a $5$-vertex of a wheel $W_2$ with the label $-1$, and then the endpoint of the parallel external edge incident to the vertex $3$ must be a $5$-vertex of the wheel $W_2$ adjacent to its vertex $-1$, and so having label $-2$. Furthermore, the second external edge incident to the vertex $3$ of $W$ ends at a $5$-vertex with the label $-1$ of a wheel $W_1$ (note that $W_1 \ne W$ by Lemma~\ref{l:nobade} and that $W_1 \ne W_2$, as the boundary is longer than $4$). That edge belongs to a parallel sequence of length $2$ between $W$ and $W_1$, the second one being incident to the vertex $1$ of $W$. The endpoint of that edge lying on $W_1$ cannot have label $-1$, and so the parallel sequence of edges is of the type as on the left in Figure~\ref{fig:parseq}, which gives the drawing as on the right in Figure~\ref{fig:33222first} (we do not show the interior vertices and edges of the wheels). The contradiction follows from considering the face $F'$: for $F'$, to have no more than three external edges on its boundary, there must be an external edge joining the vertex $-1$ of $W_1$ with the vertex $-2$ of $W_2$ which is clearly impossible, as the vertex $-2$ of $W_2$ is already connected to the vertex $-1$ in $W_2$.

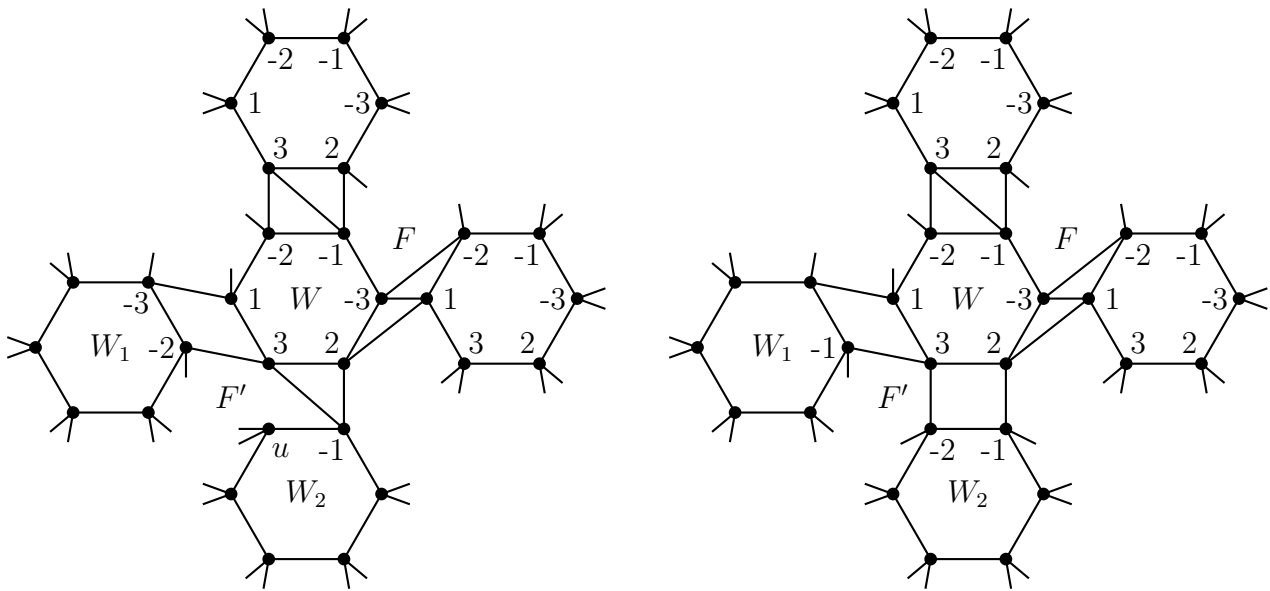
\begin{figure}[h]
\centering
\begin{tikzpicture}[scale=0.7]
\begin{scope}
\def \r {1.25}
\def \a {pi/3}
\def \b {2*pi/9}
\def \h {3.25}
\def \xs {3.25}
\def \ys {0}
\def \ger {0.5}
\def \rla {0.85}
\newarray\testarray
\readarray{testarray}{-1&-2&1&3&2&-3}
\foreach \z in {0,1} {
\foreach \x in {0,1,...,5} {
\fill  ({\r*cos((\x*\a) r)},{\h*\z+\r*sin((\x*\a) r)}) circle (3pt);
\draw[thick] ({\r*cos(((\x+1)*\a) r)},{\h*\z+\r*sin(((\x+1)*\a) r)}) -- ({\r*cos((\x*\a) r)},{\h*\z+\r*sin((\x*\a) r)}); 
\ifthenelse{\x=1 \OR \x=2}{\node at ({\rla*cos((\x*\a) r)},{\h*\z+\rla*sin((\x*\a) r)}) {\testarray(\x)}}
{
\ifthenelse{\x=4 \OR \x=5}{\node at ({\rla*cos((\x*\a) r)},{\h*\z+\rla*sin((\x*\a) r)}) {\testarray(\x)}}{
\ifthenelse{\x=0}{\node at (\rla,\h*\z) {-3}}{\node at (-\rla,\h*\z) {1}
}
}
};
\foreach \y in {1,2} {
\ifthenelse{
\(\z=1 \AND \NOT \(4=\x \OR \(5=\x \AND 1=\y\)\)\)
\OR \(\z=0 \AND \x=2 \AND \y=2\)
}{\draw[thick] ({\r*cos((\x*\a) r)},{\h*\z+\r*sin((\x*\a) r)}) --
({\r*cos((\x*\a) r)+\ger*cos((\x*\a-\a+\b*\y) r)},{\h*\z+\r*sin((\x*\a) r)+\ger*sin((\x*\a-\a+\b*\y) r)})
;}{}
}
}
}
\draw[thick](-\r,0) -- (-\r,\ger);
\draw[thick]({\r*cos((\a) r)},{\r*sin((\a) r)}) -- ({\r*cos((5*\a) r)},{\h+\r*sin((5*\a) r)});
\draw[thick]({\r*cos((2*\a) r)},{\r*sin((2*\a) r)}) -- ({\r*cos((4*\a) r)},{\h+\r*sin((4*\a) r)});
\draw[thick]({\r*cos((\a) r)},{\r*sin((\a) r)}) -- ({\r*cos((4*\a) r)},{\h+\r*sin((4*\a) r)});
%
\foreach \x in {0,1,...,5} {
\fill  ({\xs+\r*cos((\x*\a) r)},{\ys+\r*sin((\x*\a) r)}) circle (3pt);
\draw[thick] ({\xs+\r*cos(((\x+1)*\a) r)},{\ys+\r*sin(((\x+1)*\a) r)}) -- ({\xs+\r*cos((\x*\a) r)},{\ys+\r*sin((\x*\a) r)}); 
\ifthenelse{\x=1 \OR \x=2}{\node at ({\xs+\rla*cos((\x*\a) r)},{\ys+\rla*sin((\x*\a) r)}) {\testarray(\x)}}
{
\ifthenelse{\x=4 \OR \x=5}{\node at ({\xs+\rla*cos((\x*\a) r)},{\ys+\rla*sin((\x*\a) r)}) {\testarray(\x)}}{
\ifthenelse{\x=0}{\node at (\xs+\rla,\ys) {-3}}{\node at (\xs-\rla,\ys) {1}
}
}
};
\foreach \y in {1,2} {
\ifthenelse{\NOT \( \(2=\x \AND 2=\y\) \OR 3=\x \)
}{\draw[thick] ({\xs+\r*cos((\x*\a) r)},{\ys+\r*sin((\x*\a) r)}) --
({\xs+\r*cos((\x*\a) r)+\ger*cos((\x*\a-\a+\b*\y) r)},{\ys+\r*sin((\x*\a) r)+\ger*sin((\x*\a-\a+\b*\y) r)})
;}{}
}
}
\draw[thick] ({\xs+\r*cos((2*\a) r)},{\ys+\r*sin((2*\a) r)}) -- ({\r*cos((0*\a) r)},{\r*sin((0*\a) r)}) -- ({\xs+\r*cos((3*\a) r)},{\ys+\r*sin((3*\a) r)}) -- ({\r*cos((5*\a) r)},{\r*sin((5*\a) r)});
\node at (0.5*\xs,0.8*\r) {$F$};
\node at (-\r,-\h/2) {$F'$};
\node at (0,0) {$W$};
%
%
\def \z {-1}
\foreach \x in {0,1,...,5} {
\fill  ({\r*cos((\x*\a) r)},{\h*\z+\r*sin((\x*\a) r)}) circle (3pt);
\draw[thick] ({\r*cos(((\x+1)*\a) r)},{\h*\z+\r*sin(((\x+1)*\a) r)}) -- ({\r*cos((\x*\a) r)},{\h*\z+\r*sin((\x*\a) r)}); 
\ifthenelse{\x=1}{\node at ({\rla*cos((\x*\a) r)},{\h*\z+\rla*sin((\x*\a) r)}) {\testarray(\x)}}{};
\ifthenelse{\x=2}{\node at ({\rla*cos((\x*\a) r)},{\h*\z+\rla*sin((\x*\a) r)}) {$u$}}{};
\foreach \y in {1,2} {
\ifthenelse{\NOT \(1=\x \OR 2=\x \)}
{\draw[thick] ({\r*cos((\x*\a) r)},{\h*\z+\r*sin((\x*\a) r)}) --
({\r*cos((\x*\a) r)+\ger*cos((\x*\a-\a+\b*\y) r)},{\h*\z+\r*sin((\x*\a) r)+\ger*sin((\x*\a-\a+\b*\y) r)})
;}{}
}
}
\draw[thick]
({\r*cos((2*\a) r)-\ger},{\h*\z+\r*sin((2*\a) r)}) --
({\r*cos((2*\a) r)},{\h*\z+\r*sin((2*\a) r)}) --
({\r*cos((2*\a) r)-\ger},{\h*\z+\r*sin((2*\a) r)-\ger/2});
\node at (0,-\h) {$W_2$};
%
%
\def \z {-0.25}
\foreach \x in {0,1,...,5} {
\fill  ({\r*cos((\x*\a) r)-\h},{\h*\z+\r*sin((\x*\a) r)}) circle (3pt);
\draw[thick] ({\r*cos(((\x+1)*\a) r)-\h},{\h*\z+\r*sin(((\x+1)*\a) r)}) -- ({\r*cos((\x*\a) r)-\h},{\h*\z+\r*sin((\x*\a) r)}); 
\ifthenelse{\x=1}{\node at ({\rla*cos((\x*\a) r)-\h},{\h*\z+\rla*sin((\x*\a) r)}) {-3}}{};
\ifthenelse{\x=0}{\node at (\rla-\h,\h*\z) {-2}}{};
\foreach \y in {1,2} {
\ifthenelse{\NOT \(\(1=\x \AND \y=1\) \OR 0=\x\)}
{\draw[thick] ({\r*cos((\x*\a) r)-\h},{\h*\z+\r*sin((\x*\a) r)}) --
({\r*cos((\x*\a) r)+\ger*cos((\x*\a-\a+\b*\y) r)-\h},{\h*\z+\r*sin((\x*\a) r)+\ger*sin((\x*\a-\a+\b*\y) r)})
;}{}
}
}
\draw[thick] ({\r-\h},{\h*\z}) -- ({\r-\h},{\h*\z-\ger});
\node at (-\h,\h*\z) {$W_1$};
%
%
\draw[thick] ({\r*cos((5*\a) r)},{\r*sin((5*\a) r)}) -- ({\r*cos((\a) r)},{-\h+\r*sin((\a) r)}) -- ({\r*cos((4*\a) r)},{\r*sin((4*\a) r)}) --
({\r-\h},{\h*\z});
\draw[thick] ({\r*cos((\a) r)-\h},{\h*\z+\r*sin((\a) r)}) -- ({-\r},0);
\end{scope}
\begin{scope}[shift={(11,0)}]
\def \r {1.25}
\def \a {pi/3}
\def \b {2*pi/9}
\def \h {3.25}
\def \xs {3.25}
\def \ys {0}
\def \ger {0.5}
\def \rla {0.85}
\newarray\testarray
\readarray{testarray}{-1&-2&1&3&2&-3}
\foreach \z in {0,1} {
\foreach \x in {0,1,...,5} {
\fill  ({\r*cos((\x*\a) r)},{\h*\z+\r*sin((\x*\a) r)}) circle (3pt);
\draw[thick] ({\r*cos(((\x+1)*\a) r)},{\h*\z+\r*sin(((\x+1)*\a) r)}) -- ({\r*cos((\x*\a) r)},{\h*\z+\r*sin((\x*\a) r)}); 
\ifthenelse{\x=1 \OR \x=2}{\node at ({\rla*cos((\x*\a) r)},{\h*\z+\rla*sin((\x*\a) r)}) {\testarray(\x)}}
{
\ifthenelse{\x=4 \OR \x=5}{\node at ({\rla*cos((\x*\a) r)},{\h*\z+\rla*sin((\x*\a) r)}) {\testarray(\x)}}{
\ifthenelse{\x=0}{\node at (\rla,\h*\z) {-3}}{\node at (-\rla,\h*\z) {1}
}
}
};
\foreach \y in {1,2} {
\ifthenelse{
\(\z=1 \AND \NOT \(4=\x \OR \(5=\x \AND 1=\y\)\)\)
\OR \(\z=0 \AND \x=2 \AND \y=2\)
}{\draw[thick] ({\r*cos((\x*\a) r)},{\h*\z+\r*sin((\x*\a) r)}) --
({\r*cos((\x*\a) r)+\ger*cos((\x*\a-\a+\b*\y) r)},{\h*\z+\r*sin((\x*\a) r)+\ger*sin((\x*\a-\a+\b*\y) r)})
;}{}
}
}
}
\draw[thick](-\r,0) -- (-\r,\ger);
\draw[thick]({\r*cos((\a) r)},{\r*sin((\a) r)}) -- ({\r*cos((5*\a) r)},{\h+\r*sin((5*\a) r)});
\draw[thick]({\r*cos((2*\a) r)},{\r*sin((2*\a) r)}) -- ({\r*cos((4*\a) r)},{\h+\r*sin((4*\a) r)});
\draw[thick]({\r*cos((\a) r)},{\r*sin((\a) r)}) -- ({\r*cos((4*\a) r)},{\h+\r*sin((4*\a) r)});
%
\foreach \x in {0,1,...,5} {
\fill  ({\xs+\r*cos((\x*\a) r)},{\ys+\r*sin((\x*\a) r)}) circle (3pt);
\draw[thick] ({\xs+\r*cos(((\x+1)*\a) r)},{\ys+\r*sin(((\x+1)*\a) r)}) -- ({\xs+\r*cos((\x*\a) r)},{\ys+\r*sin((\x*\a) r)}); 
\ifthenelse{\x=1 \OR \x=2}{\node at ({\xs+\rla*cos((\x*\a) r)},{\ys+\rla*sin((\x*\a) r)}) {\testarray(\x)}}
{
\ifthenelse{\x=4 \OR \x=5}{\node at ({\xs+\rla*cos((\x*\a) r)},{\ys+\rla*sin((\x*\a) r)}) {\testarray(\x)}}{
\ifthenelse{\x=0}{\node at (\xs+\rla,\ys) {-3}}{\node at (\xs-\rla,\ys) {1}
}
}
};
\foreach \y in {1,2} {
\ifthenelse{\NOT \( \(2=\x \AND 2=\y\) \OR 3=\x \)
}{\draw[thick] ({\xs+\r*cos((\x*\a) r)},{\ys+\r*sin((\x*\a) r)}) --
({\xs+\r*cos((\x*\a) r)+\ger*cos((\x*\a-\a+\b*\y) r)},{\ys+\r*sin((\x*\a) r)+\ger*sin((\x*\a-\a+\b*\y) r)})
;}{}
}
}
\draw[thick] ({\xs+\r*cos((2*\a) r)},{\ys+\r*sin((2*\a) r)}) -- ({\r*cos((0*\a) r)},{\r*sin((0*\a) r)}) -- ({\xs+\r*cos((3*\a) r)},{\ys+\r*sin((3*\a) r)}) -- ({\r*cos((5*\a) r)},{\r*sin((5*\a) r)});
\node at (0.5*\xs,0.8*\r) {$F$};
\node at (-\r,-\h/2) {$F'$};
\node at (0,0) {$W$};
%
%
\def \z {-1}
\foreach \x in {0,1,...,5} {
\fill  ({\r*cos((\x*\a) r)},{\h*\z+\r*sin((\x*\a) r)}) circle (3pt);
\draw[thick] ({\r*cos(((\x+1)*\a) r)},{\h*\z+\r*sin(((\x+1)*\a) r)}) -- ({\r*cos((\x*\a) r)},{\h*\z+\r*sin((\x*\a) r)}); 
\ifthenelse{\x=1}{\node at ({\rla*cos((\x*\a) r)},{\h*\z+\rla*sin((\x*\a) r)}) {\testarray(\x)}}{};
\ifthenelse{\x=2}{\node at ({\rla*cos((\x*\a) r)},{\h*\z+\rla*sin((\x*\a) r)}) {-2}}{};
\foreach \y in {1,2} {
\ifthenelse{\NOT \(1=\x \OR 2=\x \)}
{\draw[thick] ({\r*cos((\x*\a) r)},{\h*\z+\r*sin((\x*\a) r)}) --
({\r*cos((\x*\a) r)+\ger*cos((\x*\a-\a+\b*\y) r)},{\h*\z+\r*sin((\x*\a) r)+\ger*sin((\x*\a-\a+\b*\y) r)})
;}{}
}
}
\draw[thick] ({\r*cos((2*\a) r)},{\h*\z+\r*sin((2*\a) r)}) -- ({\r*cos((2*\a) r)-\ger},{\h*\z+\r*sin((2*\a) r)-\ger/2});
\draw[thick] ({\r*cos((\a) r)},{\h*\z+\r*sin((\a) r)}) -- ({\r*cos((\a) r)+\ger},{\h*\z+\r*sin((\a) r)-\ger/2});
\node at (0,-\h) {$W_2$};
%
%
\def \z {-0.25}
\foreach \x in {0,1,...,5} {
\fill  ({\r*cos((\x*\a) r)-\h},{\h*\z+\r*sin((\x*\a) r)}) circle (3pt);
\draw[thick] ({\r*cos(((\x+1)*\a) r)-\h},{\h*\z+\r*sin(((\x+1)*\a) r)}) -- ({\r*cos((\x*\a) r)-\h},{\h*\z+\r*sin((\x*\a) r)}); 
\ifthenelse{\x=0}{\node at (\rla-\h,\h*\z) {-1}}{};
\foreach \y in {1,2} {
\ifthenelse{\NOT \(\(1=\x \AND \y=1\) \OR 0=\x\)}
{\draw[thick] ({\r*cos((\x*\a) r)-\h},{\h*\z+\r*sin((\x*\a) r)}) --
({\r*cos((\x*\a) r)+\ger*cos((\x*\a-\a+\b*\y) r)-\h},{\h*\z+\r*sin((\x*\a) r)+\ger*sin((\x*\a-\a+\b*\y) r)})
;}{}
}
}
\draw[thick] ({\r-\h},{\h*\z}) -- ({\r-\h},{\h*\z-\ger});
\node at (-\h,\h*\z) {$W_1$};
%
%
\draw[thick] ({\r*cos((5*\a) r)},{\r*sin((5*\a) r)}) -- ({\r*cos((\a) r)},{-\h+\r*sin((\a) r)});
\draw[thick] ({\r*cos((2*\a) r)},{-\h+\r*sin((2*\a) r)}) -- ({\r*cos((4*\a) r)},{\r*sin((4*\a) r)}) --
({\r-\h},{\h*\z});
\draw[thick] ({\r*cos((\a) r)-\h},{\h*\z+\r*sin((\a) r)}) -- ({-\r},0);
\end{scope}
\end{tikzpicture}
\caption{Sequence \eqref{it:33222}: the face $F'$ has more than $3$ external edges.}
\label{fig:33222first}
\end{figure}

This completes the proof of the proposition and hence of Theorem~\ref{th:noperf}.
\end{proof}

\section{Structural properties of the semi-cover} 
\label{s:all9thm}

In~\cite[Section~4]{annor25k1222}, starting with a purported planar cover $G$ of $K_{1,2,2,2}$ we constructed a specific semi-cover $G'$ (following the construction introduced in~\cite[Section~2]{hlinueny2001}) enjoying several very restrictive properties stated below in Lemma~\ref{l:Hfaces}. The proof of Negami's Conjecture would follow from the non-existence of such a semi-cover~\cite[Conjecture~2]{annor25k1222}. In this section, we establish some further properties of such semi-covers and prove non-existence under a certain additional hypothesis (see Theorem~\ref{th:notall9} below). 

We briefly recall the construction from~\cite[Section~4]{annor25k1222}. Assume $K_{1,2,2,2}$ has a finite planar cover. Out of all such covers, we choose one with the following properties:

\begin{enumerate}[label=(\Alph*),ref=\Alph*]
    \item \label{it:minfold}
    It has the smallest fold number.

    \item \label{it:max3}
    Out of those covers satisfying~\eqref{it:minfold}, it has the maximal number of triangular faces.

    \item \label{it:shortf}
    Any short cycle covering a triangle in $K_{1,2,2,2}$ bounds a face.~

    \item \label{it:longf}
     No long cycle covering a triangle in $K_{1,2,2,2}$ bounds a face.
\end{enumerate}

Note that property~\eqref{it:shortf} is guaranteed by Lemma~\ref{lem:shortcycle}, and that the possibility to choose a planar cover with property~\eqref{it:longf} among the planar covers having the other three properties follows from~\cite[Lemma~3]{annor25k1222}.

A cover $G$ of $K_{1, 2, 2, 2}$ satisfying conditions~\eqref{it:minfold}, \eqref{it:max3}, \eqref{it:shortf} and \eqref{it:longf} must contain at least one long cycle covering an octahedral $3$-cycle. Among these long cycles we choose a cycle $C$ which contains no other such long cycle inside the closed domain that it bounds; call this interior domain $F$.

We adopt the labelling convention for the vertices of $G$ and of $K_{1,2,2,2}$ as in Section~\ref{sec:prel}. Without loss of generality, we can assume that $C$ covers  the $3$-cycle $(1,2,3)$. The domain $F$ is not a face of $G$ by condition \eqref{it:longf}, and hence there are some vertices labelled $0, -1,-2$ and $-3$ lying in $F$. We define $H$ to be a connected component of the lift of the $K_4$ subgraph on the vertices $0,-1,-2,$ and $-3$ lying inside the domain $F$ that contains no other such component in its interior. Moreover, we define $G'$ to be the semi-cover of $K_{1,2,2,2}$ consisting of $H$ and all the edges and vertices of $G$ that lie in the internal faces of $H$.

Clearly, the length of the bounding cycle of every internal face of $H$ is a multiple of $3$, and some of these cycles must be long. One can prove that $H$ contains no hexagonal cycles (see Lemma~\ref{l:Hfaces}\eqref{it:no6face} below). 

The main results of this section is as follows.


\begin{thm} \label{th:notall9}
  Let $G'$ be a planar semi-cover of $K_{1,2,2,2}$ constructed above. Then the subgraph $H$ cannot have all its internal, non-triangular faces of length $9$.
\end{thm}

The proof of Theorem given in the rest of this section is quite technically involved. An important tool in the proof is a sequence of surgeries transforming the given semi-cover $G'$ and the subgraph $H$ to the ones which are smaller or which enjoy some specific structural properties. It may a priori happen that the semi-cover obtained as the result of such a surgery is no longer a semi-cover coming from a ``full'' planar cover $G$. 

To overcome this, in Lemma~\ref{l:Hfaces} we introduce a long, but comprehensive list of properties which the semi-cover $G'$ and the subgraph $H$ from Theorem~\ref{th:notall9} enjoy, and then prove in Theorem~\hyperref[{th:notall9'}]{$3'$} below that for any semi-cover $G'$ and the subgraph $H$ with these properties (but \textit{not necessarily} coming from a full cover $G$) the claim of Theorem~\ref{th:notall9} holds.

\begin{lem}\cite[Lemma~4]{annor25k1222} \label{l:Hfaces}
The semi-cover $G'$ and the subgraph $H\subset G'$ introduced above have the following properties.

\begin{enumerate}[label=\emph{(\alph*)},ref=\alph*]
    \item \label{it:s9H}
    $H \subset G'$ is the lift of the subgraph $K_4 \subset K_{1,2,2,2}$ whose vertices are $0, -1, -2, -3$, 
    $H$ is connected, the restriction of the projection map to $H$ is a (genuine) cover of $K_4$, and the outer boundary of $G'$ is a cycle of $H$.

    \item \label{it:s9tri}    
    All $3$-cycles of $G'$ are facial.
    
    \item \label{it:Htri} Any cycle in $G'$ covering an octahedral $3$-cycle is a triangular face. 

    \item \label{it:s9inn3}
    Every component of the lift of an octahedral $3$-cycle in $G'$ that is not a $3$-cycle is a path that starts and ends on the boundary of $G'$. Such a path cannot cover $(1,2,3)$ or $(-1,-2,-3)$, so the lift of these triangles in $G'$ consists only of triangles.

    
    \item \label{it:noK4}
    $H$ is not isomorphic to $K_4$. 

    \item \label{it:notouter}
    $G'$ contains at least one triangle labelled $(1,2,3)$. 


    \item \label{it:2conn}
    $H$ is $2$-connected, and so every internal face of $H$ is homeomorphic to a disc. 

    \item \label{it:mult3}
    The cyclic order of vertex labels around any non-triangular face of $H$ is $0,a_1,b_1,0,$ $a_2,b_2, \dots,$ $0,a_m,b_m$, where $a_i, b_i \in \{-1,-2,-3\}$; in particular, the length of any facial cycle is a multiple of $3$.

    \item \label{it:no6face}
    No internal face of $H$ is hexagonal. 

    \item \label{it:above} 
    If an internal face of $H$ is $3m$-gonal, where $m \ge 1$, and contains $t$ triangles labelled $(1,2,3)$, then $t < \frac23 m$.
\end{enumerate}
\end{lem}

For the remainder of this section, let $G'$ be a semi-cover with subgraph $H$, both satisfying the properties listed in Lemma~\ref{l:Hfaces}. 
We call a \emph{bead} a labelled subgraph of $H$ shown on the left in Figure~\ref{fig:empty9}, where  $\{i,j,k\}=\{1,2,3\}$. We say that a face of $H$ is \emph{empty} if it is internal and contains no $3$-cycles $(1,2,3)$ inside, and is \emph{full} otherwise. Clearly, all triangular faces of $H$ are empty by~Lemma~\ref{l:Hfaces}\eqref{it:s9tri} and every vertex of $H$ is adjacent to some full face.

The following proposition relies on the assumption that $G'$ comes from a cover $G$ of $K_{1,2,2,2}$ with the maximum possible number of triangles (property~\eqref{it:max3}). This was an assumption used in the derivation of Lemma~\ref{l:Hfaces}, but which we do not apply to the semi-covers $G'$ in general.

\begin{prop} \label{p:empty}
Suppose $G'$ has the maximum number of triangles among all semi-covers satisfying Lemma~\ref{l:Hfaces} with a given number of vertices.  Then any non-triangular empty (internal) face of $H$ can be assumed to have length $9$. Moreover, the only possible labelling of its boundary, up to permutation, is the one on the right in Figure~\ref{fig:empty9}, with all the neighbouring non-triangular faces of $H$ being full.
\end{prop}

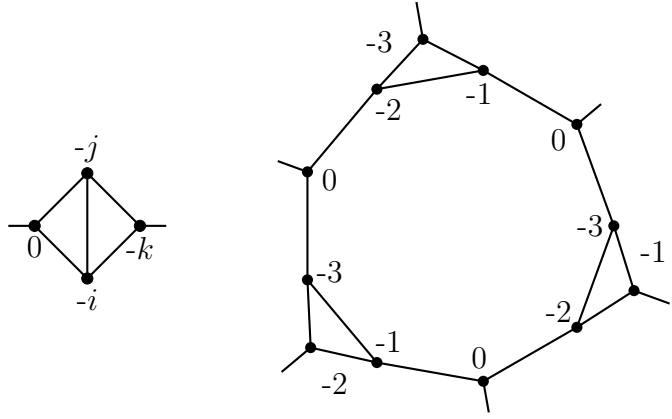
\begin{figure}[ht]
\centering
\begin{tikzpicture}[scale=0.7]
\def \r {3}
\def \b {2*pi/9}
\newarray\tarray
\readarray{tarray}{0&-1&-2&0&-3&-1&0&-2&-3}
\foreach \x in {1,2,...,9}{
\coordinate(A\x) at ({\r*cos((\x*\b) r)},{\r*sin((\x*\b) r)}); \fill (A\x) circle (3pt);
\node at ($0.85*(A\x)$) {\tarray(\x)};
}
\foreach \x in {1,2,...,9}{ \ifthenelse{\x = 9}{\pgfmathsetmacro{\y}{1}}{\pgfmathsetmacro{\y}{\x + 1}}; 
\draw[thick] (A\x) -- (A\y);
}
\foreach \z in {1,4,7}{\draw[thick] (A\z) -- ($1.2*(A\z)$);}
\foreach \z in {2,5,8}{\pgfmathsetmacro{\zo}{\z+1}; \ifthenelse{\z = 8}{\pgfmathtruncatemacro{\w}{2}}{\pgfmathtruncatemacro{\w}{\z + 3}};
\coordinate(B\z) at ({1.2*\r*cos(((\z+0.5)*\b) r)},{1.2*\r*sin(((\z+0.5)*\b) r)}); \fill (B\z) circle (3pt);
\draw[thick] (B\z) -- ($1.2*(B\z)$); \draw[thick] (A\z) -- (B\z) -- (A\zo);
\node at ($0.8*(A\zo)+0.4*(B\z)$) {\tarray(\w)};
}

\begin{scope}[shift={(-7, -4)}]
\draw[ thick ] (0, 5) -- (1,4) -- (0,3) -- (-1, 4) -- (0, 5) -- (0,3);
\draw[ thick ] (-1.5, 4)--(-1,4) (1,4) -- (1.5,4);
\filldraw[black] (-1, 4) circle (3pt) node[anchor=north]{$0$};
\filldraw[black] (0, 5) circle (3pt) node[anchor=south]{-$j$};
\filldraw[black] (1, 4) circle (3pt) node[anchor=north]{-$k$};
\filldraw[black] (0, 3) circle (3pt) node[anchor=north]{-$i$};
\end{scope}

\end{tikzpicture}
\caption{Left to right: a bead and an empty $9$-face of $H$.}
\label{fig:empty9}
\end{figure}

\begin{proof}
Consider a non-triangular empty face $F'$ of $H$. Note that $F'$ must be a face of $G'$ as well, and that no edge on its boundary $C'$ is an edge of some bead (as every bead forces both faces adjacent to it to be full). Suppose $0_1,a_1,b_1,0_2,a_2,b_2, \dots, 0_m,a_m,b_m$, where $a_i, b_i \in \{-1,-2,-3\}$, is the cyclic order of vertices of $C'$ (as per Lemma~\ref{l:Hfaces}\eqref{it:mult3}). Clearly $a_i \ne b_i, b_{i-1}$ for $i=1, \dots, m$ (where here and below we identify the subscripts $0, m+1$ with $m, 1$, respectively) and $m \ge 2$. Let $c_i$ be the third vertex of the triangle $(a_i,b_i,c_i)$, so that $\{a_i,b_i,c_i\}=\{-1,-2,-3\}$. Then $c_i$ lies outside of $F'$ and is not connected to any of the two vertices $0_i, 0_{i-1}$ on $C'$ (as that would create a bead). It follows that all the faces of $H$ which share an edge $(0_i,a_i)$ or $(b_{i-1},0_i)$ with $F'$ are full (and in particular, non-triangular). Now if for some $i=1, \dots, m$, we have $a_i=a_{i-1}$, we can replace the edges $(0_i,a_i)$ and $(0_{i-1},a_{i-1})$ of $C'$ by two chords $(0_i,a_{i-1})$ and $(0_{i-1},a_i)$ which creates an extra $3$-cycle in $G$ without destroying any, a contradiction. It follows that $a_i \ne a_{i-1}$, and by a similar argument, $b_i \ne b_{i-1}$. Then up to permutation, the only possible label sequence along $C'$ is the periodic sequence with period $9$ of the form $(0,-1,-2,0,-3,-1,0,-2,-3)$. Suppose we have $k \ge 2$ repeats of this period along $C'$. Then we can replace the first edges $(0,-1)$ of the first and the second period by two chords thus splitting $F'$ into smaller non-triangular empty faces.
\end{proof}

We make two further observations which will be used in this section.

\begin{remark} \label{rem:trianglesH}
Consider a long cycle $C \subset H$ that covers the $3$-cycle $(0,-i,-j)$ of $K_{1,2,2,2}$. Then by Lemma~\ref{l:Hfaces}\eqref{it:s9inn3}, any vertices labelled $i,j$ or $-k$ lying inside $C$ must be contained in a triangle labelled $(i,j,-k)$ (for $\{i,j,k\} = \{1,2,3\}$).
Moreover, note that for any edge $e$ in the cycle $C$, one of the two faces of $H$ adjacent to $e$ is inside $C$. Thus, for example, if $e$ lies on the boundary of $G'$, then this conclusion applies to the internal face of $H$ adjacent to $e$.
\end{remark}

\begin{remark} \label{rem:twoext} 
  Suppose a simple, oriented curve $\gamma$ with both endpoints in the external domain $F_e$ of $H$ passes through no vertices of $H$ and crosses its edges transversally. Let $i,j,k \in \{0,-1,-2,-3\}$ be pairwise distinct, and denote $a_{ij}$ the number of oriented edges $(i,j)$ of $H$ which $\gamma$ crosses with the positive orientation. Then $(a_{ij}-a_{ji})+(a_{jk}-a_{kj})+(a_{ki}-a_{ik})=0$. To see this, we take the union of $\gamma$ with a curve joining its endpoints and entirely lying in $F_e$ in such a way that the resulting closed curve $\gamma'$ is simple. Then it is easy to see that $a_{ij}-a_{ji}=p_i-p_j$, where $p_i$ is the number of vertices labelled $i$ lying in the interior of $\gamma'$.
  For example, if $\gamma$ crosses only two edges of $H$, then they must be both labelled $(i,j)$ for some $i \ne j, \; i,j \in \{0,-1,-2,-3\}$, and the two crossings must have opposite orientations.
\end{remark}

From Lemma~\ref{l:Hfaces}\eqref{it:no6face} we know that all the non-triangular internal faces of $H$ have length at least $9$ (and divisible by $3$). 
In the case of $9$-faces, Lemma~\ref{l:Hfaces}\eqref{it:above} and~\cite[Proposition~2]{annor25k1222} respectively give us the following facts.

\begin{lem} \label{l:above}
  Any full $9$-face of $H$ contains exactly one $3$-cycle $(1,2,3)$. No two internal $9$-faces of $H$ share a bead.
\end{lem}

Theorem~\ref{th:notall9} is a direct consequence of Lemma~\ref{l:Hfaces} and the following theorem.

\begin{thm} \label{th:notall9'}
  Let $G'$ be a planar semi-cover of $K_{1,2,2,2}$ with the properties listed in Lemma~\ref{l:Hfaces}. Then the subgraph $H$ cannot have all its internal, non-triangular faces of length $9$.
\end{thm}

\begin{proof}
Suppose to the contrary that $H$ has all internal non-triangular faces of length 9. We can assume that $G'$ is chosen in such a way that $H$ is the cover of $K_4$ of the smallest fold (from among all such semi-covers).

By Lemma~\ref{l:above}, all the beads of $H$ (if there are any) lie on its external boundary, that is, each of them shares an edge with the external face, which we denote $F_e$.

\begin{lem} \label{l:no2b}
  No $9$-face of $H$ may share more than one bead with the external face $F_e$.
\end{lem}
\begin{proof}
  Suppose such a face $F'$ exists. Up to relabelling, we have the labelling of vertices as on the left in Figure~\ref{fig:2beads}, where $\{i,j\}=\{1,2\}$ (note that by Lemma~\ref{l:above}, the face $F'$ contains exactly one $3$-cycle $(1,2,3)$, so the labels of the internal vertices of the two beads cannot be the same). The internal vertices $-1$ and $-2$ of the two beads must be connected to the same vertex $3$ lying in $F'$, and then it is not difficult to see that the positions of all the edges and vertices lying $F'$ as in the middle in Figure~\ref{fig:2beads} are forced. Note that we still need edges $(2,0)$ and $(1,-3)$ in $F'$ whose choice is not unique. But in any case, both edges $(-j,i)$ and $(-j,3)$ do not lie in $F'$. We can now join the vertices $0'$ and $3'$ by an edge (creating a bead) and remove all the vertices and edges of $G'$ lying ``above" it, as on the right in Figure~\ref{fig:2beads}. We get a semi-cover with a smaller $H$.

\begin{figure}[h]
\centering
\begin{tikzpicture}[scale=0.7]
\def \r {4}
\def \rv {3}
\def \ger{\r/6}
\begin{scope}
  \coordinate (0c) at (0,\rv); \fill (0c) circle (3pt); \draw (0c) node[shift=({0,0.35})] {0};
  \coordinate (0l) at (-\r,\rv); \fill (0l) circle (3pt); \draw (0l) node[shift=({-0.25,-0.35})] {0};
  \coordinate (-1l) at (-\r/2,3*\rv/2); \fill (-1l) circle (3pt); \draw (-1l) node[right=.05em] {-$3$};
  \coordinate (-3cb) at (-\r/4,0); \fill (-3cb) circle (3pt); \draw (-3cb) node[shift=({0.3,0.3})] {-2};
  \coordinate (-1cb) at (0,-\rv/4); \fill (-1cb) circle (3pt); \draw (-1cb) node[shift=({0,-0.35})] {-$3$};
  \coordinate (-2cb) at (\r/4,0); \fill (-2cb) circle (3pt); \draw (-2cb) node[shift=({-0.3,0.3})] {-$1$};
  \coordinate (0lb) at (-\r/2,-\rv/2); \fill (0lb) circle (3pt); \draw (0lb) node[shift=({0.25,0.35})] {$0'$};
  \coordinate (-1lb) at (-\r,0); \fill (-1lb) circle (3pt); \draw (-1lb) node[shift=({-0.25,0.35})] {-$3'$};
  \coordinate (-3lb) at (-3*\r/4,-\rv/4); \fill (-3lb) circle (3pt); \draw (-3lb) node[shift=({0.25,0.35})] {-$j$};
  \coordinate (-2lb) at (-5*\r/4,-\rv/4); \fill (-2lb) circle (3pt); \draw (-2lb) node[shift=({0,-0.35})] {-$i$};
  \coordinate (-3lt) at (-3*\r/4,3*\rv/2); \fill (-3lt) circle (3pt); \draw (-3lt) node[left=.05em] {-2};
  \coordinate (-2lt) at (-\r/2,5*\rv/4); \fill (-2lt) circle (3pt); \draw (-2lt) node[shift=({0.25,-0.35})] {-1};
  \draw[thick] (0lb) -- (-3lb) -- (-1lb) -- (-2lb) -- (-3lb) (0lb) -- (-1cb) -- (-2cb) ;
  \draw[thick] (-3cb) -- (0c) -- (-1l)  (-2cb) -- (0c) (0l) -- (-1lb) (-1cb) -- (-3cb) -- (-2cb);
  \draw[thick] (-3lt) -- (-1l) -- (-2lt) -- (-3lt) -- (0l) -- (-2lt);
  \draw[thick] (0lb) -- ++(0,-\ger) (-2lb) -- ++(-\ger,-\ger);
  \draw (0,3*\rv/2) node {$F_e$}; 
  \draw (-\r/2,\rv/2) node {$F'$}; 
  \draw (0lb) node[shift=({-0.25,-0.4})] {-$i$};
\end{scope}
\begin{scope}[shift={(2.1*\r,0)}]
  \coordinate (0c) at (0,\rv); \fill (0c) circle (3pt); \draw (0c) node[shift=({0,0.35})] {0};
  \coordinate (0l) at (-\r,\rv); \fill (0l) circle (3pt); \draw (0l) node[shift=({-0.25,-0.35})] {0};
  \coordinate (-1l) at (-\r/2,3*\rv/2); \fill (-1l) circle (3pt); \draw (-1l) node[right=.05em] {-$3$};
  \coordinate (-3cb) at (-\r/4,0); \fill (-3cb) circle (3pt); \draw (-3cb) node[shift=({0.3,0.3})] {-2};
  \coordinate (-1cb) at (0,-\rv/4); \fill (-1cb) circle (3pt); \draw (-1cb) node[shift=({0,-0.35})] {-$3$};
  \coordinate (-2cb) at (\r/4,0); \fill (-2cb) circle (3pt); \draw (-2cb) node[shift=({-0.3,0.3})] {-$1$};
  \coordinate (0lb) at (-\r/2,-\rv/2); \fill (0lb) circle (3pt); \draw (0lb) node[shift=({0.25,0.35})] {$0'$};
  \coordinate (-1lb) at (-\r,0); \fill (-1lb) circle (3pt); \draw (-1lb) node[shift=({-0.25,0.35})] {-$3'$};
  \coordinate (-3lb) at (-3*\r/4,-\rv/4); \fill (-3lb) circle (3pt); \draw (-3lb) node[shift=({0.2,0.3})] {-$j$};
  \coordinate (-2lb) at (-5*\r/4,-\rv/4); \fill (-2lb) circle (3pt); \draw (-2lb) node[shift=({0,-0.35})] {-$i$};
  \coordinate (-3lt) at (-3*\r/4,3*\rv/2); \fill (-3lt) circle (3pt); \draw (-3lt) node[left=.05em] {-2};
  \coordinate (-2lt) at (-\r/2,5*\rv/4); \fill (-2lt) circle (3pt); \draw (-2lt) node[shift=({0.3,0})] {-1};
  \draw[thick] (0lb) -- (-3lb) -- (-1lb) -- (-2lb) -- (-3lb) (0lb) -- (-1cb) -- (-2cb) ;
  \draw[thick] (-3cb) -- (0c) -- (-1l)  (-2cb) -- (0c) (0l) -- (-1lb) (-1cb) -- (-3cb) -- (-2cb);
  \draw[thick] (-3lt) -- (-1l) -- (-2lt) -- (-3lt) -- (0l) -- (-2lt);
  \draw[thick] (0lb) -- ++(0,-\ger) (-2lb) -- ++(-\ger,-\ger);
  \draw (0,3*\rv/2) node {$F_e$}; 
  %
  \draw (0lb) node[shift=({-0.25,-0.4})] {-$i$};
  \coordinate (1in) at (-5*\r/8,3*\rv/8); \fill (1in) circle (3pt); \draw (1in) node[shift=({-0.2,-0.3})] {$1$};
  \coordinate (2in) at (-5*\r/8,5*\rv/8); \fill (2in) circle (3pt); \draw (2in) node[shift=({-0.2,0.3})] {$2$};
  \coordinate (3in) at (-3*\r/8,\rv/2); \fill (3in) circle (3pt); \draw (3in) node[shift=({0.25,-0.25})] {$3$};
  \draw[thick] (1in) -- (2in) -- (3in) -- (1in);
  \draw[thin] (0lb) -- (1in) -- (-3cb) -- (3in) -- (0c) (3in) -- (-2lt) -- (2in) -- (-1lb);
  \draw[thin] (-3lb) -- ++(\ger/4,-\ger) (-3lb) -- ++(-\ger/4,-\ger);
\end{scope}
\begin{scope}[shift={(4.2*\r,0)}]
  \coordinate (0lb) at (-\r/2,-\rv/2); \fill (0lb) circle (3pt); \draw (0lb) node[shift=({0.25,0.35})] {$0'$};
  \coordinate (-1lb) at (-\r,0); \fill (-1lb) circle (3pt); \draw (-1lb) node[shift=({-0.25,0.35})] {-$3'$};
  \coordinate (-3lb) at (-3*\r/4,-\rv/4); \fill (-3lb) circle (3pt); \draw (-3lb) node[shift=({0.25,0.2})] {-$j$};
  \coordinate (-2lb) at (-5*\r/4,-\rv/4); \fill (-2lb) circle (3pt); \draw (-2lb) node[shift=({0,-0.35})] {-$i$};
  \draw[thick] (0lb) -- (-3lb) -- (-1lb) -- (-2lb) -- (-3lb);
  \draw[thick] (0lb) -- ++(0,-\ger) (-2lb) -- ++(-\ger,-\ger);
  \draw (-\r/2,\r/4) node {$F_e$};
  \draw (0lb) node[shift=({-0.25,-0.4})] {-$i$};
  \draw[thin] (-3lb) -- ++(\ger/4,-\ger) (-3lb) -- ++(-\ger/4,-\ger);
  \draw[thick, dashed] (0lb) .. controls (-\r/2,0) .. (-1lb);
\end{scope}
\end{tikzpicture}
\caption{Removing a $9$-face $F'$ sharing two beads with the external face $F_e$.}
\label{fig:2beads}
\end{figure}
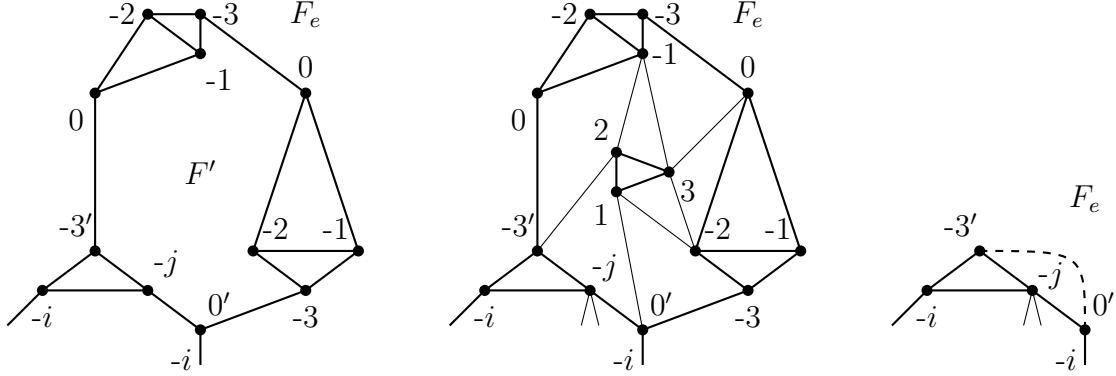

The resulting semi-cover still enjoys the properties from Lemma~\ref{l:Hfaces}. (It is easy to check that property~\eqref{it:s9inn3} is satisfied, in both cases, when $(i,j)=(1,2)$ and $(i,j)=(2,1)$.) We arrive at a contradiction with our choice of $G'$.
\end{proof}

Our starting point is the configuration described in the following lemma.

\begin{lem} \label{l:F1F2}
  There exist two $9$-faces $F_1, F_2$ of $H$, each of which shares a bead with $F_e$, and a vertex $0 \in H$ which belongs to the boundary of all three faces $F_1, F_2, F_e$; see Figure~\ref{fig:F1F2}.
\end{lem}
Without loss of generality, the labels of the three vertices adjacent to the vertex $0$ shared by $F_1, F_2$ and $F_e$ can be assigned as in Figure~\ref{fig:F1F2}, and then some other labels shown in Figure~\ref{fig:F1F2} are forced.
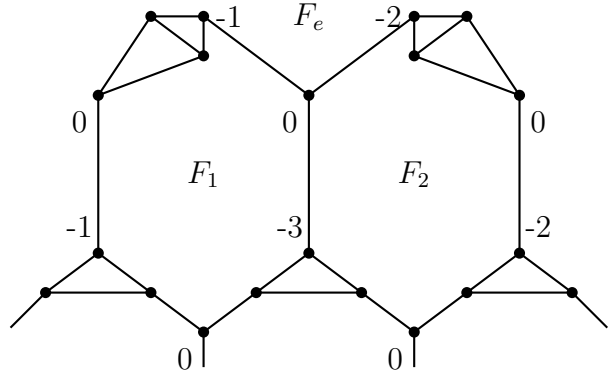
\begin{figure}[h]
\centering
\begin{tikzpicture}[scale=0.7]
\def \r {4}
\def \rv {3}
\def \ger{\r/6}
  \coordinate (0c) at (0,\rv); \fill (0c) circle (3pt); \draw (0c) node[shift=({-0.25,-0.35})] {0};
  \coordinate (0l) at (-\r,\rv); \fill (0l) circle (3pt); \draw (0l) node[shift=({-0.25,-0.35})] {0};
  \coordinate (0r) at (\r,\rv); \fill (0r) circle (3pt); \draw (0r) node[shift=({0.25,-0.35})] {0};
  \coordinate (-1l) at (-\r/2,3*\rv/2); \fill (-1l) circle (3pt); \draw (-1l) node[right=.06em] {-$1$};
  \coordinate (-2r) at (\r/2,3*\rv/2); \fill (-2r) circle (3pt); \draw (-2r) node[left=.06em] {-$2$};
  \coordinate (-3cb) at (0,0); \fill (-3cb) circle (3pt); \draw (-3cb) node[shift=({-0.25,0.35})] {-3};
  \coordinate (-1cb) at (-\r/4,-\rv/4); \fill (-1cb) circle (3pt); 
  \coordinate (-2cb) at (\r/4,-\rv/4); \fill (-2cb) circle (3pt); 
  \coordinate (0lb) at (-\r/2,-\rv/2); \fill (0lb) circle (3pt); \draw (0lb) node[shift=({-0.25,-0.35})] {0};
  \coordinate (-1lb) at (-\r,0); \fill (-1lb) circle (3pt); \draw (-1lb) node[shift=({-0.25,0.35})] {-$1$};
  \coordinate (-3lb) at (-3*\r/4,-\rv/4); \fill (-3lb) circle (3pt); 
  \coordinate (-2lb) at (-5*\r/4,-\rv/4); \fill (-2lb) circle (3pt); 
  \coordinate (0rb) at (\r/2,-\rv/2); \fill (0rb) circle (3pt); \draw (0rb) node[shift=({-0.25,-0.35})] {0};
  \coordinate (-2rb) at (\r,0); \fill (-2rb) circle (3pt); \draw (-2rb) node[shift=({0.25,0.35})] {-$2$};
  \coordinate (-1rb) at (3*\r/4,-\rv/4); \fill (-1rb) circle (3pt); 
  \coordinate (-3rb) at (5*\r/4,-\rv/4); \fill (-3rb) circle (3pt); 
  \coordinate (-3rt) at (3*\r/4,3*\rv/2); \fill (-3rt) circle (3pt); 
  \coordinate (-1rt) at (\r/2,5*\rv/4); \fill (-1rt) circle (3pt); 
  \coordinate (-3lt) at (-3*\r/4,3*\rv/2); \fill (-3lt) circle (3pt); 
  \coordinate (-2lt) at (-\r/2,5*\rv/4); \fill (-2lt) circle (3pt); 
  \draw[thick] (0rb) -- (-1rb) -- (-3rb) -- (-2rb) -- (-1rb) (0lb) -- (-3lb) -- (-1lb) -- (-2lb) -- (-3lb) (0lb) -- (-1cb) -- (-2cb) -- (0rb);
  \draw[thick] (-3cb) -- (0c) -- (-1l)  (0c) -- (-2r) (0l) -- (-1lb) (0r) -- (-2rb) (-1cb) -- (-3cb) -- (-2cb);
  \draw[thick] (-3rt) -- (-2r) -- (-1rt) -- (-3rt) -- (0r) -- (-1rt) (-3lt) -- (-1l) -- (-2lt) -- (-3lt) -- (0l) -- (-2lt);
  \draw[thick] (0lb) -- ++(0,-\ger) (0rb) -- ++(0,-\ger) (-2lb) -- ++(-\ger,-\ger) (-3rb) -- ++(\ger,-\ger);
  \draw (0,3*\rv/2) node {$F_e$}; 
  \draw (-\r/2,\rv/2) node {$F_1$}; \draw (\r/2,\rv/2) node {$F_2$};
\end{tikzpicture}
\caption{$9$-faces $F_1$ and $F_2$.}
\label{fig:F1F2}
\end{figure}

{
\begin{proof}
  Suppose $H$ is a $p$-fold cover of $K_4$ and that the external boundary $C_e$ of $H$ has length $3m, \; m > 1$. Let $\mathbf{f}_j, \, j \in \{3,9\}$, be the numbers of the internal $j$-faces of $H$, and let $\mathbf{b}$ be the number of the beads; note that $\mathbf{f}_3=\mathbf{b}+p$. From Euler's formula we find $\mathbf{f}_3+\mathbf{f}_9=2p+1$ and also $\mathbf{f}_3+3\mathbf{f}_9=4p-m$, from which $\mathbf{b}+\mathbf{f}_9=p+1$ and $\mathbf{b}+3\mathbf{f}_9=3p-m$. Hence $m=2l+1 \, (l \ge 1)$ is odd and $\mathbf{b}=l+2,\; \mathbf{f}_9=p-l-1$.

  We know that all the beads lie on the external boundary $C_e$ of $H$ (which is a cycle of length $3m=6l+3$). There are $m=2l+1$ edges on $C_e$ whose both endpoints are not $0$. As each bead contains exactly one such edge, $l+2$ of them belong to beads. It follows that there is a path $(0,i_1,i_2,0,i_3,i_4,0)$ on $C_e$ such that $i_j \in \{-1,-2,-3\}$ and each edge $(i_1,i_2)$ and $(i_3,i_4)$ lies on a bead. The vertex $0$ between them cannot belong to any of these beads because otherwise we get an internal face sharing two beads with the external boundary, in contradiction with Lemma~\ref{l:no2b}. Therefore we obtain the configuration as in Figure~\ref{fig:F1F2} (up to relabelling).
\end{proof}
}

The rest of the proof is ``local". We show that starting from the configuration in Figure~\ref{fig:F1F2}, the part of the semi-cover $G'$ satisfying the conditions of the proposition and lying in the $9$-faces $F_1$ and $F_2$ and attached to them is almost unique, and then this leads to a contradiction after attaching no more than five extra internal $9$-faces. We note that a $9$-face contains no more than one $3$-cycle $(1,2,3)$ in it; in all the cases below there is exactly one. Moreover, Property~\ref{l:Hfaces}\eqref{it:s9inn3} (and Remark~\ref{rem:trianglesH}) together with the fact that $G'$ is a semi-cover, forces most of the labels and the positions of the edges to be unique. In what follows we will sometimes omit the details of these leaving the (easy) verification to the reader.

\medskip

The face $F_1$ shares the edge(s) $(0,-1)$ with the external face, so by Remark~\ref{rem:trianglesH}, $F_1$ contains the $3$-cycles $(-3,1,2)$ and $(-2,1,3)$. Similarly, $F_2$ contains the $3$-cycles $(-3,1,2)$ and $(-1,2,3)$. It follows that the vertex $-3$ on the common boundary of $F_1$ and $F_2$ is the vertex of a $3$-cycle $(-3,1,2)$ entirely lying in either $F_1$ or $F_2$. Without loss of generality, we can assume that it lies in $F_2$. Then the vertex of the bead lying on the boundary of $F_2$ cannot be another $-3$, and so it is a $-1$. This gives us several forced choices for the vertices and edges of $G'$ lying in the closure of $F_2$ leading to the drawing on the left in Figure~\ref{fig:F1F2with}. Note that there is more than one choice for the two missing edges $(3,0)$ and $(1,-2)$; also, in $F_1$, the position of the $3$-cycle $(1,2,3)$ and its attachment to the bead is unique, but we do not know the labels yet.

We consider two cases depending on whether the label of the vertex of the bead lying on the boundary of $F_1$ is $-3$ or $-2$.

\medskip

First suppose that the vertex of the bead lying on the boundary of $F_1$ is labelled $-3$. Using the fact that $F_1$ contains the $3$-cycle $(-2,1,3)$, one sees that all the labels and the edges of $G'$ lying in the closure of $F_1$ are forced, as also are the ``missing" labels and edges in $F_2$. This gives the drawing on the right in Figure~\ref{fig:F1F2with}.

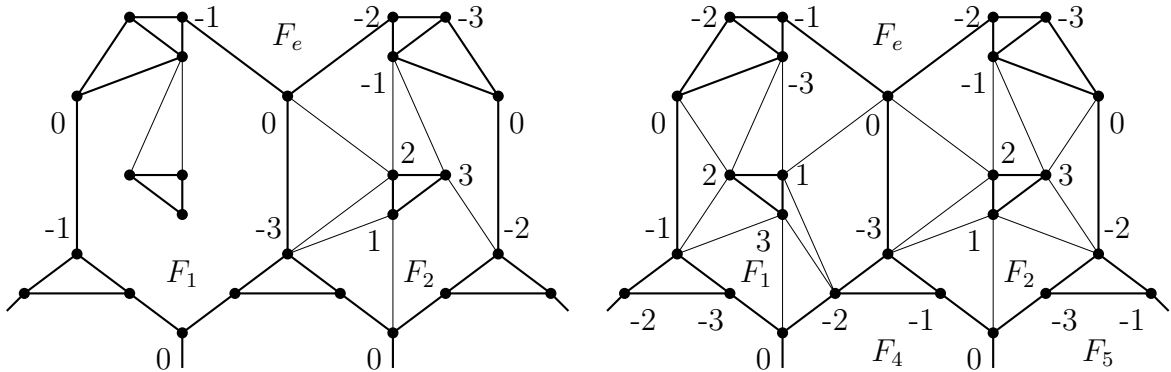
\begin{figure}[h]
\centering
\begin{tikzpicture}[scale=0.65]
\def \r {4}
\def \rv {3}
\def \ger{\r/6}
\begin{scope}
  \coordinate (0c) at (0,\rv); \fill (0c) circle (3pt); \draw (0c) node[shift=({-0.25,-0.35})] {0};
  \coordinate (0l) at (-\r,\rv); \fill (0l) circle (3pt); \draw (0l) node[shift=({-0.25,-0.35})] {0};
  \coordinate (0r) at (\r,\rv); \fill (0r) circle (3pt); \draw (0r) node[shift=({0.25,-0.35})] {0};
  \coordinate (-1l) at (-\r/2,3*\rv/2); \fill (-1l) circle (3pt); \draw (-1l) node[right=.05em] {-$1$};
  \coordinate (-2r) at (\r/2,3*\rv/2); \fill (-2r) circle (3pt); \draw (-2r) node[left=.06em] {-$2$};
  \coordinate (-3cb) at (0,0); \fill (-3cb) circle (3pt); \draw (-3cb) node[shift=({-0.25,0.35})] {-3};
  \coordinate (-1cb) at (-\r/4,-\rv/4); \fill (-1cb) circle (3pt); 
  \coordinate (-2cb) at (\r/4,-\rv/4); \fill (-2cb) circle (3pt); 
  \coordinate (0lb) at (-\r/2,-\rv/2); \fill (0lb) circle (3pt); \draw (0lb) node[shift=({-0.25,-0.35})] {0};
  \coordinate (-1lb) at (-\r,0); \fill (-1lb) circle (3pt); \draw (-1lb) node[shift=({-0.25,0.35})] {-$1$};
  \coordinate (-3lb) at (-3*\r/4,-\rv/4); \fill (-3lb) circle (3pt); 
  \coordinate (-2lb) at (-5*\r/4,-\rv/4); \fill (-2lb) circle (3pt); 
  \coordinate (0rb) at (\r/2,-\rv/2); \fill (0rb) circle (3pt); \draw (0rb) node[shift=({-0.25,-0.35})] {0};
  \coordinate (-2rb) at (\r,0); \fill (-2rb) circle (3pt); \draw (-2rb) node[shift=({0.25,0.35})] {-$2$};
  \coordinate (-1rb) at (3*\r/4,-\rv/4); \fill (-1rb) circle (3pt); 
  \coordinate (-3rb) at (5*\r/4,-\rv/4); \fill (-3rb) circle (3pt); 
  \coordinate (-3rt) at (3*\r/4,3*\rv/2); \fill (-3rt) circle (3pt); \draw (-3rt) node[right=.05em] {-$3$};
  \coordinate (-1rt) at (\r/2,5*\rv/4); \fill (-1rt) circle (3pt); \draw (-1rt) node[shift=({-0.25,-0.35})] {-$1$};
  \coordinate (-3lt) at (-3*\r/4,3*\rv/2); \fill (-3lt) circle (3pt); 
  \coordinate (-2lt) at (-\r/2,5*\rv/4); \fill (-2lt) circle (3pt); 
  \draw[thick] (0rb) -- (-1rb) -- (-3rb) -- (-2rb) -- (-1rb) (0lb) -- (-3lb) -- (-1lb) -- (-2lb) -- (-3lb) (0lb) -- (-1cb) -- (-2cb) -- (0rb);
  \draw[thick] (-3cb) -- (0c) -- (-1l)  (0c) -- (-2r) (0l) -- (-1lb) (0r) -- (-2rb) (-1cb) -- (-3cb) -- (-2cb);
  \draw[thick] (-3rt) -- (-2r) -- (-1rt) -- (-3rt) -- (0r) -- (-1rt) (-3lt) -- (-1l) -- (-2lt) -- (-3lt) -- (0l) -- (-2lt);
  \draw[thick] (0lb) -- ++(0,-\ger) (0rb) -- ++(0,-\ger) (-2lb) -- ++(-\ger/2,-\ger/2) (-3rb) -- ++(\ger/2,-\ger/2);
  \draw (0,11*\rv/8) node {$F_e$}; 
  \draw (-\r/2,-\rv/8) node {$F_1$}; \draw (5*\r/8,-\rv/8) node {$F_2$};
  \coordinate (3r) at (3*\r/4,\rv/2); \fill (3r) circle (3pt); \draw (3r) node[right=.05em] {$3$};
  \coordinate (1r) at (\r/2,\rv/4); \fill (1r) circle (3pt); \draw (1r) node[shift=({-0.25,-0.35})] {$1$};
  \coordinate (2r) at (\r/2,\rv/2); \fill (2r) circle (3pt); \draw (2r) node[shift=({0.2,0.3})] {$2$};
  \draw[thick] (3r) -- (2r) -- (1r) -- (3r);
  \draw[ultra thin] (-2rb) -- (3r) -- (-1rt) -- (2r) -- (-3cb) -- (1r) -- (0rb) (2r) -- (0c);
  \coordinate (1l) at (-\r/2,\rv/2); \fill (1l) circle (3pt); 
  \coordinate (3l) at (-3*\r/4,\rv/2); \fill (3l) circle (3pt); 
  \coordinate (2l) at (-\r/2,\rv/4); \fill (2l) circle (3pt); 
  \draw[thick] (3l) -- (2l) -- (1l) -- (3l);
  \draw[ultra thin] (1l) -- (-2lt) -- (3l);
\end{scope}
\begin{scope}[shift={(2.85*\r,0)}]
  \coordinate (0c) at (0,\rv); \fill (0c) circle (3pt); \draw (0c) node[shift=({-0.2,-0.38})] {0};
  \coordinate (0l) at (-\r,\rv); \fill (0l) circle (3pt); \draw (0l) node[shift=({-0.25,-0.35})] {0};
  \coordinate (0r) at (\r,\rv); \fill (0r) circle (3pt); \draw (0r) node[shift=({0.25,-0.35})] {0};
  \coordinate (-1l) at (-\r/2,3*\rv/2); \fill (-1l) circle (3pt); \draw (-1l) node[right=.05em] {-$1$};
  \coordinate (-2r) at (\r/2,3*\rv/2); \fill (-2r) circle (3pt); \draw (-2r) node[left=.06em] {-$2$};
  \coordinate (-3cb) at (0,0); \fill (-3cb) circle (3pt); \draw (-3cb) node[shift=({-0.25,0.35})] {-3};
  \coordinate (-1cb) at (-\r/4,-\rv/4); \fill (-1cb) circle (3pt); \draw (-1cb) node[shift=({0,-0.35})] {-$2$};
  \coordinate (-2cb) at (\r/4,-\rv/4); \fill (-2cb) circle (3pt); \draw (-2cb) node[shift=({-0.25,-0.35})] {-$1$};
  \coordinate (0lb) at (-\r/2,-\rv/2); \fill (0lb) circle (3pt); \draw (0lb) node[shift=({-0.25,-0.35})] {0};
  \coordinate (-1lb) at (-\r,0); \fill (-1lb) circle (3pt); \draw (-1lb) node[shift=({-0.25,0.35})] {-$1$};
  \coordinate (-3lb) at (-3*\r/4,-\rv/4); \fill (-3lb) circle (3pt); \draw (-3lb) node[shift=({-0.25,-0.35})] {-$3$};
  \coordinate (-2lb) at (-5*\r/4,-\rv/4); \fill (-2lb) circle (3pt); \draw (-2lb) node[shift=({0.25,-0.35})] {-$2$};
  \coordinate (0rb) at (\r/2,-\rv/2); \fill (0rb) circle (3pt); \draw (0rb) node[shift=({-0.25,-0.35})] {0};
  \coordinate (-2rb) at (\r,0); \fill (-2rb) circle (3pt); \draw (-2rb) node[shift=({0.25,0.35})] {-$2$};
  \coordinate (-1rb) at (3*\r/4,-\rv/4); \fill (-1rb) circle (3pt); \draw (-1rb) node[shift=({0.25,-0.35})] {-$3$};
  \coordinate (-3rb) at (5*\r/4,-\rv/4); \fill (-3rb) circle (3pt); \draw (-3rb) node[shift=({-0.25,-0.35})] {-$1$};
  \coordinate (-3rt) at (3*\r/4,3*\rv/2); \fill (-3rt) circle (3pt); \draw (-3rt) node[right=.05em] {-$3$};
  \coordinate (-1rt) at (\r/2,5*\rv/4); \fill (-1rt) circle (3pt); \draw (-1rt) node[shift=({-0.25,-0.35})] {-$1$};
  \coordinate (-3lt) at (-3*\r/4,3*\rv/2); \fill (-3lt) circle (3pt); \draw (-3lt) node[left=.05em] {-$2$};
  \coordinate (-2lt) at (-\r/2,5*\rv/4); \fill (-2lt) circle (3pt); \draw (-2lt) node[shift=({0.25,-0.35})] {-$3$};
  \draw[thick] (0rb) -- (-1rb) -- (-3rb) -- (-2rb) -- (-1rb) (0lb) -- (-3lb) -- (-1lb) -- (-2lb) -- (-3lb) (0lb) -- (-1cb) -- (-2cb) -- (0rb);
  \draw[thick] (-3cb) -- (0c) -- (-1l)  (0c) -- (-2r) (0l) -- (-1lb) (0r) -- (-2rb) (-1cb) -- (-3cb) -- (-2cb);
  \draw[thick] (-3rt) -- (-2r) -- (-1rt) -- (-3rt) -- (0r) -- (-1rt) (-3lt) -- (-1l) -- (-2lt) -- (-3lt) -- (0l) -- (-2lt);
  \draw[thick] (0lb) -- ++(0,-\ger) (0rb) -- ++(0,-\ger) (-2lb) -- ++(-\ger/2,-\ger/2) (-3rb) -- ++(\ger/2,-\ger/2);
  \draw (0,11*\rv/8) node {$F_e$}; 
  \draw (-5*\r/8,-\rv/8) node {$F_1$}; \draw (5*\r/8,-\rv/8) node {$F_2$};
  \draw (0,-5*\rv/8) node {$F_4$}; \draw (\r,-5*\rv/8) node {$F_5$};
  \coordinate (3r) at (3*\r/4,\rv/2); \fill (3r) circle (3pt); \draw (3r) node[right=.05em] {$3$};
  \coordinate (1r) at (\r/2,\rv/4); \fill (1r) circle (3pt); \draw (1r) node[shift=({-0.25,-0.35})] {$1$};
  \coordinate (2r) at (\r/2,\rv/2); \fill (2r) circle (3pt); \draw (2r) node[shift=({0.2,0.3})] {$2$};
  \draw[thick] (3r) -- (2r) -- (1r) -- (3r);
  \draw[ultra thin] (-2rb) -- (3r) -- (-1rt) -- (2r) -- (-3cb) -- (1r) -- (0rb) (2r) -- (0c);
  \coordinate (1l) at (-\r/2,\rv/2); \fill (1l) circle (3pt); \draw (1l) node[right=.05em] {$1$};
  \coordinate (3l) at (-3*\r/4,\rv/2); \fill (3l) circle (3pt); \draw (3l) node[left=.05em] {$2$};
  \coordinate (2l) at (-\r/2,\rv/4); \fill (2l) circle (3pt); \draw (2l) node[shift=({-0.25,-0.35})] {$3$};
  \draw[thick] (3l) -- (2l) -- (1l) -- (3l);
  \draw[ultra thin] (1l) -- (-2lt) -- (3l) (0c) -- (1l) -- (-1cb) -- (2l) -- (0lb) (2l) -- (-1lb) -- (3l) -- (0l) (1r) -- (-2rb) (3r) -- (0r);
\end{scope}
\end{tikzpicture}
\caption{$9$-faces $F_1$ and $F_2$ with parts of the semi-cover $G'$ inside.}
\label{fig:F1F2with}
\end{figure}

Consider the face $F_5$ of $H$ sharing the edge $(0,-3)$ with $F_2$ to the bottom-right of the right drawing in Figure~\ref{fig:F1F2with}. We cannot have $F_5=F_e$ since $F_5$ and $F_e$ share a single edge of $H$ (the edge $(0,-1)$; see Remark~\ref{rem:twoext}), and $F_5$ cannot be triangular. So $F_5$ is an internal $9$-face. As it shares an edge $(0,-1)$ with $F_e$,  Remark~\ref{rem:trianglesH} implies that it contains $3$-cycles $(-3,1,2)$ and $(-2,1,3)$. Now consider the face $F_4$ of $H$ sharing the edge $(0,-2)$ with $F_1$ and the edge $(0,-1)$ with $F_2$ (at the ``bottom-centre" on the right in Figure~\ref{fig:F1F2with}. We cannot have $F_4=F_e$ since $F_4$ and $F_e$ can be connected with an arc which only crosses the edges $(0,-1)$ and $(0,-2)$ of $H$ (see Remark~\ref{rem:twoext}), and $F_4$ cannot be a triangular (as it has at least two vertices $0$ on its boundary). So $F_4$ is an internal $9$-face. Moreover, as we can easily see from Figure~\ref{fig:F1F2with} (on the right), $F_4$ lies in the domains bounded by cycles with the vertices $(0,-1,-3)$ and $(0,-2,-3)$ of $H$ (for there is an arc from $F_4$ to $F_e$ which crosses only one edge in each of these cycles). Considering the cycles of $H$ with the vertices $(0,-1,-2)$, there is an arc from $F_4$ to $F_e$ (essentially the same one as before) which crosses two edges of this cycle(s), $(0,-2)$ and $(-1,0)$. But because the crossings have opposite orientations, we deduce that $F_4$ also lies in the domains bounded by a cycle with the vertices $(0,-1,-2)$. So by property~\ref{l:Hfaces}\eqref{it:s9inn3}, the face $F_4$ contains all the $3$-cycles $(-3,1,2), (-1,2,3)$ and $(-2,1,3)$ (one consequence of this is the fact that the vertices of the $3$-cycle $(1,2,3)$ lying in $F_4$ must be connected to three different vertices $0$).

\begin{figure}[h]
\centering
\begin{tikzpicture}[scale=0.65]
\def \r {4}
\def \rv {3}
\def \ger{\r/6}
\begin{scope}
  \coordinate (0c) at (0,\rv); \fill (0c) circle (3pt); \draw (0c) node[shift=({-0.25,-0.35})] {0};
  \coordinate (0l) at (-\r,\rv); \fill (0l) circle (3pt); \draw (0l) node[shift=({-0.25,-0.35})] {0};
  \coordinate (0r) at (\r,\rv); \fill (0r) circle (3pt); \draw (0r) node[shift=({0.25,-0.35})] {0};
  \coordinate (-1l) at (-\r/2,3*\rv/2); \fill (-1l) circle (3pt); \draw (-1l) node[right=.05em] {-$3$};
  \coordinate (-2r) at (\r/2,3*\rv/2); \fill (-2r) circle (3pt); \draw (-2r) node[left=.06em] {-$2$};
  \coordinate (-3cb) at (0,0); \fill (-3cb) circle (3pt); \draw (-3cb) node[shift=({-0.25,0.35})] {-2};
  \coordinate (-1cb) at (-\r/4,-\rv/4); \fill (-1cb) circle (3pt); 
  \coordinate (-2cb) at (\r/4,-\rv/4); \fill (-2cb) circle (3pt); 
  \coordinate (0lb) at (-\r/2,-\rv/2); \fill (0lb) circle (3pt); \draw (0lb) node[shift=({-0.25,-0.25})] {0};
  \coordinate (-1lb) at (-\r,0); \fill (-1lb) circle (3pt); \draw (-1lb) node[shift=({-0.25,0.35})] {-$1$};
  \coordinate (-3lb) at (-3*\r/4,-\rv/4); \fill (-3lb) circle (3pt); 
  \coordinate (-2lb) at (-5*\r/4,-\rv/4); \fill (-2lb) circle (3pt); 
  \coordinate (0rb) at (\r/2,-\rv/2); \fill (0rb) circle (3pt); \draw (0rb) node[shift=({-0.25,-0.25})] {0};
  \coordinate (-2rb) at (\r,0); \fill (-2rb) circle (3pt); 
  \coordinate (-1rb) at (3*\r/4,-\rv/4); \fill (-1rb) circle (3pt); 
  \coordinate (-3rb) at (5*\r/4,-\rv/4); \fill (-3rb) circle (3pt); 
  \coordinate (-3rt) at (\r/4,5*\rv/4); \fill (-3rt) circle (3pt); \draw (-3rt) node[shift=({-0.2,-0.35})] {-$3$};
  \coordinate (-1rt) at (3*\r/4,5*\rv/4); \fill (-1rt) circle (3pt); \draw (-1rt) node[shift=({0.25,0.25})] {-$1$};
  \coordinate (-3lt) at (-\r/4,5*\rv/4); \fill (-3lt) circle (3pt); \draw (-3lt) node[shift=({0.25,0.25})] {-$1$};
  \coordinate (-2lt) at (-3*\r/4,5*\rv/4); \fill (-2lt) circle (3pt); \draw (-2lt) node[shift=({0,-0.35})] {-$2$};
  \draw[thick] (0rb) -- (-1rb) -- (-3rb) -- (-2rb) -- (-1rb) (0lb) -- (-3lb) -- (-1lb) -- (-2lb) -- (-3lb) (0lb) -- (-1cb) -- (-2cb) -- (0rb);
  \draw[thick] (-3cb) -- (0c) -- (-1l)  (0c) -- (-2r) (0l) -- (-1lb) (0r) -- (-2rb) (-1cb) -- (-3cb) -- (-2cb);
  \draw[thick] (-3rt) -- (-2r) -- (-1rt) -- (-3rt)  (0r) -- (-1rt) (-3lt) -- (-1l) -- (-2lt) -- (-3lt) (0l) -- (-2lt);
  \draw[thick] (0lb) -- ++(0,-2*\ger/3) (0rb) -- ++(0,-2*\ger/3) (-2lb) -- ++(-\ger/2,-\ger/2) (-3rb) -- ++(\ger/2,-\ger/2);
  \draw[thick] (-1l) -- ++(0,\ger/2) (-2r) -- ++(0,\ger/2) (0l) -- ++(-\ger,\ger) (0r) -- ++(\ger,\ger);
  \draw (\r,3*\rv/2) node {$F_e$}; 
  \draw (-\r/2,-\rv/8) node {$F_4$}; \draw (\r/2,-\rv/8) node {$F_5$};
  \draw (\r/8,3*\rv/2) node {$F_2$}; \draw (-\r,3*\rv/2) node {$F_1$};
  \coordinate (3F1) at ($(-2lt)+(-\ger/4,2*\ger/3)$); \draw (3F1) node[shift=({0,0.2})] {$3$};
  \coordinate (1F1) at ($(-2lt)+(\ger/4,2*\ger/3)$); \draw (1F1) node[shift=({0,0.2})] {$1$};
  \draw[ultra thin] (3F1) -- (-2lt) -- (1F1);
  \coordinate (30F1) at ($(0l)+(0,2*\ger/3)$); \draw (30F1) node[shift=({0,0.2})] {$3$};
  \coordinate (1F2) at ($(0c)+(0,2*\ger/3)$); \draw (1F2) node[shift=({0,0.2})] {$1$};
  \draw[ultra thin] (30F1) -- (0l) (0c)-- (1F2);
  \coordinate (3r) at (3*\r/4,\rv/2); \fill (3r) circle (3pt); 
  \coordinate (1r) at (\r/2,\rv/4); \fill (1r) circle (3pt); \draw (1r) node[shift=({-0.25,-0.3})] {$3$};
  \coordinate (2r) at (\r/2,\rv/2); \fill (2r) circle (3pt); 
  \draw[thick] (3r) -- (2r) -- (1r) -- (3r);
  \draw[ultra thin] (3r) -- (-3rt) -- (2r);
  \coordinate (1l) at (-\r/2,\rv/2); \fill (1l) circle (3pt); 
  \coordinate (3l) at (-3*\r/4,\rv/2); \fill (3l) circle (3pt); 
  \coordinate (2l) at (-\r/2,\rv/4); \fill (2l) circle (3pt); \draw (2l) node[shift=({-0.25,-0.3})] {$1$};
  \draw[thick] (3l) -- (2l) -- (1l) -- (3l);
  \draw[ultra thin] (1l) -- (-3lt) -- (3l);
\end{scope}
\begin{scope}[shift={(2.85*\r,0)}]
  \coordinate (0c) at (0,\rv); \fill (0c) circle (3pt); \draw (0c) node[shift=({-0.25,-0.35})] {0};
  \coordinate (0l) at (-\r,\rv); \fill (0l) circle (3pt); \draw (0l) node[shift=({-0.25,-0.35})] {0};
  \coordinate (0r) at (\r,\rv); \fill (0r) circle (3pt); \draw (0r) node[shift=({0.25,-0.35})] {0};
  \coordinate (-1l) at (-\r/2,3*\rv/2); \fill (-1l) circle (3pt); \draw (-1l) node[right=.05em] {-$3$};
  \coordinate (-2r) at (\r/2,3*\rv/2); \fill (-2r) circle (3pt); \draw (-2r) node[left=.06em] {-$2$};
  \coordinate (-3cb) at (0,0); \fill (-3cb) circle (3pt); \draw (-3cb) node[shift=({0.25,0.35})] {-2};
  \coordinate (-1cb) at (-\r/4,-\rv/4); \fill (-1cb) circle (3pt); \draw (-1cb) node[shift=({0,-0.35})] {-$1$};
  \coordinate (-2cb) at (\r/4,-\rv/4); \fill (-2cb) circle (3pt); \draw (-2cb) node[shift=({-0.25,-0.35})] {-$3$};
  \coordinate (0lb) at (-\r/2,-\rv/2); \fill (0lb) circle (3pt); \draw (0lb) node[shift=({-0.25,-0.25})] {0};
  \coordinate (-1lb) at (-\r,0); \fill (-1lb) circle (3pt); \draw (-1lb) node[shift=({-0.25,0.35})] {-$1$};
  \coordinate (-3lb) at (-3*\r/4,-\rv/4); \fill (-3lb) circle (3pt); \draw (-3lb) node[shift=({-0.25,-0.35})] {-$3$};
  \coordinate (-2lb) at (-5*\r/4,-\rv/4); \fill (-2lb) circle (3pt); 
  \coordinate (0rb) at (\r/2,-\rv/2); \fill (0rb) circle (3pt); \draw (0rb) node[shift=({-0.25,-0.25})] {0};
  \coordinate (-2rb) at (\r,0); \fill (-2rb) circle (3pt); 
  \coordinate (-1rb) at (3*\r/4,-\rv/4); \fill (-1rb) circle (3pt); 
  \coordinate (-3rb) at (5*\r/4,-\rv/4); \fill (-3rb) circle (3pt); 
  \coordinate (-3rt) at (\r/4,5*\rv/4); \fill (-3rt) circle (3pt); \draw (-3rt) node[shift=({-0.2,-0.35})] {-$3$};
  \coordinate (-1rt) at (3*\r/4,5*\rv/4); \fill (-1rt) circle (3pt); \draw (-1rt) node[shift=({0.25,0.25})] {-$1$};
  \coordinate (-3lt) at (-\r/4,5*\rv/4); \fill (-3lt) circle (3pt); \draw (-3lt) node[shift=({0.25,0.25})] {-$1$};
  \coordinate (-2lt) at (-3*\r/4,5*\rv/4); \fill (-2lt) circle (3pt); \draw (-2lt) node[shift=({0,-0.35})] {-$2$};
  \draw[thick] (0rb) -- (-1rb) -- (-3rb) -- (-2rb) -- (-1rb) (0lb) -- (-3lb) -- (-1lb) -- (-2lb) -- (-3lb) (0lb) -- (-1cb) -- (-2cb) -- (0rb);
  \draw[thick] (-3cb) -- (0c) -- (-1l)  (0c) -- (-2r) (0l) -- (-1lb) (0r) -- (-2rb) (-1cb) -- (-3cb) -- (-2cb);
  \draw[thick] (-3rt) -- (-2r) -- (-1rt) -- (-3rt)  (0r) -- (-1rt) (-3lt) -- (-1l) -- (-2lt) -- (-3lt) (0l) -- (-2lt);
  \draw[thick] (0lb) -- ++(0,-2*\ger/3) (0rb) -- ++(0,-2*\ger/3) (-2lb) -- ++(-\ger/2,-\ger/2) (-3rb) -- ++(\ger/2,-\ger/2);
  \draw[thick] (-1l) -- ++(0,\ger/2) (-2r) -- ++(0,\ger/2) (0l) -- ++(-\ger,\ger) (0r) -- ++(\ger,\ger);
  \draw (\r,3*\rv/2) node {$F_e$}; 
  \draw (\r/2,-\rv/8) node {$F_5$};
  \draw (\r/8,3*\rv/2) node {$F_2$}; \draw (-\r,3*\rv/2) node {$F_1$};
  \coordinate (3F1) at ($(-2lt)+(-\ger/4,2*\ger/3)$); \draw (3F1) node[shift=({0,0.2})] {$3$};
  \coordinate (1F1) at ($(-2lt)+(\ger/4,2*\ger/3)$); \draw (1F1) node[shift=({0,0.2})] {$1$};
  \draw[ultra thin] (3F1) -- (-2lt) -- (1F1);
  \coordinate (30F1) at ($(0l)+(0,2*\ger/3)$); \draw (30F1) node[shift=({0,0.2})] {$3$};
  \coordinate (1F2) at ($(0c)+(0,2*\ger/3)$); \draw (1F2) node[shift=({0,0.2})] {$1$};
  \draw[ultra thin] (30F1) -- (0l) (0c)-- (1F2);
  \coordinate (3r) at (3*\r/4,\rv/2); \fill (3r) circle (3pt); 
  \coordinate (1r) at (\r/2,\rv/4); \fill (1r) circle (3pt); \draw (1r) node[shift=({-0.25,-0.3})] {$3$};
  \coordinate (2r) at (\r/2,\rv/2); \fill (2r) circle (3pt); 
  \draw[thick] (3r) -- (2r) -- (1r) -- (3r);
  \draw[ultra thin] (3r) -- (-3rt) -- (2r);
  \coordinate (1l) at (-\r/2,\rv/2); \fill (1l) circle (3pt); \draw (1l) node[right=.06em] {$3$};
  \coordinate (3l) at (-3*\r/4,\rv/2); \fill (3l) circle (3pt); \draw (3l) node[left=.05em] {$2$};
  \coordinate (2l) at (-\r/2,\rv/4); \fill (2l) circle (3pt); \draw (2l) node[shift=({0.25,-0.3})] {$1$};
  \draw[thick] (3l) -- (2l) -- (1l) -- (3l);
  \draw[ultra thin] (1l) -- (-3lt) -- (3l) (0l) -- (3l) -- (-3lb) -- (2l) -- (0lb) (2l) -- (-3cb) -- (1l) -- (0c);
\end{scope}
\end{tikzpicture}
\caption{$9$-faces $F_4$ and $F_5$.}
\label{fig:F4F5}
\end{figure}
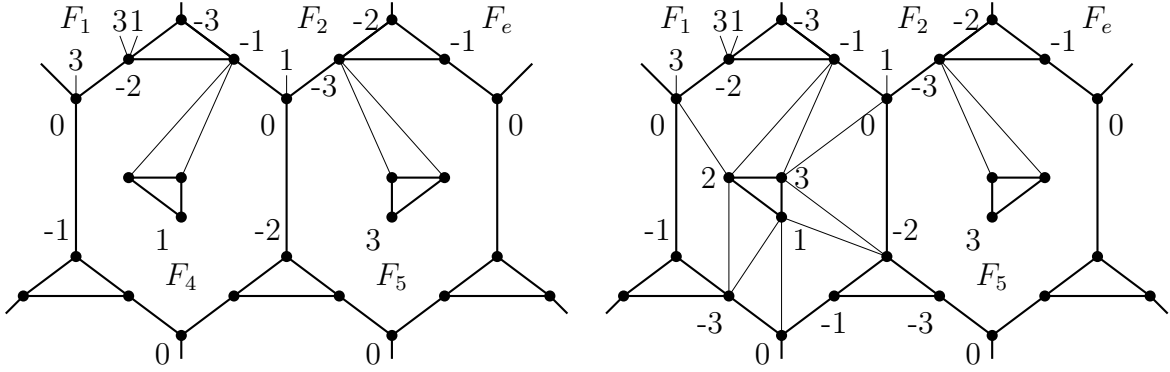

On the left in Figure~\ref{fig:F4F5}, we have the drawing of the faces $F_4$ and $F_5$, together with some labels and edges which are forced from the drawings of $F_1$ and $F_2$ (on the right in Figure~\ref{fig:F1F2with}). Now, in $F_4$, the edge $(1,0)$ is forced, and then the fact that the vertices  $1,2,3$ are connected to three distinct vertices $0$ and that $F_4$ contains the $3$-cycles $(-3,1,2), (-1,2,3)$ and $(-2,1,3)$ easily forces all the other edges and labels on the boundary, as on the right in Figure~\ref{fig:F4F5}. But now in $F_5$, the edge $(3,-1)$ is forced, and as $F_5$ contains the $3$-cycle $(-2,1,3)$, the vertices $1,2$ and $-2$ and the edges $(3,-2)$ and $(1,-2)$ are also forced. But then the edges $(3,0)$ and $(2,-1)$ have to cross. This contradiction completes the proof of the proposition in the first case.

\medskip

We now consider the second case assuming that the vertex of the bead lying on the boundary of $F_1$ is labelled $-2$ (in the drawing on the left in Figure~\ref{fig:F1F2with}). Then the vertex $2$ lying in $F_1$ is forced, as also is the edge $(2,-3)$. As $F_1$ contains a $3$-cycle $(-3,1,2)$, the edge $(1,-3)$ and the vertices $1$ and $3$ are also forced, as also are the edges $(1,0)$ and $(2,-1)$. We arrive at the drawing shown on the left in Figure~\ref{fig:F1F2with2}. Note that there is more than one possible choice of the edges $(3,-1)$ lying in $F_1$ and $(3,0),(1,-2)$ lying in $F_2$, and also of some labels.

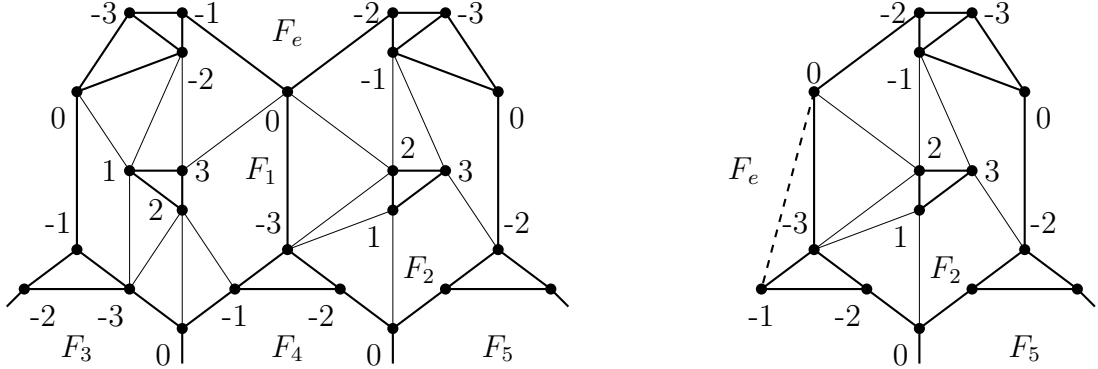
\begin{figure}[h]
\centering
\begin{tikzpicture}[scale=0.65]
\def \r {4}
\def \rv {3}
\def \ger{\r/6}
\begin{scope}
  \coordinate (0c) at (0,\rv); \fill (0c) circle (3pt); \draw (0c) node[shift=({-0.2,-0.38})] {0};
  \coordinate (0l) at (-\r,\rv); \fill (0l) circle (3pt); \draw (0l) node[shift=({-0.25,-0.35})] {0};
  \coordinate (0r) at (\r,\rv); \fill (0r) circle (3pt); \draw (0r) node[shift=({0.25,-0.35})] {0};
  \coordinate (-1l) at (-\r/2,3*\rv/2); \fill (-1l) circle (3pt); \draw (-1l) node[right=.05em] {-$1$};
  \coordinate (-2r) at (\r/2,3*\rv/2); \fill (-2r) circle (3pt); \draw (-2r) node[left=.07em] {-$2$};
  \coordinate (-3cb) at (0,0); \fill (-3cb) circle (3pt); \draw (-3cb) node[shift=({-0.25,0.35})] {-3};
  \coordinate (-1cb) at (-\r/4,-\rv/4); \fill (-1cb) circle (3pt); \draw (-1cb) node[shift=({0,-0.35})] {-$1$};
  \coordinate (-2cb) at (\r/4,-\rv/4); \fill (-2cb) circle (3pt); \draw (-2cb) node[shift=({-0.25,-0.35})] {-$2$};
  \coordinate (0lb) at (-\r/2,-\rv/2); \fill (0lb) circle (3pt); \draw (0lb) node[shift=({-0.25,-0.35})] {0};
  \coordinate (-1lb) at (-\r,0); \fill (-1lb) circle (3pt); \draw (-1lb) node[shift=({-0.25,0.35})] {-$1$};
  \coordinate (-3lb) at (-3*\r/4,-\rv/4); \fill (-3lb) circle (3pt); \draw (-3lb) node[shift=({-0.25,-0.35})] {-$3$};
  \coordinate (-2lb) at (-5*\r/4,-\rv/4); \fill (-2lb) circle (3pt); \draw (-2lb) node[shift=({0.25,-0.35})] {-$2$};
  \coordinate (0rb) at (\r/2,-\rv/2); \fill (0rb) circle (3pt); \draw (0rb) node[shift=({-0.25,-0.35})] {0};
  \coordinate (-2rb) at (\r,0); \fill (-2rb) circle (3pt); \draw (-2rb) node[shift=({0.25,0.35})] {-$2$};
  \coordinate (-1rb) at (3*\r/4,-\rv/4); \fill (-1rb) circle (3pt); 
  \coordinate (-3rb) at (5*\r/4,-\rv/4); \fill (-3rb) circle (3pt); 
  \coordinate (-3rt) at (3*\r/4,3*\rv/2); \fill (-3rt) circle (3pt); \draw (-3rt) node[right=.05em] {-$3$};
  \coordinate (-1rt) at (\r/2,5*\rv/4); \fill (-1rt) circle (3pt); \draw (-1rt) node[shift=({-0.25,-0.35})] {-$1$};
  \coordinate (-3lt) at (-3*\r/4,3*\rv/2); \fill (-3lt) circle (3pt); \draw (-3lt) node[left=.05em] {-$3$};
  \coordinate (-2lt) at (-\r/2,5*\rv/4); \fill (-2lt) circle (3pt); \draw (-2lt) node[shift=({0.25,-0.35})] {-$2$};
  \draw[thick] (0rb) -- (-1rb) -- (-3rb) -- (-2rb) -- (-1rb) (0lb) -- (-3lb) -- (-1lb) -- (-2lb) -- (-3lb) (0lb) -- (-1cb) -- (-2cb) -- (0rb);
  \draw[thick] (-3cb) -- (0c) -- (-1l)  (0c) -- (-2r) (0l) -- (-1lb) (0r) -- (-2rb) (-1cb) -- (-3cb) -- (-2cb);
  \draw[thick] (-3rt) -- (-2r) -- (-1rt) -- (-3rt) -- (0r) -- (-1rt) (-3lt) -- (-1l) -- (-2lt) -- (-3lt) -- (0l) -- (-2lt);
  \draw[thick] (0lb) -- ++(0,-\ger) (0rb) -- ++(0,-\ger) (-2lb) -- ++(-\ger/2,-\ger/2) (-3rb) -- ++(\ger/2,-\ger/2);
  \draw (0,11*\rv/8) node {$F_e$}; 
  \draw (-\r/8,\rv/2) node {$F_1$}; \draw (5*\r/8,-\rv/8) node {$F_2$};
  \draw (-\r,-5*\rv/8) node {$F_3$}; \draw (0,-5*\rv/8) node {$F_4$}; \draw (\r,-5*\rv/8) node {$F_5$};
  \coordinate (3r) at (3*\r/4,\rv/2); \fill (3r) circle (3pt); \draw (3r) node[right=.05em] {$3$};
  \coordinate (1r) at (\r/2,\rv/4); \fill (1r) circle (3pt); \draw (1r) node[shift=({-0.25,-0.35})] {$1$};
  \coordinate (2r) at (\r/2,\rv/2); \fill (2r) circle (3pt); \draw (2r) node[shift=({0.2,0.3})] {$2$};
  \draw[thick] (3r) -- (2r) -- (1r) -- (3r);
  \draw[ultra thin] (-2rb) -- (3r) -- (-1rt) -- (2r) -- (-3cb) -- (1r) -- (0rb) (2r) -- (0c); 
  \coordinate (1l) at (-\r/2,\rv/2); \fill (1l) circle (3pt); \draw (1l) node[right=.06em] {$3$};
  \coordinate (3l) at (-3*\r/4,\rv/2); \fill (3l) circle (3pt); \draw (3l) node[left=.05em] {$1$};
  \coordinate (2l) at (-\r/2,\rv/4); \fill (2l) circle (3pt); \draw (2l) node[shift=({-0.35,0})] {$2$};
  \draw[thick] (3l) -- (2l) -- (1l) -- (3l);
  \draw[ultra thin] (1l) -- (-2lt) -- (3l) (0c) -- (1l)  (-1cb) -- (2l) -- (0lb) (2l) -- (-3lb) -- (3l) -- (0l);
\end{scope}
\begin{scope}[shift={(2.5*\r,0)}]
  \coordinate (0c) at (0,\rv); \fill (0c) circle (3pt); \draw (0c) node[shift=({0,0.25})] {0};
  \coordinate (0r) at (\r,\rv); \fill (0r) circle (3pt); \draw (0r) node[shift=({0.25,-0.35})] {0};
  \coordinate (-2r) at (\r/2,3*\rv/2); \fill (-2r) circle (3pt); \draw (-2r) node[left=.07em] {-$2$};
  \coordinate (-3cb) at (0,0); \fill (-3cb) circle (3pt); \draw (-3cb) node[shift=({-0.25,0.35})] {-3};
  \coordinate (-1cb) at (-\r/4,-\rv/4); \fill (-1cb) circle (3pt); \draw (-1cb) node[shift=({0,-0.35})] {-$1$};
  \coordinate (-2cb) at (\r/4,-\rv/4); \fill (-2cb) circle (3pt); \draw (-2cb) node[shift=({-0.25,-0.35})] {-$2$};
  \coordinate (0rb) at (\r/2,-\rv/2); \fill (0rb) circle (3pt); \draw (0rb) node[shift=({-0.25,-0.35})] {0};
  \coordinate (-2rb) at (\r,0); \fill (-2rb) circle (3pt); \draw (-2rb) node[shift=({0.25,0.35})] {-$2$};
  \coordinate (-1rb) at (3*\r/4,-\rv/4); \fill (-1rb) circle (3pt); 
  \coordinate (-3rb) at (5*\r/4,-\rv/4); \fill (-3rb) circle (3pt); 
  \coordinate (-3rt) at (3*\r/4,3*\rv/2); \fill (-3rt) circle (3pt); \draw (-3rt) node[right=.05em] {-$3$};
  \coordinate (-1rt) at (\r/2,5*\rv/4); \fill (-1rt) circle (3pt); \draw (-1rt) node[shift=({-0.25,-0.35})] {-$1$};
  \draw[thick] (0rb) -- (-1rb) -- (-3rb) -- (-2rb) -- (-1rb) (-1cb) -- (-2cb) -- (0rb);
  \draw[thick] (-3cb) -- (0c)  (0c) -- (-2r) (0r) -- (-2rb) (-1cb) -- (-3cb) -- (-2cb);
  \draw[thick] (-3rt) -- (-2r) -- (-1rt) -- (-3rt) -- (0r) -- (-1rt);
  \draw[thick] (0rb) -- ++(0,-\ger) (-3rb) -- ++(\ger/2,-\ger/2);
  \draw (-\r/3,\rv/2) node {$F_e$}; 
  \draw (5*\r/8,-\rv/8) node {$F_2$};
  \draw (\r,-5*\rv/8) node {$F_5$};
  \coordinate (3r) at (3*\r/4,\rv/2); \fill (3r) circle (3pt); \draw (3r) node[right=.05em] {$3$};
  \coordinate (1r) at (\r/2,\rv/4); \fill (1r) circle (3pt); \draw (1r) node[shift=({-0.25,-0.35})] {$1$};
  \coordinate (2r) at (\r/2,\rv/2); \fill (2r) circle (3pt); \draw (2r) node[shift=({0.2,0.3})] {$2$};
  \draw[thick] (3r) -- (2r) -- (1r) -- (3r);
  \draw[ultra thin] (-2rb) -- (3r) -- (-1rt) -- (2r) -- (-3cb) -- (1r) -- (0rb) (2r) -- (0c); 
  \draw[thick, dashed] (0c) -- (-1cb);
\end{scope}
\end{tikzpicture}
\caption{$9$-faces $F_1$ and $F_2$ with parts of the semi-cover $G'$ inside.}
\label{fig:F1F2with2}
\end{figure}

We now consider the faces $F_3, F_4$ and $F_5$ of $H$ as shown on the left in Figure~\ref{fig:F1F2with2}. They must be pairwise different, and neither of $F_3$ and $F_5$ may coincide with $F_e$ (as each of them shares a single edge with $F_e$). Assuming that $F_4$ coincides with $F_e$, we can join the vertices $0$ and $-1$ of $F_1$ and remove all the faces of $H$ to the left of the resulting edges (including $F_1$ and $F_3$) which results in a smaller semi-cover contradicting the choice of $G'$ (or Lemma~\ref{l:no2b}), shown on the right in Figure~\ref{fig:F1F2with2}. Moreover, neither of the faces $F_3$ and $F_5$ can be triangular (contradiction with Lemma~\ref{l:no2b}), and also $F_4$ cannot be triangular, as it has at least two vertices $0$. We conclude that all three faces $F_3,F_4$ and $F_5$ are internal $9$-faces of $H$.

Figure~\ref{fig:F3F4F5} shows the faces $F_3,F_4$ and $F_5$, together with some forced labels and edges, and the segments of the edges belonging to $F_1$ and $F_2$ whose endpoints lie on the boundaries of $F_3$ and $F_4$ (note that the choice of some of these edges is not unique; the labels ``$1?$" and ``$3?$" indicate that the corresponding edges may go either to $F_2$ or $F_1$ respectively, or to $F_4$).

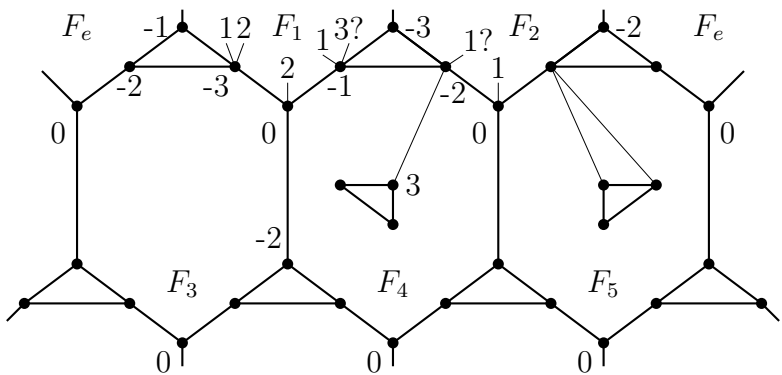
\begin{figure}[h]
\centering
\begin{tikzpicture}[scale=0.7]
\def \r {4}
\def \rv {3}
\def \ger{\r/6}
  \coordinate (0c) at (0,\rv); \fill (0c) circle (3pt); \draw (0c) node[shift=({-0.25,-0.35})] {0};
  \coordinate (0l) at (-\r,\rv); \fill (0l) circle (3pt); \draw (0l) node[shift=({-0.25,-0.35})] {0};
  \coordinate (0r) at (\r,\rv); \fill (0r) circle (3pt); \draw (0r) node[shift=({0.25,-0.35})] {0};
  \coordinate (-1l) at (-\r/2,3*\rv/2); \fill (-1l) circle (3pt); \draw (-1l) node[right=.05em] {-$3$};
  \coordinate (-2r) at (\r/2,3*\rv/2); \fill (-2r) circle (3pt); \draw (-2r) node[right=.05em] {-$2$};
  \coordinate (-3cb) at (0,0); \fill (-3cb) circle (3pt); 
  \coordinate (-1cb) at (-\r/4,-\rv/4); \fill (-1cb) circle (3pt); 
  \coordinate (-2cb) at (\r/4,-\rv/4); \fill (-2cb) circle (3pt); 
  \coordinate (0lb) at (-\r/2,-\rv/2); \fill (0lb) circle (3pt); \draw (0lb) node[shift=({-0.25,-0.25})] {0};
  \coordinate (-1lb) at (-\r,0); \fill (-1lb) circle (3pt); \draw (-1lb) node[shift=({-0.25,0.35})] {-$2$};
  \coordinate (-3lb) at (-3*\r/4,-\rv/4); \fill (-3lb) circle (3pt); 
  \coordinate (-2lb) at (-5*\r/4,-\rv/4); \fill (-2lb) circle (3pt); 
  \coordinate (0rb) at (\r/2,-\rv/2); \fill (0rb) circle (3pt); \draw (0rb) node[shift=({-0.25,-0.25})] {0};
  \coordinate (-2rb) at (\r,0); \fill (-2rb) circle (3pt); 
  \coordinate (-1rb) at (3*\r/4,-\rv/4); \fill (-1rb) circle (3pt); 
  \coordinate (-3rb) at (5*\r/4,-\rv/4); \fill (-3rb) circle (3pt); 
  \coordinate (-3rt) at (\r/4,5*\rv/4); \fill (-3rt) circle (3pt); 
  \coordinate (-1rt) at (3*\r/4,5*\rv/4); \fill (-1rt) circle (3pt); 
  \coordinate (-3lt) at (-\r/4,5*\rv/4); \fill (-3lt) circle (3pt); \draw (-3lt) node[shift=({0.1,-0.35})] {-$2$};
  \coordinate (-2lt) at (-3*\r/4,5*\rv/4); \fill (-2lt) circle (3pt); \draw (-2lt) node[shift=({0,-0.25})] {-$1$};
  \coordinate (0ll) at (-2*\r,\rv); \fill (0ll) circle (3pt); \draw (0ll) node[shift=({-0.25,-0.35})] {0};
  \coordinate (-1ll) at (-3*\r/2,3*\rv/2); \fill (-1ll) circle (3pt); \draw (-1ll) node[left=.06em] {-$1$};
  \coordinate (0llb) at (-3*\r/2,-\rv/2); \fill (0llb) circle (3pt); \draw (0llb) node[shift=({-0.25,-0.25})] {0};
  \coordinate (-1llb) at (-2*\r,0); \fill (-1llb) circle (3pt); 
  \coordinate (-3llb) at (-7*\r/4,-\rv/4); \fill (-3llb) circle (3pt); 
  \coordinate (-2llb) at (-9*\r/4,-\rv/4); \fill (-2llb) circle (3pt); 
  \coordinate (-3llt) at (-5*\r/4,5*\rv/4); \fill (-3llt) circle (3pt); \draw (-3llt) node[shift=({-0.25,-0.25})] {-$3$};
  \coordinate (-2llt) at (-7*\r/4,5*\rv/4); \fill (-2llt) circle (3pt); \draw (-2llt) node[shift=({0,-0.25})] {-$2$};
  \draw[thick] (0rb) -- (-1rb) -- (-3rb) -- (-2rb) -- (-1rb) (0lb) -- (-3lb) -- (-1lb) -- (-2lb) -- (-3lb) (0lb) -- (-1cb) -- (-2cb) -- (0rb);
  \draw[thick] (-3cb) -- (0c) -- (-1l)  (0c) -- (-2r) (0l) -- (-1lb) (0r) -- (-2rb) (-1cb) -- (-3cb) -- (-2cb);
  \draw[thick] (-3rt) -- (-2r) -- (-1rt) -- (-3rt)  (0r) -- (-1rt) (-3lt) -- (-1l) -- (-2lt) -- (-3lt) (0l) -- (-2lt);
  \draw[thick]  (-2lb) -- (0llb) -- (-3llb) -- (-2llb) -- (-1llb) -- (0ll) -- (-2llt) -- (-1ll) -- (-3llt) -- (0l) (-2llt) -- (-3llt) (-1llb) -- (-3llb);
  \draw[thick] (0lb) -- ++(0,-2*\ger/3) (0llb) -- ++(0,-2*\ger/3) (0rb) -- ++(0,-2*\ger/3) (-2llb) -- ++(-\ger/2,-\ger/2) (-3rb) -- ++(\ger/2,-\ger/2);
  \draw[thick] (-1l) -- ++(0,\ger/2) (-1ll) -- ++(0,\ger/2) (-2r) -- ++(0,\ger/2) (0ll) -- ++(-\ger,\ger) (0r) -- ++(\ger,\ger);
  \draw (\r,3*\rv/2) node {$F_e$}; \draw (-2*\r,3*\rv/2) node {$F_e$};
  \draw (-3*\r/2,-\rv/8) node {$F_3$}; \draw (-\r/2,-\rv/8) node {$F_4$}; \draw (\r/2,-\rv/8) node {$F_5$};
  \draw (\r/8,3*\rv/2) node {$F_2$}; \draw (-\r,3*\rv/2) node {$F_1$};
  \coordinate (3F1) at ($(-2lt)+(\ger/4,2*\ger/3)$); \draw (3F1) node[shift=({0,0.2})] {$3?$};
  \coordinate (1F1) at ($(-2lt)+(-2*\ger/4,\ger/3)$); \draw (1F1) node[shift=({0,0.2})] {$1$};
  \draw[ultra thin] (3F1) -- (-2lt) -- (1F1);
  \coordinate (20F1) at ($(0l)+(0,2*\ger/3)$); \draw (20F1) node[shift=({0,0.2})] {$2$};
  \coordinate (1F2) at ($(0c)+(0,2*\ger/3)$); \draw (1F2) node[shift=({0,0.2})] {$1$};
  \draw[ultra thin] (20F1) -- (0l) (0c)-- (1F2);
  \coordinate (1F31) at ($(-3llt)+(-\ger/4,2*\ger/3)$); \draw (1F31) node[shift=({0,0.2})] {$1$};
  \coordinate (2F31) at ($(-3llt)+(\ger/4,2*\ger/3)$); \draw (2F31) node[shift=({0,0.2})] {$2$};
  \draw[ultra thin] (1F31) -- (-3llt) -- (2F31);
  \coordinate (1F12) at ($(-3lt)+(\ger/2,\ger/3)$); \draw (1F12) node[shift=({0.2,0.2})] {$1?$};
  \draw[ultra thin] (1F12) -- (-3lt);
  \coordinate (3r) at (3*\r/4,\rv/2); \fill (3r) circle (3pt); 
  \coordinate (1r) at (\r/2,\rv/4); \fill (1r) circle (3pt); 
  \coordinate (2r) at (\r/2,\rv/2); \fill (2r) circle (3pt); 
  \draw[thick] (3r) -- (2r) -- (1r) -- (3r);
  \draw[ultra thin] (3r) -- (-3rt) -- (2r);
  \coordinate (1l) at (-\r/2,\rv/2); \fill (1l) circle (3pt); \draw (1l) node[right=.05em] {$3$};
  \coordinate (3l) at (-3*\r/4,\rv/2); \fill (3l) circle (3pt); 
  \coordinate (2l) at (-\r/2,\rv/4); \fill (2l) circle (3pt); 
  \draw[thick] (3l) -- (2l) -- (1l) -- (3l);
  \draw[ultra thin] (1l) -- (-3lt);
\end{tikzpicture}
\caption{$9$-faces $F_3, F_4$ and $F_5$.}
\label{fig:F3F4F5}
\end{figure}

The face $F_3$ shares the single edge $(0,-2)$ with $F_e$, so by Remark~\ref{rem:trianglesH}, it contains $3$-cycles $(-1,2,3)$ and $(-3,1,2)$. Moreover, the vertex $-2$ on the common boundary of $F_3$ and $F_4$ is connected to the vertex $3$ lying in $F_3$ (we already have the edge $(-2,3)$ in $F_4$). It is not hard to verify that we obtain the drawing as in Figure~\ref{fig:F3inF4F5}, with more than one choice for the edge $(1,0)$ lying in $F_3$ and a choice of which of the two unlabelled vertices has label $-3$.

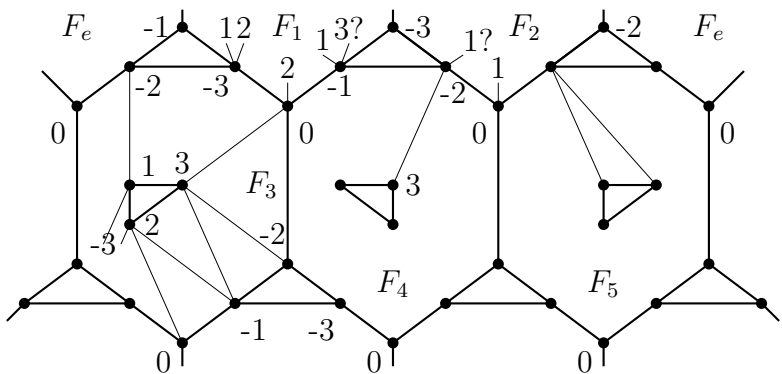
\begin{figure}[h]
\centering
\begin{tikzpicture}[scale=0.7]
\def \r {4}
\def \rv {3}
\def \ger{\r/6}
  \coordinate (0c) at (0,\rv); \fill (0c) circle (3pt); \draw (0c) node[shift=({-0.25,-0.35})] {0};
  \coordinate (0l) at (-\r,\rv); \fill (0l) circle (3pt); \draw (0l) node[shift=({0.25,-0.35})] {0};
  \coordinate (0r) at (\r,\rv); \fill (0r) circle (3pt); \draw (0r) node[shift=({0.25,-0.35})] {0};
  \coordinate (-1l) at (-\r/2,3*\rv/2); \fill (-1l) circle (3pt); \draw (-1l) node[right=.05em] {-$3$};
  \coordinate (-2r) at (\r/2,3*\rv/2); \fill (-2r) circle (3pt); \draw (-2r) node[right=.05em] {-$2$};
  \coordinate (-3cb) at (0,0); \fill (-3cb) circle (3pt); 
  \coordinate (-1cb) at (-\r/4,-\rv/4); \fill (-1cb) circle (3pt); 
  \coordinate (-2cb) at (\r/4,-\rv/4); \fill (-2cb) circle (3pt); 
  \coordinate (0lb) at (-\r/2,-\rv/2); \fill (0lb) circle (3pt); \draw (0lb) node[shift=({-0.25,-0.25})] {0};
  \coordinate (-1lb) at (-\r,0); \fill (-1lb) circle (3pt); \draw (-1lb) node[shift=({-0.2,0.4})] {-$2$};
  \coordinate (-3lb) at (-3*\r/4,-\rv/4); \fill (-3lb) circle (3pt); \draw (-3lb) node[shift=({-0.25,-0.35})] {-$3$};
  \coordinate (-2lb) at (-5*\r/4,-\rv/4); \fill (-2lb) circle (3pt); \draw (-2lb) node[shift=({0.25,-0.35})] {-$1$};
  \coordinate (0rb) at (\r/2,-\rv/2); \fill (0rb) circle (3pt); \draw (0rb) node[shift=({-0.25,-0.25})] {0};
  \coordinate (-2rb) at (\r,0); \fill (-2rb) circle (3pt); 
  \coordinate (-1rb) at (3*\r/4,-\rv/4); \fill (-1rb) circle (3pt); 
  \coordinate (-3rb) at (5*\r/4,-\rv/4); \fill (-3rb) circle (3pt); 
  \coordinate (-3rt) at (\r/4,5*\rv/4); \fill (-3rt) circle (3pt); 
  \coordinate (-1rt) at (3*\r/4,5*\rv/4); \fill (-1rt) circle (3pt); 
  \coordinate (-3lt) at (-\r/4,5*\rv/4); \fill (-3lt) circle (3pt); \draw (-3lt) node[shift=({0.1,-0.35})] {-$2$};
  \coordinate (-2lt) at (-3*\r/4,5*\rv/4); \fill (-2lt) circle (3pt); \draw (-2lt) node[shift=({0,-0.25})] {-$1$};
  \coordinate (0ll) at (-2*\r,\rv); \fill (0ll) circle (3pt); \draw (0ll) node[shift=({-0.25,-0.35})] {0};
  \coordinate (-1ll) at (-3*\r/2,3*\rv/2); \fill (-1ll) circle (3pt); \draw (-1ll) node[left=.06em] {-$1$};
  \coordinate (0llb) at (-3*\r/2,-\rv/2); \fill (0llb) circle (3pt); \draw (0llb) node[shift=({-0.25,-0.25})] {0};
  \coordinate (-1llb) at (-2*\r,0); \fill (-1llb) circle (3pt); 
  \coordinate (-3llb) at (-7*\r/4,-\rv/4); \fill (-3llb) circle (3pt); 
  \coordinate (-2llb) at (-9*\r/4,-\rv/4); \fill (-2llb) circle (3pt); 
  \coordinate (-3llt) at (-5*\r/4,5*\rv/4); \fill (-3llt) circle (3pt); \draw (-3llt) node[shift=({-0.25,-0.25})] {-$3$};
  \coordinate (-2llt) at (-7*\r/4,5*\rv/4); \fill (-2llt) circle (3pt); \draw (-2llt) node[shift=({0.25,-0.25})] {-$2$};
  \draw[thick] (0rb) -- (-1rb) -- (-3rb) -- (-2rb) -- (-1rb) (0lb) -- (-3lb) -- (-1lb) -- (-2lb) -- (-3lb) (0lb) -- (-1cb) -- (-2cb) -- (0rb);
  \draw[thick] (-3cb) -- (0c) -- (-1l)  (0c) -- (-2r) (0l) -- (-1lb) (0r) -- (-2rb) (-1cb) -- (-3cb) -- (-2cb);
  \draw[thick] (-3rt) -- (-2r) -- (-1rt) -- (-3rt)  (0r) -- (-1rt) (-3lt) -- (-1l) -- (-2lt) -- (-3lt) (0l) -- (-2lt);
  \draw[thick]  (-2lb) -- (0llb) -- (-3llb) -- (-2llb) -- (-1llb) -- (0ll) -- (-2llt) -- (-1ll) -- (-3llt) -- (0l) (-2llt) -- (-3llt) (-1llb) -- (-3llb);
  \draw[thick] (0lb) -- ++(0,-2*\ger/3) (0llb) -- ++(0,-2*\ger/3) (0rb) -- ++(0,-2*\ger/3) (-2llb) -- ++(-\ger/2,-\ger/2) (-3rb) -- ++(\ger/2,-\ger/2);
  \draw[thick] (-1l) -- ++(0,\ger/2) (-1ll) -- ++(0,\ger/2) (-2r) -- ++(0,\ger/2) (0ll) -- ++(-\ger,\ger) (0r) -- ++(\ger,\ger);
  \draw (\r,3*\rv/2) node {$F_e$}; \draw (-2*\r,3*\rv/2) node {$F_e$};
  \draw (-9*\r/8,\rv/2) node {$F_3$}; \draw (-\r/2,-\rv/8) node {$F_4$}; \draw (\r/2,-\rv/8) node {$F_5$};
  \draw (\r/8,3*\rv/2) node {$F_2$}; \draw (-\r,3*\rv/2) node {$F_1$};
  \coordinate (3F1) at ($(-2lt)+(\ger/4,2*\ger/3)$); \draw (3F1) node[shift=({0,0.2})] {$3?$};
  \coordinate (1F1) at ($(-2lt)+(-2*\ger/4,\ger/3)$); \draw (1F1) node[shift=({0,0.2})] {$1$};
  \draw[ultra thin] (3F1) -- (-2lt) -- (1F1);
  \coordinate (20F1) at ($(0l)+(0,2*\ger/3)$); \draw (20F1) node[shift=({0,0.2})] {$2$};
  \coordinate (1F2) at ($(0c)+(0,2*\ger/3)$); \draw (1F2) node[shift=({0,0.2})] {$1$};
  \draw[ultra thin] (20F1) -- (0l) (0c)-- (1F2);
  \coordinate (1F31) at ($(-3llt)+(-\ger/4,2*\ger/3)$); \draw (1F31) node[shift=({0,0.2})] {$1$};
  \coordinate (2F31) at ($(-3llt)+(\ger/4,2*\ger/3)$); \draw (2F31) node[shift=({0,0.2})] {$2$};
  \draw[ultra thin] (1F31) -- (-3llt) -- (2F31);
  \coordinate (1F12) at ($(-3lt)+(\ger/2,\ger/3)$); \draw (1F12) node[shift=({0.2,0.2})] {$1?$};
  \draw[ultra thin] (1F12) -- (-3lt);
  \coordinate (3r) at (3*\r/4,\rv/2); \fill (3r) circle (3pt); 
  \coordinate (1r) at (\r/2,\rv/4); \fill (1r) circle (3pt); 
  \coordinate (2r) at (\r/2,\rv/2); \fill (2r) circle (3pt); 
  \draw[thick] (3r) -- (2r) -- (1r) -- (3r);
  \draw[ultra thin] (3r) -- (-3rt) -- (2r);
  \coordinate (1l) at (-\r/2,\rv/2); \fill (1l) circle (3pt); \draw (1l) node[right=.05em] {$3$};
  \coordinate (3l) at (-3*\r/4,\rv/2); \fill (3l) circle (3pt); 
  \coordinate (2l) at (-\r/2,\rv/4); \fill (2l) circle (3pt); 
  \draw[thick] (3l) -- (2l) -- (1l) -- (3l);
  \draw[ultra thin] (1l) -- (-3lt);
  \coordinate (1ll) at (-7*\r/4,\rv/2); \fill (1ll) circle (3pt); \draw (1ll) node[shift=({0.25,0.25})] {$1$};
  \coordinate (3ll) at (-3*\r/2,\rv/2); \fill (3ll) circle (3pt); \draw (3ll) node[above=.05em] {$3$};
  \coordinate (2ll) at (-7*\r/4,\rv/4); \fill (2ll) circle (3pt); \draw (2ll) node[shift=({0.3,0})] {$2$};
  \draw[thick] (3ll) -- (2ll) -- (1ll) -- (3ll);
  \draw[ultra thin] (-2llt) -- (1ll) (0llb) -- (2ll) -- (-2lb) -- (3ll) -- (-1lb) (3ll) -- (0l);
  \draw[ultra thin] (1ll) -- ++(-\r/8,-3*\rv/8) (2ll) -- ++(-\r/24,-\rv/8);
  \draw (-15*\r/8,\rv/8) node {-$3$};
\end{tikzpicture}
\caption{$9$-faces $F_3, F_4$ and $F_5$ with a part of $G'$ lying in $F_3$.}
\label{fig:F3inF4F5}
\end{figure}

Remark~\ref{rem:trianglesH} does not give us any information on $F_4$. However, we know that the edge $(-2,1)$ from the vertex $-2$ on the common boundary of $F_3$ and $F_4$ lies in $F_4$ (this tells us that the edge $(1,-2)$ in $F_2$ is forced). Moreover, that edge is forced (to be able to have $(2,0)$). In addition, the labels $1$ and $2$ are forced, forcing the edge $(3,0)$, and then the edge $(1,-3)$. We arrive at the drawing as in Figure~\ref{fig:F3F4inF5} (note that several edges are still missing from $F_4$). We also label the vertices of the $3$-cycle $(1,2,3)$ lying $F_5$ by $(i,j,k)$, where $\{i,j,k\}=\{1,2,3\}$.

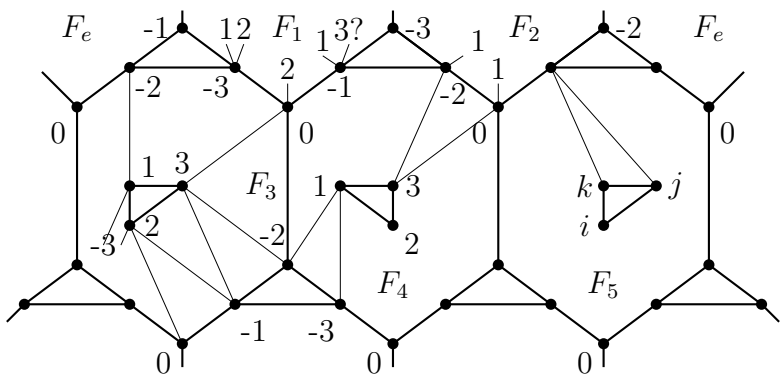
\begin{figure}[h]
\centering
\begin{tikzpicture}[scale=0.7]
\def \r {4}
\def \rv {3}
\def \ger{\r/6}
  \coordinate (0c) at (0,\rv); \fill (0c) circle (3pt); \draw (0c) node[shift=({-0.25,-0.35})] {0};
  \coordinate (0l) at (-\r,\rv); \fill (0l) circle (3pt); \draw (0l) node[shift=({0.25,-0.35})] {0};
  \coordinate (0r) at (\r,\rv); \fill (0r) circle (3pt); \draw (0r) node[shift=({0.25,-0.35})] {0};
  \coordinate (-1l) at (-\r/2,3*\rv/2); \fill (-1l) circle (3pt); \draw (-1l) node[right=.05em] {-$3$};
  \coordinate (-2r) at (\r/2,3*\rv/2); \fill (-2r) circle (3pt); \draw (-2r) node[right=.05em] {-$2$};
  \coordinate (-3cb) at (0,0); \fill (-3cb) circle (3pt); 
  \coordinate (-1cb) at (-\r/4,-\rv/4); \fill (-1cb) circle (3pt); 
  \coordinate (-2cb) at (\r/4,-\rv/4); \fill (-2cb) circle (3pt); 
  \coordinate (0lb) at (-\r/2,-\rv/2); \fill (0lb) circle (3pt); \draw (0lb) node[shift=({-0.25,-0.25})] {0};
  \coordinate (-1lb) at (-\r,0); \fill (-1lb) circle (3pt); \draw (-1lb) node[shift=({-0.2,0.4})] {-$2$};
  \coordinate (-3lb) at (-3*\r/4,-\rv/4); \fill (-3lb) circle (3pt); \draw (-3lb) node[shift=({-0.25,-0.35})] {-$3$};
  \coordinate (-2lb) at (-5*\r/4,-\rv/4); \fill (-2lb) circle (3pt); \draw (-2lb) node[shift=({0.25,-0.35})] {-$1$};
  \coordinate (0rb) at (\r/2,-\rv/2); \fill (0rb) circle (3pt); \draw (0rb) node[shift=({-0.25,-0.25})] {0};
  \coordinate (-2rb) at (\r,0); \fill (-2rb) circle (3pt); 
  \coordinate (-1rb) at (3*\r/4,-\rv/4); \fill (-1rb) circle (3pt); 
  \coordinate (-3rb) at (5*\r/4,-\rv/4); \fill (-3rb) circle (3pt); 
  \coordinate (-3rt) at (\r/4,5*\rv/4); \fill (-3rt) circle (3pt); 
  \coordinate (-1rt) at (3*\r/4,5*\rv/4); \fill (-1rt) circle (3pt); 
  \coordinate (-3lt) at (-\r/4,5*\rv/4); \fill (-3lt) circle (3pt); \draw (-3lt) node[shift=({0.1,-0.35})] {-$2$};
  \coordinate (-2lt) at (-3*\r/4,5*\rv/4); \fill (-2lt) circle (3pt); \draw (-2lt) node[shift=({0,-0.25})] {-$1$};
  \coordinate (0ll) at (-2*\r,\rv); \fill (0ll) circle (3pt); \draw (0ll) node[shift=({-0.25,-0.35})] {0};
  \coordinate (-1ll) at (-3*\r/2,3*\rv/2); \fill (-1ll) circle (3pt); \draw (-1ll) node[left=.06em] {-$1$};
  \coordinate (0llb) at (-3*\r/2,-\rv/2); \fill (0llb) circle (3pt); \draw (0llb) node[shift=({-0.25,-0.25})] {0};
  \coordinate (-1llb) at (-2*\r,0); \fill (-1llb) circle (3pt); 
  \coordinate (-3llb) at (-7*\r/4,-\rv/4); \fill (-3llb) circle (3pt); 
  \coordinate (-2llb) at (-9*\r/4,-\rv/4); \fill (-2llb) circle (3pt); 
  \coordinate (-3llt) at (-5*\r/4,5*\rv/4); \fill (-3llt) circle (3pt); \draw (-3llt) node[shift=({-0.25,-0.25})] {-$3$};
  \coordinate (-2llt) at (-7*\r/4,5*\rv/4); \fill (-2llt) circle (3pt); \draw (-2llt) node[shift=({0.25,-0.25})] {-$2$};
  \draw[thick] (0rb) -- (-1rb) -- (-3rb) -- (-2rb) -- (-1rb) (0lb) -- (-3lb) -- (-1lb) -- (-2lb) -- (-3lb) (0lb) -- (-1cb) -- (-2cb) -- (0rb);
  \draw[thick] (-3cb) -- (0c) -- (-1l)  (0c) -- (-2r) (0l) -- (-1lb) (0r) -- (-2rb) (-1cb) -- (-3cb) -- (-2cb);
  \draw[thick] (-3rt) -- (-2r) -- (-1rt) -- (-3rt)  (0r) -- (-1rt) (-3lt) -- (-1l) -- (-2lt) -- (-3lt) (0l) -- (-2lt);
  \draw[thick]  (-2lb) -- (0llb) -- (-3llb) -- (-2llb) -- (-1llb) -- (0ll) -- (-2llt) -- (-1ll) -- (-3llt) -- (0l) (-2llt) -- (-3llt) (-1llb) -- (-3llb);
  \draw[thick] (0lb) -- ++(0,-2*\ger/3) (0llb) -- ++(0,-2*\ger/3) (0rb) -- ++(0,-2*\ger/3) (-2llb) -- ++(-\ger/2,-\ger/2) (-3rb) -- ++(\ger/2,-\ger/2);
  \draw[thick] (-1l) -- ++(0,\ger/2) (-1ll) -- ++(0,\ger/2) (-2r) -- ++(0,\ger/2) (0ll) -- ++(-\ger,\ger) (0r) -- ++(\ger,\ger);
  \draw (\r,3*\rv/2) node {$F_e$}; \draw (-2*\r,3*\rv/2) node {$F_e$};
  \draw (-9*\r/8,\rv/2) node {$F_3$}; \draw (-\r/2,-\rv/8) node {$F_4$}; \draw (\r/2,-\rv/8) node {$F_5$};
  \draw (\r/8,3*\rv/2) node {$F_2$}; \draw (-\r,3*\rv/2) node {$F_1$};
  \coordinate (3F1) at ($(-2lt)+(\ger/4,2*\ger/3)$); \draw (3F1) node[shift=({0,0.2})] {$3?$};
  \coordinate (1F1) at ($(-2lt)+(-2*\ger/4,\ger/3)$); \draw (1F1) node[shift=({0,0.2})] {$1$};
  \draw[ultra thin] (3F1) -- (-2lt) -- (1F1);
  \coordinate (20F1) at ($(0l)+(0,2*\ger/3)$); \draw (20F1) node[shift=({0,0.2})] {$2$};
  \coordinate (1F2) at ($(0c)+(0,2*\ger/3)$); \draw (1F2) node[shift=({0,0.2})] {$1$};
  \draw[ultra thin] (20F1) -- (0l) (0c)-- (1F2);
  \coordinate (1F31) at ($(-3llt)+(-\ger/4,2*\ger/3)$); \draw (1F31) node[shift=({0,0.2})] {$1$};
  \coordinate (2F31) at ($(-3llt)+(\ger/4,2*\ger/3)$); \draw (2F31) node[shift=({0,0.2})] {$2$};
  \draw[ultra thin] (1F31) -- (-3llt) -- (2F31);
  \coordinate (1F12) at ($(-3lt)+(\ger/2,\ger/3)$); \draw (1F12) node[shift=({0.2,0.2})] {$1$};
  \draw[ultra thin] (1F12) -- (-3lt);
  \coordinate (3r) at (3*\r/4,\rv/2); \fill (3r) circle (3pt); \draw (3r) node[right=.05em] {$j$};
  \coordinate (1r) at (\r/2,\rv/4); \fill (1r) circle (3pt); \draw (1r) node[shift=({-0.25,0})] {$i$};
  \coordinate (2r) at (\r/2,\rv/2); \fill (2r) circle (3pt); \draw (2r) node[shift=({-0.25,0})] {$k$};
  \draw[thick] (3r) -- (2r) -- (1r) -- (3r);
  \draw[ultra thin] (3r) -- (-3rt) -- (2r);
  \coordinate (1l) at (-\r/2,\rv/2); \fill (1l) circle (3pt); \draw (1l) node[right=.05em] {$3$};
  \coordinate (3l) at (-3*\r/4,\rv/2); \fill (3l) circle (3pt); \draw (3l) node[left=.05em] {$1$};
  \coordinate (2l) at (-\r/2,\rv/4); \fill (2l) circle (3pt); \draw (2l) node[shift=({0.25,-0.25})] {$2$};
  \draw[thick] (3l) -- (2l) -- (1l) -- (3l);
  \draw[ultra thin] (-1lb) -- (3l) -- (-3lb)  (0c) -- (1l) -- (-3lt);
  \coordinate (1ll) at (-7*\r/4,\rv/2); \fill (1ll) circle (3pt); \draw (1ll) node[shift=({0.25,0.25})] {$1$};
  \coordinate (3ll) at (-3*\r/2,\rv/2); \fill (3ll) circle (3pt); \draw (3ll) node[above=.05em] {$3$};
  \coordinate (2ll) at (-7*\r/4,\rv/4); \fill (2ll) circle (3pt); \draw (2ll) node[shift=({0.3,0})] {$2$};
  \draw[thick] (3ll) -- (2ll) -- (1ll) -- (3ll);
  \draw[ultra thin] (-2llt) -- (1ll) (0llb) -- (2ll) -- (-2lb) -- (3ll) -- (-1lb) (3ll) -- (0l);
  \draw[ultra thin] (1ll) -- ++(-\r/8,-3*\rv/8) (2ll) -- ++(-\r/24,-\rv/8);
  \draw (-15*\r/8,\rv/8) node {-$3$};
\end{tikzpicture}
\caption{$9$-faces $F_3, F_4$ and $F_5$ with a part of $G'$ lying in $F_3, F_4$.}
\label{fig:F3F4inF5}
\end{figure}

Note that by Remark~\ref{rem:trianglesH}, the face $F_5$ contains a $3$-cycle $(-2,1,3)$. Moreover, $i \ne 2$. If $j=2$, then the vertices $\{i,k\}=\{1,3\}$ must be both connected to a $-2$ on the boundary of $F_5$, and hence must be connected to the two different vertices $0$. We have the labelling and the two forced edges $(k,0)$ and $(i,0)$ as on the left in Figure~\ref{fig:F5}. But then the edge $(2,0)$ is also forced, as also is $(i,-k)$ which leads to a contradiction: no $3$-cycle $(-2,1,3)$. We deduce that $k=2$ and $\{i,j\}=\{1,3\}$. By the similar argument, the vertices $i$ and $j$ must be connected to different vertices $0$, and then the edges $(i,0)$ and $(j,0)$ are forced, as is also the edge $(2,-j)$. We obtain the drawing on the right in Figure~\ref{fig:F5}.

\begin{figure}[h]
\centering
\begin{tikzpicture}[scale=0.7]
\def \r {4}
\def \rv {3}
\def \ger{\r/6}
\begin{scope}
  \coordinate (0c) at (0,\rv); \fill (0c) circle (3pt); \draw (0c) node[shift=({0.25,-0.35})] {0};
  \coordinate (0r) at (\r,\rv); \fill (0r) circle (3pt); \draw (0r) node[shift=({0.25,-0.35})] {0};
  \coordinate (-2r) at (\r/2,3*\rv/2); \fill (-2r) circle (3pt); \draw (-2r) node[right=.05em] {-$2$};
  \coordinate (-3cb) at (0,0); \fill (-3cb) circle (3pt); 
  \coordinate (-1cb) at (-\r/4,-\rv/4); \fill (-1cb) circle (3pt); 
  \coordinate (-2cb) at (\r/4,-\rv/4); \fill (-2cb) circle (3pt); 
  \coordinate (0rb) at (\r/2,-\rv/2); \fill (0rb) circle (3pt); \draw (0rb) node[shift=({-0.25,-0.25})] {0};
  \coordinate (-2rb) at (\r,0); \fill (-2rb) circle (3pt); 
  \coordinate (-1rb) at (3*\r/4,-\rv/4); \fill (-1rb) circle (3pt); 
  \coordinate (-3rb) at (5*\r/4,-\rv/4); \fill (-3rb) circle (3pt); 
  \coordinate (-3rt) at (\r/4,5*\rv/4); \fill (-3rt) circle (3pt); \draw (-3rt) node[shift=({-0.2,0.25})] {-$i$};
  \coordinate (-1rt) at (3*\r/4,5*\rv/4); \fill (-1rt) circle (3pt); \draw (-1rt) node[shift=({0.25,0.25})] {-$k$};
  \draw[thick] (0rb) -- (-1rb) -- (-3rb) -- (-2rb) -- (-1rb) (-1cb) -- (-2cb) -- (0rb);
  \draw[thick] (-3cb) -- (0c) (0r) -- (-2rb) (-1cb) -- (-3cb) -- (-2cb);
  \draw[thick] (0c) -- (-3rt) -- (-2r) -- (-1rt) -- (-3rt)  (0r) -- (-1rt);
  \draw[thick] (0rb) -- ++(0,-2*\ger/3) (-3rb) -- ++(\ger/2,-\ger/2)  (-1cb) -- ++(-\ger/2,-\ger/2) (-2r) -- ++(0,\ger/2) (0r) -- ++(\ger,\ger);
  \draw (\r,3*\rv/2) node {$F_e$};
  \draw (\r/2,-\rv/8) node {$F_5$}; \draw (-\r/8,\rv/4) node {$F_4$};
  \draw (-\r/16,3*\rv/2) node {$F_2$};
  \coordinate (1F2) at ($(0c)+(0,2*\ger/3)$); \draw (1F2) node[shift=({0,0.2})] {$1$};
  \coordinate (3F4) at ($(0c)+(-\ger,-\ger)$); \draw (3F4) node[shift=({-0.2,-0.2})] {$3$};
  \coordinate (-2) at ($(0c)+(-\ger,\ger)$); \draw (-2) node[shift=({-0.2,0.2})] {-$2$};
  \draw[ultra thin] (3F4) -- (0c)-- (1F2);
  \draw[thick] (0c) -- (-2);
  \coordinate (3r) at (3*\r/4,\rv/2); \fill (3r) circle (3pt); \draw (3r) node[above=.05em] {$2$};
  \coordinate (1r) at (\r/2,\rv/4); \fill (1r) circle (3pt); \draw (1r) node[shift=({-0.25,0})] {$i$};
  \coordinate (2r) at (\r/2,\rv/2); \fill (2r) circle (3pt); \draw (2r) node[shift=({-0.25,0})] {$k$};
  \draw[thick] (3r) -- (2r) -- (1r) -- (3r);
  \draw[ultra thin] (3r) -- (-3rt) -- (2r);
  \draw[ultra thin] (2r) .. controls (\r/4,0) .. (0rb);
  \draw[ultra thin] (1r) .. controls (3*\r/4,0) .. (0r);
\end{scope}
\begin{scope}[shift={(2.5*\r,0)}]
  \coordinate (0c) at (0,\rv); \fill (0c) circle (3pt); \draw (0c) node[shift=({0.25,-0.35})] {0};
  \coordinate (0r) at (\r,\rv); \fill (0r) circle (3pt); \draw (0r) node[shift=({0.25,-0.35})] {0};
  \coordinate (-2r) at (\r/2,3*\rv/2); \fill (-2r) circle (3pt); \draw (-2r) node[right=.05em] {-$2$};
  \coordinate (-3cb) at (0,0); \fill (-3cb) circle (3pt); \draw (-3cb) node[shift=({-0.25,0.35})] {-$j$};
  \coordinate (-1cb) at (-\r/4,-\rv/4); \fill (-1cb) circle (3pt); 
  \coordinate (-2cb) at (\r/4,-\rv/4); \fill (-2cb) circle (3pt); 
  \coordinate (0rb) at (\r/2,-\rv/2); \fill (0rb) circle (3pt); \draw (0rb) node[shift=({-0.25,-0.25})] {0};
  \coordinate (-2rb) at (\r,0); \fill (-2rb) circle (3pt); 
  \coordinate (-1rb) at (3*\r/4,-\rv/4); \fill (-1rb) circle (3pt); 
  \coordinate (-3rb) at (5*\r/4,-\rv/4); \fill (-3rb) circle (3pt); 
  \coordinate (-3rt) at (\r/4,5*\rv/4); \fill (-3rt) circle (3pt); \draw (-3rt) node[shift=({-0.2,0.25})] {-$i$};
  \coordinate (-1rt) at (3*\r/4,5*\rv/4); \fill (-1rt) circle (3pt); \draw (-1rt) node[shift=({0.25,0.25})] {-$j$};
  \draw[thick] (0rb) -- (-1rb) -- (-3rb) -- (-2rb) -- (-1rb) (-1cb) -- (-2cb) -- (0rb);
  \draw[thick] (-3cb) -- (0c) (0r) -- (-2rb) (-1cb) -- (-3cb) -- (-2cb);
  \draw[thick] (0c) -- (-3rt) -- (-2r) -- (-1rt) -- (-3rt)  (0r) -- (-1rt);
  \draw[thick] (0rb) -- ++(0,-2*\ger/3) (-3rb) -- ++(\ger/2,-\ger/2)  (-1cb) -- ++(-\ger/2,-\ger/2) (-2r) -- ++(0,\ger/2) (0r) -- ++(\ger,\ger);
  \draw (\r,3*\rv/2) node {$F_e$};
  \draw (5*\r/8,-\rv/8) node {$F_5$}; \draw (-3*\r/8,\rv/4) node {$F_4$};
  \draw (-\r/16,3*\rv/2) node {$F_2$};
  \coordinate (1F2) at ($(0c)+(0,2*\ger/3)$); \draw (1F2) node[shift=({0,0.2})] {$1$};
  \coordinate (3F4) at ($(0c)+(-\ger,-\ger)$); \draw (3F4) node[shift=({-0.2,-0.2})] {$3$};
  \draw[ultra thin] (3F4) -- (0c)-- (1F2);
  \coordinate (-2) at ($(0c)+(-\ger,\ger)$); \draw (-2) node[shift=({-0.2,0.2})] {-$2$};
  \draw[thick] (0c) -- (-2);
  \coordinate (3r) at (3*\r/4,\rv/2); \fill (3r) circle (3pt); \draw (3r) node[right=.05em] {$j$};
  \coordinate (1r) at (\r/2,\rv/4); \fill (1r) circle (3pt); \draw (1r) node[shift=({-0.25,0})] {$i$};
  \coordinate (2r) at (\r/2,\rv/2); \fill (2r) circle (3pt); \draw (2r) node[shift=({-0.25,0.1})] {$2$};
  \draw[thick] (3r) -- (2r) -- (1r) -- (3r);
  \draw[ultra thin] (3r) -- (-3rt) -- (2r);
  \draw[ultra thin] (1r) -- (0rb) (3r) -- (0r) (2r) -- (-3cb);
\end{scope}
\end{tikzpicture}
\caption{The face $F_5$.}
\label{fig:F5}
\end{figure}
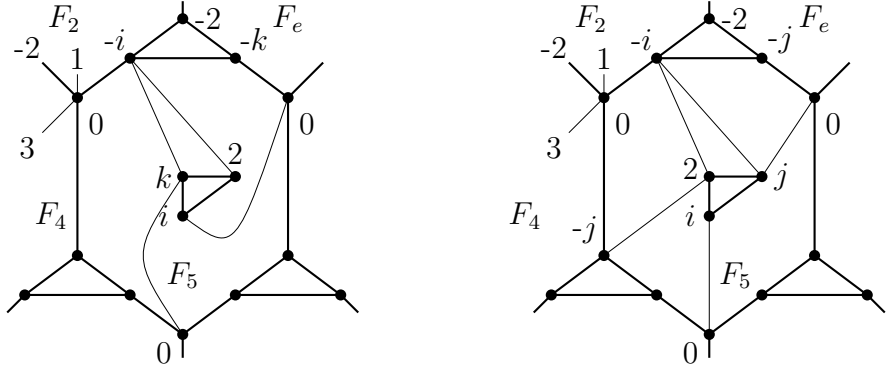

But now we cannot have $j=1$, as otherwise we cannot connect the vertex $2$ to $-1$ in $F_4$. It follows that $j=3$. Then in $F_4$, the edges $(2,-3)$ and $(2,-1)$ are forced, which then forces the placement of the vertices $-2$ on the boundary of $F_5$ and the edges $(1,-2)$ and $(3,-2)$ of $F_5$.

This forces most of the remaining edges and we obtain the drawing as in Figure~\ref{fig:F3F4F5in}. Note that the drawing of $G'$ in $F_5$ is ``fully assembled", while we have one edge $(3,-1)$ missing in $F_4$, and two edges $(1,-3)$ and $(2,-3)$ missing in $F_3$ (connected to the same vertex $-3$ on the boundary of $F_3$.

\begin{figure}[h]
\centering
\begin{tikzpicture}[scale=0.7]
\def \r {4}
\def \rv {3}
\def \ger{\r/6}
  \coordinate (0c) at (0,\rv); \fill (0c) circle (3pt); \draw (0c) node[shift=({-0.25,-0.35})] {0};
  \coordinate (0l) at (-\r,\rv); \fill (0l) circle (3pt); \draw (0l) node[shift=({0.25,-0.35})] {0};
  \coordinate (0r) at (\r,\rv); \fill (0r) circle (3pt); \draw (0r) node[shift=({0.25,-0.35})] {0};
  \coordinate (-1l) at (-\r/2,3*\rv/2); \fill (-1l) circle (3pt); \draw (-1l) node[right=.05em] {-$3$};
  \coordinate (-2r) at (\r/2,3*\rv/2); \fill (-2r) circle (3pt); \draw (-2r) node[right=.05em] {-$2$};
  \coordinate (-3cb) at (0,0); \fill (-3cb) circle (3pt); \draw (-3cb) node[shift=({-0.25,0.35})] {-3};
  \coordinate (-1cb) at (-\r/4,-\rv/4); \fill (-1cb) circle (3pt); \draw (-1cb) node[shift=({0,-0.35})] {-$1$};
  \coordinate (-2cb) at (\r/4,-\rv/4); \fill (-2cb) circle (3pt); \draw (-2cb) node[shift=({-0.25,-0.35})] {-$2$};
  \coordinate (0lb) at (-\r/2,-\rv/2); \fill (0lb) circle (3pt); \draw (0lb) node[shift=({-0.25,-0.25})] {0};
  \coordinate (-1lb) at (-\r,0); \fill (-1lb) circle (3pt); \draw (-1lb) node[shift=({-0.2,0.4})] {-$2$};
  \coordinate (-3lb) at (-3*\r/4,-\rv/4); \fill (-3lb) circle (3pt); \draw (-3lb) node[shift=({-0.25,-0.35})] {-$3$};
  \coordinate (-2lb) at (-5*\r/4,-\rv/4); \fill (-2lb) circle (3pt); \draw (-2lb) node[shift=({0.25,-0.35})] {-$1$};
  \coordinate (0rb) at (\r/2,-\rv/2); \fill (0rb) circle (3pt); \draw (0rb) node[shift=({-0.25,-0.25})] {0};
  \coordinate (-2rb) at (\r,0); \fill (-2rb) circle (3pt); \draw (-2rb) node[shift=({0.25,0.35})] {-$2$};
  \coordinate (-1rb) at (3*\r/4,-\rv/4); \fill (-1rb) circle (3pt); 
  \coordinate (-3rb) at (5*\r/4,-\rv/4); \fill (-3rb) circle (3pt); 
  \coordinate (-3rt) at (\r/4,5*\rv/4); \fill (-3rt) circle (3pt); \draw (-3rt) node[shift=({-0.2,-0.35})] {-$1$};
  \coordinate (-1rt) at (3*\r/4,5*\rv/4); \fill (-1rt) circle (3pt); \draw (-1rt) node[shift=({0.25,0.25})] {-$3$};
  \coordinate (-3lt) at (-\r/4,5*\rv/4); \fill (-3lt) circle (3pt); \draw (-3lt) node[shift=({0.1,-0.35})] {-$2$};
  \coordinate (-2lt) at (-3*\r/4,5*\rv/4); \fill (-2lt) circle (3pt); \draw (-2lt) node[shift=({0,-0.25})] {-$1$};
  \coordinate (0ll) at (-2*\r,\rv); \fill (0ll) circle (3pt); \draw (0ll) node[shift=({-0.25,-0.35})] {0};
  \coordinate (-1ll) at (-3*\r/2,3*\rv/2); \fill (-1ll) circle (3pt); \draw (-1ll) node[left=.06em] {-$1$};
  \coordinate (0llb) at (-3*\r/2,-\rv/2); \fill (0llb) circle (3pt); \draw (0llb) node[shift=({-0.25,-0.25})] {0};
  \coordinate (-1llb) at (-2*\r,0); \fill (-1llb) circle (3pt); 
  \coordinate (-3llb) at (-7*\r/4,-\rv/4); \fill (-3llb) circle (3pt); 
  \coordinate (-2llb) at (-9*\r/4,-\rv/4); \fill (-2llb) circle (3pt); 
  \coordinate (-3llt) at (-5*\r/4,5*\rv/4); \fill (-3llt) circle (3pt); \draw (-3llt) node[shift=({-0.25,-0.25})] {-$3$};
  \coordinate (-2llt) at (-7*\r/4,5*\rv/4); \fill (-2llt) circle (3pt); \draw (-2llt) node[shift=({0.25,-0.25})] {-$2$};
  \draw[thick] (0rb) -- (-1rb) -- (-3rb) -- (-2rb) -- (-1rb) (0lb) -- (-3lb) -- (-1lb) -- (-2lb) -- (-3lb) (0lb) -- (-1cb) -- (-2cb) -- (0rb);
  \draw[thick] (-3cb) -- (0c) -- (-1l)  (0c) -- (-2r) (0l) -- (-1lb) (0r) -- (-2rb) (-1cb) -- (-3cb) -- (-2cb);
  \draw[thick] (-3rt) -- (-2r) -- (-1rt) -- (-3rt)  (0r) -- (-1rt) (-3lt) -- (-1l) -- (-2lt) -- (-3lt) (0l) -- (-2lt);
  \draw[thick]  (-2lb) -- (0llb) -- (-3llb) -- (-2llb) -- (-1llb) -- (0ll) -- (-2llt) -- (-1ll) -- (-3llt) -- (0l) (-2llt) -- (-3llt) (-1llb) -- (-3llb);
  \draw[thick] (0lb) -- ++(0,-2*\ger/3) (0llb) -- ++(0,-2*\ger/3) (0rb) -- ++(0,-2*\ger/3) (-2llb) -- ++(-\ger/2,-\ger/2) (-3rb) -- ++(\ger/2,-\ger/2);
  \draw[thick] (-1l) -- ++(0,\ger/2) (-1ll) -- ++(0,\ger/2) (-2r) -- ++(0,\ger/2) (0ll) -- ++(-\ger,\ger) (0r) -- ++(\ger,\ger);
  \draw (\r,3*\rv/2) node {$F_e$}; \draw (-2*\r,3*\rv/2) node {$F_e$};
  \draw (-9*\r/8,\rv/2) node {$F_3$}; \draw (-\r/2,-\rv/8) node {$F_4$}; \draw (\r/2,-\rv/8) node {$F_5$};
  \draw (\r/8,3*\rv/2) node {$F_2$}; \draw (-\r,3*\rv/2) node {$F_1$};
  \coordinate (3F1) at ($(-2lt)+(\ger/4,2*\ger/3)$); \draw (3F1) node[shift=({0,0.2})] {$3?$};
  \coordinate (1F1) at ($(-2lt)+(-2*\ger/4,\ger/3)$); \draw (1F1) node[shift=({0,0.2})] {$1$};
  \draw[ultra thin] (3F1) -- (-2lt) -- (1F1);
  \coordinate (20F1) at ($(0l)+(0,2*\ger/3)$); \draw (20F1) node[shift=({0,0.2})] {$2$};
  \coordinate (1F2) at ($(0c)+(0,2*\ger/3)$); \draw (1F2) node[shift=({0,0.2})] {$1$};
  \draw[ultra thin] (20F1) -- (0l) (0c)-- (1F2);
  \coordinate (1F31) at ($(-3llt)+(-\ger/4,2*\ger/3)$); \draw (1F31) node[shift=({0,0.2})] {$1$};
  \coordinate (2F31) at ($(-3llt)+(\ger/4,2*\ger/3)$); \draw (2F31) node[shift=({0,0.2})] {$2$};
  \draw[ultra thin] (1F31) -- (-3llt) -- (2F31);
  \coordinate (1F12) at ($(-3lt)+(\ger/2,\ger/3)$); \draw (1F12) node[shift=({0.2,0.2})] {$1$};
  \draw[ultra thin] (1F12) -- (-3lt);
  \coordinate (3r) at (3*\r/4,\rv/2); \fill (3r) circle (3pt); \draw (3r) node[right=.05em] {$3$};
  \coordinate (1r) at (\r/2,\rv/4); \fill (1r) circle (3pt); \draw (1r) node[shift=({-0.25,0})] {$1$};
  \coordinate (2r) at (\r/2,\rv/2); \fill (2r) circle (3pt); \draw (2r) node[shift=({-0.25,0})] {$2$};
  \draw[thick] (3r) -- (2r) -- (1r) -- (3r);
  \draw[ultra thin] (3r) -- (-3rt) -- (2r);
  \coordinate (1l) at (-\r/2,\rv/2); \fill (1l) circle (3pt); \draw (1l) node[right=.05em] {$3$};
  \coordinate (3l) at (-3*\r/4,\rv/2); \fill (3l) circle (3pt); \draw (3l) node[left=.05em] {$1$};
  \coordinate (2l) at (-\r/2,\rv/4); \fill (2l) circle (3pt); \draw (2l) node[shift=({0.25,-0.25})] {$2$};
  \draw[thick] (3l) -- (2l) -- (1l) -- (3l);
  \draw[ultra thin] (-1lb) -- (3l) -- (-3lb)  (0c) -- (1l) -- (-3lt);
  \coordinate (1ll) at (-7*\r/4,\rv/2); \fill (1ll) circle (3pt); \draw (1ll) node[shift=({0.25,0.25})] {$1$};
  \coordinate (3ll) at (-3*\r/2,\rv/2); \fill (3ll) circle (3pt); \draw (3ll) node[above=.05em] {$3$};
  \coordinate (2ll) at (-7*\r/4,\rv/4); \fill (2ll) circle (3pt); \draw (2ll) node[shift=({0.3,0})] {$2$};
  \draw[thick] (3ll) -- (2ll) -- (1ll) -- (3ll);
  \draw[ultra thin] (-2llt) -- (1ll) (0llb) -- (2ll) -- (-2lb) -- (3ll) -- (-1lb) (3ll) -- (0l);
  \draw[ultra thin] (1ll) -- ++(-\r/8,-3*\rv/8) (2ll) -- ++(-\r/24,-\rv/8);
  \draw (-15*\r/8,\rv/8) node {-$3$};
  \draw[ultra thin] (0c) -- (2r) -- (-3cb)--(1r) -- (0rb) (1r)--(-2rb)--(3r)--(0r);
  \draw[ultra thin] (-3lb) -- (2l) -- (-1cb) (2l)--(0lb) (3l)--(0l) (1ll)--(0ll);
  \draw (-\r,-3*\rv/4) node {$F_6$}; \draw (0,-3*\rv/4) node {$F_7$};
\end{tikzpicture}
\caption{$9$-faces $F_3, F_4$ and $F_5$ with a part of $G'$ lying in them.}
\label{fig:F3F4F5in}
\end{figure}
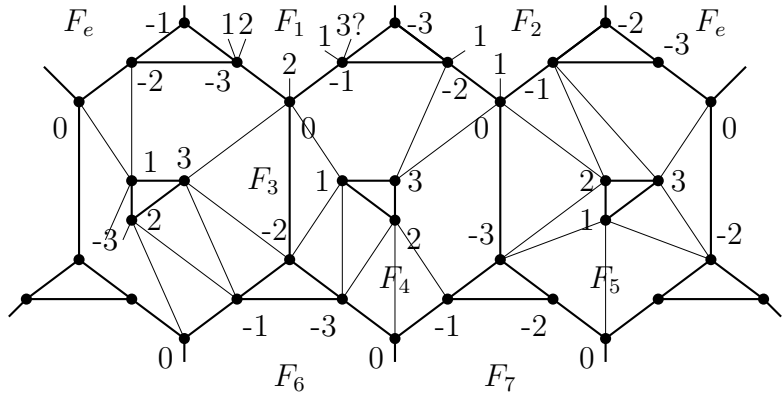

Consider the faces $F_6$ and $F_7$ attached to $F_3, F_4, F_5$ ``from the bottom" as shown in Figure~\ref{fig:F3F4F5in}. They cannot be the same face and cannot be triangular. Moreover, there is an arc from $F_7$ to $F_e$ which only crosses the edges $(0,-2)$ and $(0,-3)$ of $H$, and an arc from $F_6$ to $F_e$ which only crosses the edges $(0,-2)$ and $(0,-1)$ of $H$. We deduce that both $F_6$ and $F_7$ are internal $9$-faces of $H$, by Remark~\ref{rem:twoext}. Moreover, from Remark~\ref{rem:trianglesH} we obtain that $F_6$ contains the $3$-cycles $(-2,1,3)$ and $(-1,2,3)$, and $F_7$, the $3$-cycles $(-3,1,2)$ and $(-2,1,3)$. The faces $F_6$ and $F_7$, together with some forced labels and edges are shown on the left in Figure~\ref{fig:F6F7}.

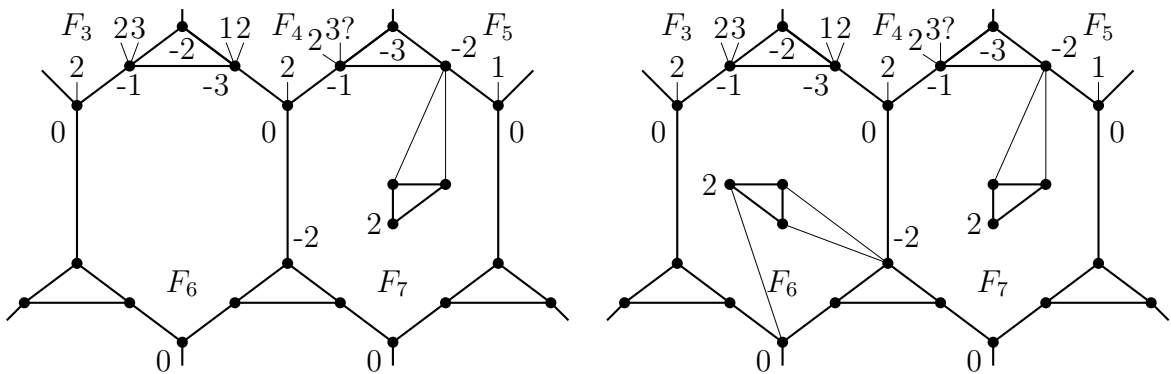
\begin{figure}[h]
\centering
\begin{tikzpicture}[scale=0.65]
\def \r {4}
\def \rv {3}
\def \ger{\r/6}
\begin{scope}
  \coordinate (0c) at (0,\rv); \fill (0c) circle (3pt); \draw (0c) node[shift=({-0.25,-0.35})] {0};
  \coordinate (0l) at (-\r,\rv); \fill (0l) circle (3pt); \draw (0l) node[shift=({-0.25,-0.35})] {0};
  \coordinate (0r) at (\r,\rv); \fill (0r) circle (3pt); \draw (0r) node[shift=({0.25,-0.35})] {0};
  \coordinate (-1l) at (-\r/2,3*\rv/2); \fill (-1l) circle (3pt); \draw (-1l) node[below=.05em] {-$2$};
  \coordinate (-2r) at (\r/2,3*\rv/2); \fill (-2r) circle (3pt); \draw (-2r) node[below=.06em] {-$3$};
  \coordinate (-3cb) at (0,0); \fill (-3cb) circle (3pt); \draw (-3cb) node[shift=({0.25,0.35})] {-2};
  \coordinate (-1cb) at (-\r/4,-\rv/4); \fill (-1cb) circle (3pt); 
  \coordinate (-2cb) at (\r/4,-\rv/4); \fill (-2cb) circle (3pt); 
  \coordinate (0lb) at (-\r/2,-\rv/2); \fill (0lb) circle (3pt); \draw (0lb) node[shift=({-0.25,-0.25})] {0};
  \coordinate (-1lb) at (-\r,0); \fill (-1lb) circle (3pt); 
  \coordinate (-3lb) at (-3*\r/4,-\rv/4); \fill (-3lb) circle (3pt); 
  \coordinate (-2lb) at (-5*\r/4,-\rv/4); \fill (-2lb) circle (3pt); 
  \coordinate (0rb) at (\r/2,-\rv/2); \fill (0rb) circle (3pt); \draw (0rb) node[shift=({-0.25,-0.25})] {0};
  \coordinate (-2rb) at (\r,0); \fill (-2rb) circle (3pt); 
  \coordinate (-1rb) at (3*\r/4,-\rv/4); \fill (-1rb) circle (3pt); 
  \coordinate (-3rb) at (5*\r/4,-\rv/4); \fill (-3rb) circle (3pt); 
  \coordinate (-3rt) at (\r/4,5*\rv/4); \fill (-3rt) circle (3pt); \draw (-3rt) node[shift=({0,-0.25})] {-$1$};
  \coordinate (-1rt) at (3*\r/4,5*\rv/4); \fill (-1rt) circle (3pt); \draw (-1rt) node[shift=({0.25,0.25})] {-$2$};
  \coordinate (-3lt) at (-\r/4,5*\rv/4); \fill (-3lt) circle (3pt); \draw (-3lt) node[shift=({-0.25,-0.25})] {-$3$};
  \coordinate (-2lt) at (-3*\r/4,5*\rv/4); \fill (-2lt) circle (3pt); \draw (-2lt) node[shift=({0,-0.25})] {-$1$};
  \draw[thick] (0rb) -- (-1rb) -- (-3rb) -- (-2rb) -- (-1rb) (0lb) -- (-3lb) -- (-1lb) -- (-2lb) -- (-3lb) (0lb) -- (-1cb) -- (-2cb) -- (0rb);
  \draw[thick] (-3cb) -- (0c) -- (-1l)  (0c) -- (-2r) (0l) -- (-1lb) (0r) -- (-2rb) (-1cb) -- (-3cb) -- (-2cb);
  \draw[thick] (-3rt) -- (-2r) -- (-1rt) -- (-3rt)  (0r) -- (-1rt) (-3lt) -- (-1l) -- (-2lt) -- (-3lt) (0l) -- (-2lt);
  \draw[thick] (0lb) -- ++(0,-2*\ger/3) (0rb) -- ++(0,-2*\ger/3) (-2lb) -- ++(-\ger/2,-\ger/2) (-3rb) -- ++(\ger/2,-\ger/2);
  \draw[thick] (-1l) -- ++(0,\ger/2) (-2r) -- ++(0,\ger/2) (0l) -- ++(-\ger,\ger) (0r) -- ++(\ger,\ger);
  \draw (\r,3*\rv/2) node {$F_5$}; 
  \draw (-\r/2,-\rv/8) node {$F_6$}; \draw (\r/2,-\rv/8) node {$F_7$};
  \draw (0,3*\rv/2) node {$F_4$}; \draw (-\r,3*\rv/2) node {$F_3$};
  \coordinate (3F1) at ($(-2lt)+(-\ger/4,2*\ger/3)$); \draw (3F1) node[shift=({0,0.2})] {$2$};
  \coordinate (1F1) at ($(-2lt)+(\ger/4,2*\ger/3)$); \draw (1F1) node[shift=({0,0.2})] {$3$};
  \draw[ultra thin] (3F1) -- (-2lt) -- (1F1);
  \coordinate (1F4) at ($(-3lt)+(-\ger/4,2*\ger/3)$); \draw (1F4) node[shift=({0,0.2})] {$1$};
  \coordinate (2F4) at ($(-3lt)+(\ger/4,2*\ger/3)$); \draw (2F4) node[shift=({0,0.2})] {$2$};
  \draw[ultra thin] (1F4) -- (-3lt) -- (2F4);
  \coordinate (3F74) at ($(-3rt)+(-\ger/2,\ger/3)$); \draw (3F74) node[shift=({-0.1,0.2})] {$2$};
  \coordinate (1F74) at ($(-3rt)+(0,2*\ger/3)$); \draw (1F74) node[shift=({0,0.2})] {$3?$};
  \draw[ultra thin] (3F74) -- (-3rt) -- (1F74);
  \coordinate (30F1) at ($(0l)+(0,2*\ger/3)$); \draw (30F1) node[shift=({0,0.2})] {$2$};
  \coordinate (1F2) at ($(0c)+(0,2*\ger/3)$); \draw (1F2) node[shift=({0,0.2})] {$2$};
  \coordinate (1F5) at ($(0r)+(0,2*\ger/3)$); \draw (1F5) node[shift=({0,0.2})] {$1$};
  \draw[ultra thin] (30F1) -- (0l) (0c)-- (1F2) (0r)--(1F5);
  \coordinate (3r) at (3*\r/4,\rv/2); \fill (3r) circle (3pt); 
  \coordinate (1r) at (\r/2,\rv/4); \fill (1r) circle (3pt); \draw (1r) node[shift=({-0.25,0})] {$2$};
  \coordinate (2r) at (\r/2,\rv/2); \fill (2r) circle (3pt); 
  \draw[thick] (3r) -- (2r) -- (1r) -- (3r);
  \draw[ultra thin] (3r) -- (-1rt) -- (2r);
\end{scope}
\begin{scope}[shift={(2.85*\r,0)}]
  \coordinate (0c) at (0,\rv); \fill (0c) circle (3pt); \draw (0c) node[shift=({-0.25,-0.35})] {0};
  \coordinate (0l) at (-\r,\rv); \fill (0l) circle (3pt); \draw (0l) node[shift=({-0.25,-0.35})] {0};
  \coordinate (0r) at (\r,\rv); \fill (0r) circle (3pt); \draw (0r) node[shift=({0.25,-0.35})] {0};
  \coordinate (-1l) at (-\r/2,3*\rv/2); \fill (-1l) circle (3pt); \draw (-1l) node[below=.05em] {-$2$};
  \coordinate (-2r) at (\r/2,3*\rv/2); \fill (-2r) circle (3pt); \draw (-2r) node[below=.06em] {-$3$};
  \coordinate (-3cb) at (0,0); \fill (-3cb) circle (3pt); \draw (-3cb) node[shift=({0.25,0.35})] {-2};
  \coordinate (-1cb) at (-\r/4,-\rv/4); \fill (-1cb) circle (3pt); 
  \coordinate (-2cb) at (\r/4,-\rv/4); \fill (-2cb) circle (3pt); 
  \coordinate (0lb) at (-\r/2,-\rv/2); \fill (0lb) circle (3pt); \draw (0lb) node[shift=({-0.25,-0.25})] {0};
  \coordinate (-1lb) at (-\r,0); \fill (-1lb) circle (3pt); 
  \coordinate (-3lb) at (-3*\r/4,-\rv/4); \fill (-3lb) circle (3pt); 
  \coordinate (-2lb) at (-5*\r/4,-\rv/4); \fill (-2lb) circle (3pt); 
  \coordinate (0rb) at (\r/2,-\rv/2); \fill (0rb) circle (3pt); \draw (0rb) node[shift=({-0.25,-0.25})] {0};
  \coordinate (-2rb) at (\r,0); \fill (-2rb) circle (3pt); 
  \coordinate (-1rb) at (3*\r/4,-\rv/4); \fill (-1rb) circle (3pt); 
  \coordinate (-3rb) at (5*\r/4,-\rv/4); \fill (-3rb) circle (3pt); 
  \coordinate (-3rt) at (\r/4,5*\rv/4); \fill (-3rt) circle (3pt); \draw (-3rt) node[shift=({0,-0.25})] {-$1$};
  \coordinate (-1rt) at (3*\r/4,5*\rv/4); \fill (-1rt) circle (3pt); \draw (-1rt) node[shift=({0.25,0.25})] {-$2$};
  \coordinate (-3lt) at (-\r/4,5*\rv/4); \fill (-3lt) circle (3pt); \draw (-3lt) node[shift=({-0.25,-0.25})] {-$3$};
  \coordinate (-2lt) at (-3*\r/4,5*\rv/4); \fill (-2lt) circle (3pt); \draw (-2lt) node[shift=({0,-0.25})] {-$1$};
  \draw[thick] (0rb) -- (-1rb) -- (-3rb) -- (-2rb) -- (-1rb) (0lb) -- (-3lb) -- (-1lb) -- (-2lb) -- (-3lb) (0lb) -- (-1cb) -- (-2cb) -- (0rb);
  \draw[thick] (-3cb) -- (0c) -- (-1l)  (0c) -- (-2r) (0l) -- (-1lb) (0r) -- (-2rb) (-1cb) -- (-3cb) -- (-2cb);
  \draw[thick] (-3rt) -- (-2r) -- (-1rt) -- (-3rt)  (0r) -- (-1rt) (-3lt) -- (-1l) -- (-2lt) -- (-3lt) (0l) -- (-2lt);
  \draw[thick] (0lb) -- ++(0,-2*\ger/3) (0rb) -- ++(0,-2*\ger/3) (-2lb) -- ++(-\ger/2,-\ger/2) (-3rb) -- ++(\ger/2,-\ger/2);
  \draw[thick] (-1l) -- ++(0,\ger/2) (-2r) -- ++(0,\ger/2) (0l) -- ++(-\ger,\ger) (0r) -- ++(\ger,\ger);
  \draw (\r,3*\rv/2) node {$F_5$}; 
  \draw (-\r/2,-\rv/8) node {$F_6$}; \draw (\r/2,-\rv/8) node {$F_7$};
  \draw (0,3*\rv/2) node {$F_4$}; \draw (-\r,3*\rv/2) node {$F_3$};
  \coordinate (3F1) at ($(-2lt)+(-\ger/4,2*\ger/3)$); \draw (3F1) node[shift=({0,0.2})] {$2$};
  \coordinate (1F1) at ($(-2lt)+(\ger/4,2*\ger/3)$); \draw (1F1) node[shift=({0,0.2})] {$3$};
  \draw[ultra thin] (3F1) -- (-2lt) -- (1F1);
  \coordinate (1F4) at ($(-3lt)+(-\ger/4,2*\ger/3)$); \draw (1F4) node[shift=({0,0.2})] {$1$};
  \coordinate (2F4) at ($(-3lt)+(\ger/4,2*\ger/3)$); \draw (2F4) node[shift=({0,0.2})] {$2$};
  \draw[ultra thin] (1F4) -- (-3lt) -- (2F4);
  \coordinate (3F74) at ($(-3rt)+(-\ger/2,\ger/3)$); \draw (3F74) node[shift=({-0.1,0.2})] {$2$};
  \coordinate (1F74) at ($(-3rt)+(0,2*\ger/3)$); \draw (1F74) node[shift=({0,0.2})] {$3?$};
  \draw[ultra thin] (3F74) -- (-3rt) -- (1F74);
  \coordinate (30F1) at ($(0l)+(0,2*\ger/3)$); \draw (30F1) node[shift=({0,0.2})] {$2$};
  \coordinate (1F2) at ($(0c)+(0,2*\ger/3)$); \draw (1F2) node[shift=({0,0.2})] {$2$};
  \coordinate (1F5) at ($(0r)+(0,2*\ger/3)$); \draw (1F5) node[shift=({0,0.2})] {$1$};
  \draw[ultra thin] (30F1) -- (0l) (0c)-- (1F2) (0r)--(1F5);
  \coordinate (3r) at (3*\r/4,\rv/2); \fill (3r) circle (3pt); 
  \coordinate (1r) at (\r/2,\rv/4); \fill (1r) circle (3pt); \draw (1r) node[shift=({-0.25,0})] {$2$};
  \coordinate (2r) at (\r/2,\rv/2); \fill (2r) circle (3pt); 
  \draw[thick] (3r) -- (2r) -- (1r) -- (3r);
  \draw[ultra thin] (3r) -- (-1rt) -- (2r);
  \coordinate (1l) at (-\r/2,\rv/2); \fill (1l) circle (3pt); 
  \coordinate (3l) at (-3*\r/4,\rv/2); \fill (3l) circle (3pt); \draw (3l) node[left=.05em] {$2$};
  \coordinate (2l) at (-\r/2,\rv/4); \fill (2l) circle (3pt); 
  \draw[thick] (3l) -- (2l) -- (1l) -- (3l);
  \draw[ultra thin] (1l) -- (-3cb) -- (2l) (3l)--(0lb);
\end{scope}
\end{tikzpicture}
\caption{$9$-faces $F_6$ and $F_7$.}
\label{fig:F6F7}
\end{figure}

Now the vertex $-2$ shared by $F_6$ and $F_7$ must have its both edges $(-2,1)$ and $(-2,3)$ in $F_6$. Then the vertex $2$ lying in $F_6$ can be connected to a unique vertex $0$, as on the right in Figure~\ref{fig:F6F7}. But now it is easy to see that the part of the semi-cover $G'$ lying in $F_6$ cannot contain the $3$-cycle $(-1,2,3)$, which gives the desired contradiction.
\end{proof}

\acknowledgements
\label{sec:ack}
Dickson Annor was supported by a La Trobe Graduate Research Scholarship. Michael Payne was partially supported by a
DECRA from the Australian Research Council.


\bibliographystyle{abbrvnat}

\bibliography{references1}

@article{hlinueny2001,
 AUTHOR = {Hlin\v{e}n\'{y}, Petr},
     TITLE = {Another two graphs with no planar covers},
   JOURNAL = {J. Graph Theory},
  FJOURNAL = {Journal of Graph Theory},
    VOLUME = {37},
      YEAR = {2001},
    NUMBER = {4},
     PAGES = {227--242},
      OPTissn = {0364-9024,1097-0118},
   MRCLASS = {05C10},
  MRNUMBER = {1834854},
MRREVIEWER = {D.\ S.\ Archdeacon},
       DOI = {10.1002/jgt.1018},
       OPTOPTurl = {https://doi.org/10.1002/jgt.1018},
}

@article{negami1988,
 AUTHOR = {Negami, Seiya},
     TITLE = {The spherical genus and virtually planar graphs},
   JOURNAL = {Discrete Math.},
  FJOURNAL = {Discrete Mathematics},
    VOLUME = {70},
      YEAR = {1988},
    NUMBER = {2},
     PAGES = {159--168},
      OPTissn = {0012-365X,1872-681X},
   MRCLASS = {05C10},
  MRNUMBER = {949775},
MRREVIEWER = {Joan\ Hutchinson},
       DOI = {10.1016/0012-365X(88)90090-8},
       OPTOPTurl = {https://doi.org/10.1016/0012-365X(88)90090-8},
}

@article{archdeacon1990parity,
  AUTHOR = {Archdeacon, Dan and Richter, R. Bruce},
     TITLE = {On the parity of planar covers},
   JOURNAL = {J. Graph Theory},
  FJOURNAL = {Journal of Graph Theory},
    VOLUME = {14},
      YEAR = {1990},
    NUMBER = {2},
     PAGES = {199--204},
      OPTissn = {0364-9024,1097-0118},
   MRCLASS = {05C10},
  MRNUMBER = {1053603},
MRREVIEWER = {Bojan\ Mohar},
       DOI = {10.1002/jgt.3190140208},
       OPTOPTurl = {https://doi.org/10.1002/jgt.3190140208},
}

@book{fellows1985encoding,
  title={Encoding graphs in graphs (graph categories, multiprocessing, complexity)},
  author={Fellows, Michael Ralph},
  year={(1985)},
  publisher={University of California, San Diego}
}

@article{negami2024another,
    AUTHOR = {Negami, Seiya},
     TITLE = {Another approach to planar cover conjecture focusing on
              rotation systems},
   JOURNAL = {J. Math. Soc. Japan},
  FJOURNAL = {Journal of the Mathematical Society of Japan},
    VOLUME = {76},
      YEAR = {2024},
    NUMBER = {3},
     PAGES = {975--996},
      OPTissn = {0025-5645,1881-1167},
   MRCLASS = {05C10 (57M15)},
  MRNUMBER = {4777847},
MRREVIEWER = {Liangxia\ Wan},
       DOI = {10.2969/jmsj/90769076},
       OPTOPTurl = {https://doi.org/10.2969/jmsj/90769076},
}

@article{archdeacon2002two,
 AUTHOR = {Archdeacon, Dan},
     TITLE = {Two graphs without planar covers},
   JOURNAL = {J. Graph Theory},
  FJOURNAL = {Journal of Graph Theory},
    VOLUME = {41},
      YEAR = {2002},
    NUMBER = {4},
     PAGES = {318--326},
      OPTissn = {0364-9024,1097-0118},
   MRCLASS = {05C10},
  MRNUMBER = {1936947},
MRREVIEWER = {Seiya\ Negami},
       DOI = {10.1002/jgt.10075},
       OPTOPTurl = {https://doi.org/10.1002/jgt.10075},
}

@article{hlinvey1998k4,
 AUTHOR = {Hlin\v{e}n\'{y}, Petr},
     TITLE = {{$K_{4,4}-e$} has no finite planar cover},
   JOURNAL = {J. Graph Theory},
  FJOURNAL = {Journal of Graph Theory},
    VOLUME = {27},
      YEAR = {1998},
    NUMBER = {1},
     PAGES = {51--60},
      OPTissn = {0364-9024,1097-0118},
   MRCLASS = {05C70 (57M10)},
  MRNUMBER = {1487786},
MRREVIEWER = {Seiya\ Negami},
       DOI = {10.1002/(SICI)1097-0118(199801)27:1<51::AID-JGT8>3.3.CO;2-S},
       OPTOPTurl =
              {https://doi.org/10.1002/(SICI)1097-0118(199801)27:1<51::AID-JGT8>3.3.CO;2-S},
}

@article{negami1988graphs,
  AUTHOR = {Negami, Seiya},
     TITLE = {Graphs which have no finite planar covering},
   JOURNAL = {Bull. Inst. Math. Acad. Sinica},
  FJOURNAL = {Bulletin of the Institute of Mathematics. Academia Sinica},
    VOLUME = {16},
      YEAR = {1988},
    NUMBER = {4},
     PAGES = {377--384},
      OPTissn = {0304-9825},
   MRCLASS = {05C10 (57M15)},
  MRNUMBER = {1042176},
MRREVIEWER = {D.\ S.\ Archdeacon},
}

@article{hlinveny201020,
  AUTHOR = {Hlin\v{e}n\'{y}, Petr},
     TITLE = {20 years of {N}egami's planar cover conjecture},
   JOURNAL = {Graphs Combin.},
  FJOURNAL = {Graphs and Combinatorics},
    VOLUME = {26},
      YEAR = {2010},
    NUMBER = {4},
     PAGES = {525--536},
      OPTissn = {0911-0119,1435-5914},
   MRCLASS = {05-02 (05C10 05C60 05C70)},
  MRNUMBER = {2669457},
MRREVIEWER = {D.\ S.\ Archdeacon},
       DOI = {10.1007/s00373-010-0934-9},
       OPTOPTurl = {https://doi.org/10.1007/s00373-010-0934-9},
}

@article{brianski2024high,
  title={On high genus extensions of {N}egami's conjecture},
  author={Bria{\'n}ski, Marcin and Davies, James and Tan, Jane},
  journal={arXiv preprint arXiv:2412.04420},
  year={2024}
}

@article {annor25k1222,
    AUTHOR = {Annor, Dickson Y. B. and Nikolayevsky, Yuri and Payne, Michael S.},
     TITLE = {${K}_{1,2,2,2}$ has no $n$-fold planar cover for $n<14$},
   JOURNAL = {Graphs Combin.},
  FJOURNAL = {Graphs and Combinatorics},
    VOLUME = {41},
      YEAR = {2025},
    NUMBER = {6},
     PAGES = {Paper No. 126, 22},
      OPTissn = {0911-0119,1435-5914},
   MRCLASS = {05C55},
  MRNUMBER = {4827919},
MRREVIEWER = {Emerson\ L.\ Monte Carmelo},
       DOI = {10.1007/s00373-025-02983-w},
       OPTurl = {https://doi.org/10.1007/s00373-025-02983-w},
}
\label{sec:biblio}

\end{document}